\documentclass{amsart}

\usepackage{amsmath, amssymb}
\usepackage{mathtools}

\usepackage[english]{babel}
\usepackage{comment}
\usepackage{xcolor}
\usepackage{enumitem}
\setlist[enumerate,1]{
    label=(\roman*), 
    ref=(\roman*), 
    font=\normalfont, 
    topsep=3pt, 
    itemsep=2pt, 
    parsep=0pt, 
    partopsep=0pt, 
    leftmargin=*, 
    widest=iii 
}
\setlist[enumerate,2]{
    label=(\alph*), 
    ref=\theenumi(\alph*), 
    font=\normalfont, 
    topsep=2pt, 
    itemsep=1pt, 
    parsep=0pt, 
    leftmargin=*,
    widest=b
}
\usepackage{graphicx}
\usepackage{booktabs}

\usepackage[protrusion=true, expansion=true]{microtype}

\usepackage[colorlinks=false, linkbordercolor=green, citebordercolor=red]{hyperref}
\usepackage{amsthm}

\usepackage{orcidlink}

\usepackage[noabbrev, capitalize]{cleveref}

\makeatletter
\AtBeginEnvironment{fact}{\crefalias{theorem}{fact}}
\AtBeginEnvironment{lemma}{\crefalias{theorem}{lemma}}
\AtBeginEnvironment{proposition}{\crefalias{theorem}{proposition}}
\AtBeginEnvironment{corollary}{\crefalias{theorem}{corollary}}
\AtBeginEnvironment{definition}{\crefalias{theorem}{definition}}
\AtBeginEnvironment{example}{\crefalias{theorem}{example}}
\AtBeginEnvironment{remark}{\crefalias{theorem}{remark}}
\AtBeginEnvironment{notation}{\crefalias{theorem}{notation}}
\makeatother

\crefformat{equation}{(#2#1#3)}
\Crefformat{equation}{Equation (#2#1#3)}

\crefname{theorem}{Theorem}{Theorems}
\Crefname{theorem}{Theorem}{Theorems}
\crefname{lemma}{Lemma}{Lemmas}
\Crefname{lemma}{Lemma}{Lemmas}
\crefname{proposition}{Proposition}{Propositions}
\Crefname{proposition}{Proposition}{Propositions}
\crefname{corollary}{Corollary}{Corollaries}
\Crefname{corollary}{Corollary}{Corollaries}
\crefname{fact}{Fact}{Facts}
\Crefname{fact}{Fact}{Facts}
\crefname{definition}{Definition}{Definitions}
\Crefname{definition}{Definition}{Definitions}
\crefname{example}{Example}{Examples}
\Crefname{example}{Example}{Examples}
\crefname{remark}{Remark}{Remarks}
\Crefname{remark}{Remark}{Remarks}
\crefname{notation}{Notation}{Notations}
\Crefname{notation}{Notation}{Notations}

\numberwithin{equation}{section}

\theoremstyle{plain}
\newtheorem{theorem}{Theorem}[section]
\newtheorem{lemma}[theorem]{Lemma}
\newtheorem{corollary}[theorem]{Corollary}
\newtheorem{proposition}[theorem]{Proposition}
\newtheorem{fact}[theorem]{Fact}

\theoremstyle{definition}
\newtheorem{definition}[theorem]{Definition}
\newtheorem{example}[theorem]{Example}
\newtheorem{notation}[theorem]{Notation}

\theoremstyle{remark}
\newtheorem{remark}[theorem]{Remark}

\theoremstyle{plain}
\makeatletter
\newtheorem*{rep@theorem}{\pointerep@title}
\newenvironment{reptheorem}[2]{%
  \def\pointerep@title{#1 \ref{#2}}%
  \begin{rep@theorem}}
  {\end{rep@theorem}}
\makeatother

\newcommand{\Z}{\mathbb{Z}}
\newcommand{\R}{\mathbb{R}}
\newcommand{\C}{\mathbb{C}}
\renewcommand{\L}{\mathbb{L}^3}
\newcommand{\B}{\mathbb{B}_{*}}
\renewcommand{\epsilon}{\varepsilon}
\renewcommand{\phi}{\varphi}
\renewcommand{\Re}{\mathrm{Re}}
\renewcommand{\Im}{\mathrm{Im}}
\renewcommand{\hat}{\widehat}
\renewcommand{\tilde}{\widetilde}
\newcommand{\I}{\sqrt{-1}}
\newcommand{\Flux}{\mathrm{Flux}}
\newcommand{\MF}{\mathrm{MF}}
\newcommand{\GMF}{\mathrm{GMF}}
\newcommand{\M}{\mathrm{M}}
\newcommand{\GL}{\mathrm{GL}}
\newcommand{\emb}{\xhookrightarrow[]{\cong}}
\newcommand{\hml}[1]{H_1(#1,\mathbb{Z})}
\newcommand{\grad}{\mathrm{grad} \, }
\newcommand{\Prd}{\mathcal{P}}
\DeclareMathOperator{\Int}{Int}

\usepackage{marginnote}
\newcommand{\checkmarker}{\marginnote{\textcolor{magenta}{\scalebox{0.8}{\bfseries [Check!]}}}}

\usepackage{marginnote}
\usepackage{autonum}

\begin{document}

\title[]
{Approximation and Interpolation Theorems for Maximal Surfaces with Singularities}
\author{Shuki Sano}
\address[Shuki Sano]{
Department of Mathematics, \endgraf
Institute of Science Tokyo, \endgraf
O-okayama, Meguro, Tokyo, 152-8551, Japan
}
\email{sano.s.7465@m.isct.ac.jp}

\date{\today}



\email{}

\begin{abstract}
  In this paper, we prove an approximation and interpolation theorem for maxfaces in the Lorentz--Minkowski $3$-space $\mathbb{L}^3$. Alarc\'on, Forstneri\v{c}, and L\'opez established approximation and interpolation theorems for conformal minimal surfaces using the Enneper--Weierstrass representation formula. We survey their methods and apply them to maxfaces. Furthermore, by incorporating singularity criteria based on the Weierstrass data of maxfaces into the approximation and interpolation theorem, we demonstrate the existence of a maxface with prescribed singularities at specified points, as well as the existence of a maxface whose singular set has a dense image in $\mathbb{L}^3$.
\end{abstract}

\maketitle
\tableofcontents

\section{Introduction}
 Runge's theorem is a well-known approximation theorem in complex analysis. It asserts that for any compact set $K$ in the complex plane $\mathbb{C}$ whose complement has no relatively compact connected components, every function holomorphic on $K$ can be uniformly approximated on $K$ by holomorphic functions on $\mathbb{C}$ \cite{1885.Runge}. A subset $K \subset \mathbb{C}$ is called a \emph{Runge set} if its complement has no relatively compact connected components. H. Behnke and K. Stein \cite{BS.49.Runge.gene} generalized Runge's theorem to a uniform approximation theorem for functions on compact subsets of arbitrary open Riemann surfaces, and it has since been further extended to functions on compact holomorphically convex subsets of Stein manifolds. Extending such Runge-type approximation theorems to maps from Stein manifolds to complex manifolds is one of the subjects of Oka theory (see \cite{F17} for details).
 
 Mergelyan's theorem \cite{51.Merglyan.original} provides a strictly stronger approximation result than Runge's theorem. It guarantees that any continuous function on a compact Runge set $K \subset \mathbb{C}$ that is holomorphic in the interior of $K$ can be uniformly approximated on $K$ by holomorphic functions on $\mathbb{C}$. This theorem was generalized by E. Bishop \cite{Mergelyan.thm} to a uniform approximation theorem for functions on compact subsets of arbitrary open Riemann surfaces. Mergelyan-type approximations for manifold-valued maps have been studied through concepts such as the Mergelyan property.  In particular, for a map from a compact Runge set $K$ in an open Riemann surface to an Oka manifold, Forstneri\v{c} \cite{F19.Mergelyan.for.mfd-valued.map} established a Mergelyan-type approximation. Moreover, such holomorphic approximation theorems often incorporate interpolation conditions, ensuring that the approximating maps agree with given values or jets on specified submanifolds or finite sets. For a comprehensive review of holomorphic approximation theory, we refer to \cite{2020.legacy.hol.approx.}.
 
 It is known that such approximation and interpolation theorems also hold for conformal minimal surfaces in $\mathbb{R}^n$ ($n \geq 3$); this was established by A. Alarc\'on, F. Forstneri\v c, and F. J. L\'opez \cite{AFL21}. At the heart of these approximation and interpolation theorems for conformal minimal surfaces lies the \emph{Enneper--Weierstrass representation formula}. This formula relates a conformal minimal surface $x \colon M \to \mathbb{R}^n$ from an open Riemann surface $M$ to a holomorphic map $f$ from $M$ to the punctured null quadric in $\mathbb{C}^n$. By exploiting the fact that the punctured null quadric is an \emph{Oka manifold}, Alarc\'on et al. achieved the approximation and interpolation of conformal minimal surfaces by applying Runge-type approximation and interpolation theorems to the holomorphic map $f$. However, $f$ is subject to period conditions, and a new holomorphic map obtained by applying these theorems must also satisfy the same conditions. To overcome this issue, Alarc\'on et al. introduced the concept of a  \emph{period dominating spray}, establishing a result that allows for approximation and interpolation while preserving the period conditions {\cite[Lemma 3.3.1]{AFL21}}.
 
 On the other hand, conformal minimal surfaces are not the only surfaces that admit a correspondence with holomorphic maps. Other examples include constant mean curvature $1$ (CMC $1$) surfaces in the hyperbolic $3$-space $\mathbb{H}^3$ and CMC $1$ faces in the de Sitter $3$-space $\mathbb{S}^3_1$; approximation and interpolation theorems for these surfaces were established by A. Alarc\'on and J. Hidalgo \cite{AH.hol.null.in.SL2}. Conformal maximal surfaces in the Lorentz--Minkowski $3$-space $\L$ provide another such example, for which an Enneper--Weierstrass type representation formula was given by O. Kobayashi \cite{83.Kobayashi}. Furthermore, F. J. M. Estudillo and A. Romero \cite{92.ER.generalized.maximal} defined generalized maximal surfaces as maximal surfaces that admit singularities, and M. Umehara and K. Yamada \cite{UY06} defined \emph{maxfaces} as generalized maximal surfaces that are free of branch points. Umehara and Yamada also provided an Enneper--Weierstrass type representation formula for maxfaces, along with criteria for identifying cuspidal edges and swallowtails using their Weierstrass data. Subsequently, further singularity criteria using the Weierstrass data of maxfaces were obtained: a criterion for cuspidal cross caps was given by S. Fujimori, K. Saji, M. Umehara, and K. Yamada \cite{FSUY08}, and criteria for cuspidal butterflies and cuspidal $S_1^-$ singularities were provided by Y. Ogata and K. Teramoto \cite{OT18.duality.cb.cS}.
 
 Following the arguments in the proof of the approximation and interpolation theorem for conformal minimal surfaces in \cite{AFL21}, and further incorporating the singularity criteria for maxfaces, the author has obtained the following approximation and interpolation theorem for maxfaces.

 \begin{reptheorem}{Theorem}{thm:main}
   Assume that $M$ is an open Riemann surface, 
   $\theta$ is a nonvanishing holomorphic $1$-form on $M$, 
   $S \subset M$ is a connected Runge admissible set,
   $\Lambda \subset M$ and $\Sigma \subset \Int(S) \cup \Lambda$ are closed discrete subsets, 
   $V \subset M$ is an open neighborhood of $\Lambda$,
   $f : S \cup V \to \L$ is a map such that 
   $(f|_S, \phi|_S \, \theta)$ is a generalized maxface and $f|_V$ is a maxface, 
   where $\phi = 2 \, \partial f / \theta$.

   Given a positive number $\epsilon$, 
   a map $k \colon \Lambda \to \Z_{>0}$, and 
   a group homomorphism $\mathfrak{p} \colon H_1(M, \Z) \to \R^3$ with 
   $\mathfrak{p}|_{H_1(S, \Z)} = \Flux_f^{\mathcal{C}}$ 
   $($where $\mathcal{C}$ is a suitable homology basis of $S)$,
   there exists a full maxface $\tilde{f} \colon M \to \L$ satisfying the following conditions.
   \begin{enumerate}
     \item
     $\|\tilde{f} - f \|_S \le \epsilon$.
     
     \item
     The difference $\tilde{f} - f$ vanishes to order $k(p)$ at every point $p \in \Lambda$.
     
     \item
     $\Flux_{\tilde{f}} = \mathfrak{p}$ on $H_1(M, \Z)$.

     \item 
     If the maxface $f|_{\Int(S) \cup V}$ is $\mathcal{A}$-equivalent to a cuspidal edge 
     $($resp. swallowtail, cuspidal cross cap, cuspidal butterfly, cuspidal $S_1^-$ singularity$)$ 
     at $p \in \Sigma$, then $\tilde{f}$ is also $\mathcal{A}$-equivalent to a cuspidal edge 
     $($resp. swallowtail, cuspidal cross cap, cuspidal butterfly, cuspidal $S_1^-$ singularity$)$
     at $p$.
   \end{enumerate}
 \end{reptheorem}

One of the purposes of this paper is to survey the methods introduced by Alarc\'on et al. in \cite{AFL21}. Although the above theorem is proved using arguments similar to those for conformal minimal surfaces, the class of target surfaces is different. For this reason, we examine the arguments in \cite{AFL21} step by step to verify that analogous reasoning can be applied to maxfaces. Tracing the arguments in this manner also allows us to clarify exactly where the proof requires different approaches from the conformal minimal surface case. The second purpose of this paper is to present two results as corollaries of the above theorem. The first corollary guarantees the existence of a maxface with prescribed singularities at specified points, and the second corollary guarantees the existence of a maxface whose singular set has a dense image in $\L$. These corollaries are assertions concerning surface singularities—which do not appear on conformal minimal surfaces—and represent results independent of those in \cite{AFL21}.

 This paper is organized as follows. In Section 2, we prepare the facts concerning the approximation and interpolation of holomorphic maps required for the proof of \cref{thm:main}. We also confirm that the $2$-dimensional complex manifold $\B$, which arises as the target space of the holomorphic maps obtained via the Enneper--Weierstrass type representation formula for maxfaces, is an Oka manifold. Subsequently, we recall the definition of maxfaces and their Enneper--Weierstrass type representation formula, and describe the singularity criteria in detail. In Section 3, we first define a \emph{generalized maxface} (\cref{def:GMG.GNLI}) as a maxface defined on an \emph{admissible set} (\cref{def.admissible.set}). Furthermore, we prove that {\cite[Lemma 3.3.1]{AFL21}} also holds for maps taking values in $\B$. In Section 4, we prove \cref{thm:main} following \cite{AFL21}. Afterward, we state the precise assertions and proofs of the two corollaries (\cref{cor:main.,cor:dense.im.sing}) that guarantee the existence of special maxfaces.

 \bigskip

 \emph{Acknowledgments.}
The author would like to express sincere gratitude to Yu Kawakami and Kotaro Yamada for their invaluable guidance. The author is also deeply grateful to Jun Matsumoto for his many helpful suggestions.

\section{Preliminaries}
\subsection{Notation and Conventions}
  Throughout this paper, we use the following notations and conventions.
  By a Riemann surface, we always mean a connected $1$-dimensional complex manifold. 
  For a differentiable function $f$ defined on a Riemann surface $M$,  
  we write $\partial f \coloneqq (\partial f / \partial z) \, dz$, 
  where $z$ is a coordinate. 
  We denote by $\mathcal{O}(M)$ the space of holomorphic functions on $M$. 
  Given a compact set $K \subset M$, $\mathcal{O}(K)$ denotes the space of 
  holomorphic functions on some open neighborhood of $K$, and 
  $\overline{\mathcal{O}}(K)$ denotes the uniform closure of the restrictions 
  $\{f|_K : f \in \mathcal{O}(K)\}$. Moreover, we denote by $\mathcal{A}(K)$ 
  the space of continuous functions on $K$ that are holomorphic on the interior of $K$. 
  Given a complex manifold $X$, we use the analogous notations 
  $\mathcal{O}(M,X)$, $\mathcal{O}(K,X)$, $\overline{\mathcal{O}}(K,X)$, and 
  $\mathcal{A}(K,X)$ for spaces of mappings into $X$.
   
  For the Euclidean spaces $\R^n$ and $\C^n$, 
  we denote the standard Euclidean norm by $\|\cdot\|$, and 
  the absolute value of a complex number $z$ by $|z|$. 
  We denote the open ball centered at $x$ with radius $r > 0$ 
  with respect to the Euclidean norm by $B_r(x)$. 
  In particular, the open disk in $\C$ centered at $z$ with radius $r > 0$ is 
  denoted by $D_r(z)$. 
  
  Furthermore, for a topological space $X$ and its subset $A$, 
  we denote its interior by $\Int(A)$.
  When the inclusion map $A \hookrightarrow X$ induces an isomorphism between 
  their first singular homology groups, we simply write $\hml{A} \emb \hml{X}$.
  When $A$ is compact and $f : A \to \R^n$ is a continuous map, 
  we define the uniform norm of $f$ on $A$ by $\|f\|_A \coloneqq \sup_{p\in A}\|f(p)\|$. 

\subsection{Holomorphic Approximations and Interpolations}
  In this subsection, we review approximations and interpolations 
  by holomorphic functions and maps.
  
  Let $M$ be an open Riemann surface and $K$ be a compact subset of $M$. 
  A \emph{hole} of $K$ in $M$ is a relatively compact connected component of 
  $M \setminus K$. By Runge’s theorem, if $K \subset M$ is a compact subset with 
  no holes in $M$, every function $f \in \mathcal{O}(K)$ is the uniform limit on 
  $K$ of functions in $\mathcal{O}(M)$. 
  A closed subset of $M$ without holes is called a \emph{Runge set}.
  As a theorem that yields a stronger conclusion than Runge's theorem, 
  we have the Bishop--Mergelyan theorem.

  \begin{fact}[The Bishop--Mergelyan theorem, {\cite{Mergelyan.thm}}] \label{fact:Mergelyan}
    If $K$ is a compact Runge set in an open Riemann surface $M$, 
    then every function in $\mathcal{A}(K)$ can be approximated uniformly on $K$ 
    by functions in $\mathcal{O}(M)$.
  \end{fact}

  Weierstrass's theorem for holomorphic functions on the complex plane was 
  extended to open Riemann surfaces by Florack \cite{Weierstrass.thm} in 1948.
  
  \begin{fact}[{\cite{Weierstrass.thm}}] \label{fact:Weierstrass.thm.op.Riem.surf}
    Let $M$ be an open Riemann surface and 
    let $A=\{a_i\}_{i=1}^{\infty}$ be a closed discrete subset of $M$. 
    Given positive integers $k_{i}\in\Z$, there exists a function $f \in \mathcal{O}(M)$
    which vanishes to order $k_i$ at the point $a_i$ for every $i$ and 
    has no other zeros.
  \end{fact}

  The following theorem is known as a refinement of the above theorems.

  \begin{fact}[{\cite[Theorem 1.12.14]{AFL21}}] \label{fact:Weierstrass.Florack}
    Let $M$ be an open Riemann surface, 
    $K$ be a compact Runge subset of $M$, 
    $A=\{a_i\}_{i=1}^{\infty}$ be a closed discrete subset of $M$, 
    $U\subset M$ be an open neighborhood of $A\cup K$ and 
    $f:U\to\C\cup\{\infty\}$ be a meromorphic function 
      whose only zeros and poles are at the points of $A$. 
    Given integers $k_i \in \Z_{>0}$ and a number $\epsilon>0$, 
    there exists a meromorphic function $F:M\to\C\cup\{\infty\}$ such that
      \begin{enumerate}
        \item 
        $\|F-f\|_K<\epsilon$,
     
        \item 
        $F-f$ vanishes to order $k_i$ at the point $a_i\in A$ for every $i\in\Z_{>0}$, and
     
        \item 
        $F$ has no zeros and poles on $M\setminus A$.
      \end{enumerate}
  \end{fact}

  For manifold-valued maps, approximation and interpolation theorems also hold.

  \begin{fact}[{\cite[Theorem 1.13.1]{AFL21}}, 
  {\cite[Theorem 1.4]{F19.Mergelyan.for.mfd-valued.map}}] 
  \label{fact:Mergelyan.for.mfd-valued}
    Let $M$ be a Riemann surface and $X$ be an arbitrary complex manifold. 
    If $K$ is a compact set in $M$ such that $\mathcal{A}(K)=\overline{\mathcal O}(K)$, 
    then $\mathcal{A}(K,X) = \overline{\mathcal{O}}(K,X)$. This holds in particular 
    if $K$ has at most finitely many holes in $M$. Furthermore, the approximating maps 
    can be chosen to agree with the given map at any finite set of points $p_1, \dots, p_m$ in $K$; 
    at the points $p_j\in\Int(K)$ we can interpolate to any given finite order.
  \end{fact}

  Oka manifolds are defined as manifolds on which a Runge-type approximation theorem 
  for holomorphic maps holds.

  \begin{definition}[{\cite[Definition 5.4.1]{F17}}]\label{def:oka.mfd}
    A complex manifold $X$ is an \emph{Oka manifold} if every holomorphic map 
    $K\to X$ from a neighborhood of any compact convex set $K\subset\C^n$ $(n\in\Z_{>0})$ 
    can be approximated uniformly on $K$ by entire holomorphic maps $\C^n\to X$.  
  \end{definition}

  \begin{fact}[{\cite[Theorem 1.13.3]{AFL21}}, {\cite[Theorem 5.4.4]{F17}}] 
  \label{fact:Runge.for.Oka-valued}
    Assume that $M$ is an open Riemann surface, 
    $A\subset M$ is a closed discrete subset of $M$ and 
    $X$ is an Oka manifold endowed with a Riemannian distance function $\mathrm{dist}$. 
    Given a compact Runge subset $K\subset M$, 
    a continuous map $f:M\to X$ which is holomorphic on a neighborhood of $K\cup A$ and 
    a positive integer $s\geq1$, there exist for every $\epsilon>0$
    a neighborhood $U$ of $K\cup A$ and a homotopy $f_t:M\to X$ $(t\in[0,1])$ such that 
    $f_0=f$ and the following conditions hold for every $t\in[0,1]$.
    \begin{enumerate}
      \item 
      The map $f_t$ is holomorphic on $U$.
      
      \item
      We have that $\sup_{p\in K}\mathrm{dist}(f_t(p),f(p))<\epsilon$.
      
      \item
      The map $f_t$ agrees with $f$ to any given finite order $s$ at every point of $A$.
      
      \item
      The map $f_1$ is holomorphic on $M$.
    \end{enumerate}
  \end{fact}
 
  The following fact ensures that the complex manifold 
  \begin{equation} \label{eq:def.B}
    \B \coloneqq 
    \{(z^0, z^1, z^2) \in \C^3 \setminus \{0\} : -(z^0)^2 + (z^1)^2 + (z^2)^2 = 0\}
  \end{equation}
  is an Oka manifold.
  
  \begin{fact}[{\cite[Example 1.13.8]{AFL21}}]\label{fact:B.is.Oka}
    If $P(z_1, \dots, z_n)$ is a homogeneous quadratic polynomial on $\C^n$ for some $n\ge2$
    such that $A=\{P=0\}$ is smooth away from the origin, 
    then the manifold $X=A\setminus\{0\}$ is an Oka manifold.
  \end{fact}

  A function $\rho : M \to \R$ on a Riemann surface is said to be \emph{strongly subharmonic} 
  if it satisfies $\partial^2 \rho / \partial z \partial \bar{z} > 0$ for any complex coordinate $z$.
  If a strongly subharmonic function $\rho$ is an exhaustion function 
  (i.e., $\rho^{-1}((-\infty, c])$)is compact for every $c \in \R$), 
  then for any $c \in \R$, the set $\{\rho \le c\}$ is a compact Runge set 
  provided that it is not empty.
  The following properties are also known for a compact Runge subset of an open Riemann surface:

  \begin{fact}[{\cite[Proposition 1.12.5]{AFL21}}] \label{fact:str.subham.Morse.exh.fucn}
    If $K$ is a compact Runge subset in an open Riemann surface $M$ and 
    $U \subset M$ is an open set containing $K$, then there exists a strongly subharmonic
    Morse exhaustion function $\rho$ on $M$ such that 
    $K \subset \{\rho < 0\} \subset \{\rho \le 0\} \subset U$.
    In particular, for any compact Runge subset $K$ and an open neighborhood $U$ of $K$, 
    there exists a compact Runge subset containing $K$ in its interior and contained in $U$.
  \end{fact}

\subsection{Maxfaces}
  The Lorentz--Minkowski $3$-space $\L$ is the $3$-dimensional affine space $\R^3$ 
  with the inner product  
  \begin{equation}\label{Lmetric}
    \langle\ ,\ \rangle\coloneq-(dx^0)^2+(dx^1)^2+(dx^2)^2,
  \end{equation}
  where $(x^0,x^1,x^2)$ is the canonical coordinate of $\R^3$. 
  An immersion $f:M\to \L$ from an oriented $2$-dimensional manifold $M$ into $\L$ is 
  called \emph{space-like} if the induced metric 
  $ds^2\coloneqq f^*\langle\ ,\ \rangle=\langle df,df\rangle$ is positive definite on $M$. 

  A smooth map $\nu:M\to \L$ is called a (\emph{Lorentzian}) \emph{unit normal vector field} of 
  a space-like immersion $f:M\to \L$ if 
  $\langle df_p(v),\nu(p)\rangle=0$ and $\langle \nu(p),\nu(p)\rangle=-1$ hold for all 
  $p\in M$ and $v\in T_pM$.
  A space-like immersion $f:M\to \L$ is called \emph{maximal} 
  if the mean curvature function vanishes identically.

  Umehara and Yamada \cite{UY06} defined maxfaces as maximal surfaces with singularities 
  other than branch points.

  \begin{definition}[{\cite[Remark 1.2]{FKKRUY10}}, {\cite{UY06}}] \label{def:maxface}
    A smooth map $f:M\to \L$ is called a \emph{maxface} if there exists 
    an open dense subset $W_f \subset M$ such that $f|_{W_f}$ is a maximal immersion, 
    and $df$ has no zeros on $M$.
    A point where $ds^2 = \langle df,df\rangle$ degenerates is 
    called a \emph{singular point} of $f$.
  \end{definition}

  \begin{notation} \label{notation:MF}
    We denote by $\MF(M)$ the space of maxfaces on $M$.
  \end{notation}
  
  Maxfaces are surfaces that admit a Weierstrass-type representation formula 
  similar to that of conformal minimal immersions.

  \begin{fact}
  [Enneper--Weierstrass type representation for maxface, {\cite[Theorem 2.6]{UY06}}]
  \label{fact:Weierstrass.rep.maxface}
    Let $M$ be a Riemann surface and
    let $f:M\to \L$ be a smooth map.
    Then, the following are equivalent.
    \begin{enumerate}
      \item 
      The map $f$ is a maxface which is a conformal maximal immersion 
      on the open dense subset $W_f \subset M$. 
    
      \item 
      There exists a meromorphic function $g$ and 
      a holomorphic $1$-form $\omega$ on $M$ such that
      \begin{itemize}
        \item 
        $(1+|g|^2)^2\omega\overline{\omega}$ is a Riemannian metric on $M$,

        \item 
        $1-|g|^2$ does not vanish identically,

        \item 
        $\Re\displaystyle\int_C\left(-2g,1+g^2,\I(1-g^2)\right)\omega=0$
        for all closed curves $C$ in $M$, and

        \item 
        $f(p) = f(p_0) + \Re\displaystyle\int_{p_0}^p \left(-2g,1+g^2,\I(1-g^2)\right) \omega$,
        where $p_0 \in M$ is a base point.
      \end{itemize}
    \end{enumerate}  
  When $f$ satisfies either $(\mathrm{i})$ or $(\mathrm{ii})$, 
  the set of singular points of $f$ is given by $\{p\in M:|g(p)|=1\}$.
  \end{fact}

  As follows from \cref{fact:Weierstrass.rep.maxface}, each component function 
  of a maxface is harmonic, which implies that there is no maxface defined on 
  a compact Riemann surface. Therefore, in what follows, we assume that the domain 
  of a maxface is an open Riemann surface $M$. 
  As shown by Gunning and Narasimhan \cite{GN67.nonvani.hol.1-form}, 
  any open Riemann surface $M$ admits a holomorphic immersion into $\C$. 
  By differentiating this holomorphic immersion, we obtain 
  a nonvanishing holomorphic $1$-form $\theta$ on $M$.

  Let $\B$ be the complex manifold defined by \eqref{eq:def.B}. 
  By \cref{fact:Weierstrass.rep.maxface}, a smooth map $f : M \to \L$ 
  on an open Riemann surface $M$ is a maxface which is a conformal maximal immersion 
  on the open dense subset $W_f \subset M$ 
  if and only if there exists a holomorphic map 
  $\phi = (\phi^0, \phi^1, \phi^2) \in \mathcal{O}(M, \B)$ satisfying the following conditions:
  \begin{itemize}
    \item 
    $-|\phi^0|^2 + |\phi^1|^2 + |\phi^2|^2$ does not vanish identically on $M$,

    \item 
    $\Re \displaystyle\int_C \phi \, \theta = 0$ for all closed curves $C$ in $M$, and

    \item 
    $f(p) = f(p_0) + \Re \displaystyle \int_{p_0}^p \phi \, \theta$, 
    where $p_0 \in M$ is a base point.
  \end{itemize}
  The set of singular points of $f$ is given by 
  $\{-|\phi^0|^2 + |\phi^1|^2 + |\phi^2|^2 = 0\}$.
  For a maxface $f$, we define its \emph{flux} as the homomorphism 
  $\Flux_f \colon \hml{M} \to \R^3$ given by
  \begin{equation}
    \Flux_f([C]) \coloneqq \Im \int_C \phi \, \theta = \Im \int_C 2 \, \partial f.
  \end{equation}

  A holomorphic immersion $F = (F^0,F^1,F^2):M\to\C^3$ is called a 
  \emph{Lorentzian null immersion} if
  \begin{equation}
    -(dF^0)^2 + (dF^1)^2 + (dF^2)^2 = 0
  \end{equation}
  holds on $M$. It is immediate to see that $F$ is a Lorentzian null immersion 
  if and only if $\phi = dF / \theta$ is a holomorphic map into $\B$. 
  Furthermore, if $-|dF^0|^2 + |dF^1|^2 + |dF^2|^2$ does not vanish identically on $M$, 
  then $f = \Re F$ is a maxface.

  Now, let us briefly recall the singularities of maxfaces.
  For $j=1, 2$, let $f_j : U_j \to \R^3$ be smooth maps defined on open neighborhoods 
  $U_j$ of $p_j \in \R^2$. We say that $f_2$ is 
  \emph{$\mathcal{A}$-equivalent} (or \emph{left-right equivalent}) to $f_1$ at $p_2$ 
  if there exist a diffeomorphism 
  $\psi : U_1 \to U_2$ (replacing $U_1$ and $U_2$ with sufficiently small ones 
  if necessary) with $\psi(p_1) = p_2$, and a diffeomorphism 
  $\Psi : \Omega_2 \to \Omega_1$ between sufficiently small open neighborhoods 
  $\Omega_j$ of $f_j(p_j)$ in $\R^3$ ($j=1,2$) such that $\Psi \circ f_2 \circ \psi = f_1$.
  \begin{itemize}
    \item 
    A \emph{cuspidal edge} is the map $(u,v) \mapsto (u^2, u^3, v)$ defined on 
    a neighborhood of the origin in $\R^2$.
   
    \item 
    A \emph{swallowtail} is the map $(u,v) \mapsto (u, 4v^3 + 2uv, 3v^4+uv^2)$ defined on 
    a neighborhood of the origin in $\R^2$.

    \item 
    A \emph{cuspidal cross cap} is the map $(u,v) \mapsto (u, uv^3, v^2)$ defined on 
    a neighborhood of the origin in $\R^2$.

    \item 
    A \emph{cuspidal butterfly} is the map $(u, v) \mapsto (u, 4v^5 + uv^2, 5v^4 + 2uv)$ 
    defined on a neighborhood of the origin in $\R^2$.

    \item 
    A \emph{cuspidal $S_1^-$ singularity} is the map $(u, v) \mapsto (u, v^2, v^3(u^2 - v^2))$ 
    defined on a neighborhood of the origin in $\R^2$.
  \end{itemize}
  
  Criteria for identifying the singularities of a maxface 
  in terms of its Weierstrass data are known.

  \begin{fact}
  [{\cite[Theorem 2.4]{FSUY08}}, {\cite[Theorem 30]{OT18.duality.cb.cS}}, 
  {\cite[Theorem 3.1]{UY06}}]\label{fact:hanntei}
    Let $U$ be a domain in the complex plane $(\C,z)$, and 
    let $f \colon U \to \L$ be a maxface with Weierstrass data $(g, \omega = \hat{\omega} \, dz)$.
    We define functions $\alpha, \beta$, and $\gamma$ on $U$ by
    $\alpha \coloneqq g' / (g^2\hat{\omega})$, $\beta \coloneqq g \, \alpha' / g'$,
    $\gamma \coloneqq g \, \beta' / g'$, where $'=d/dz$. Then,
    \begin{enumerate}
      \item 
      $f$ is $\mathcal{A}$-equivalent to a cuspidal edge at $p\in U$ if and only if
      $\Re\left(\alpha\right)\ne 0$ and $\Im\left(\alpha\right)\ne0$ hold at $p$.
      
      \item
      $f$ is $\mathcal{A}$-equivalent to a swallowtail at $p\in U$ if and only if
      $\Re\left(\alpha\right) \ne 0$, $\Im\left(\alpha\right) = 0$ and 
      $\Re\left(\beta\right)\ne0$ hold at $p$.
      
      \item
      $f$ is $\mathcal{A}$-equivalent to a cuspidal cross cap at $p\in U$ if and only if
      $\Re\left(\alpha\right) = 0$, $\Im\left(\alpha\right) \ne 0$ and
      $\Im\left(\beta\right) \ne 0$ hold at $p$.

      \item 
      $f$ is $\mathcal{A}$-equivalent to a cuspidal butterfly at $p\in U$ if and only if
      $\Im\left(\alpha\right) = 0$, $\Re\left(\beta\right) = 0$ and
      $\Im\left(\gamma\right) \ne 0$ hold at $p$.

      \item 
      $f$ is $\mathcal{A}$-equivalent to a cuspidal $S_1^-$ singularity at $p\in U$ 
      if and only if
      $\Im\left(\alpha\right) \ne 0$, $\Im\left(\beta\right) = 0$ and
      $\Re\left(\gamma\right) \ne 0$ hold at $p$.
    \end{enumerate}
  \end{fact}

  \begin{example}[Lorentzian Enneper surface 
  {\cite[Example 2.6]{FSUY08}}, {\cite[Example 5.2]{UY06}}] \label{ex:L.Enneper}
    The \emph{Lorentzian Enneper surface} is the maxface $f : \C \to \L$ 
    with Weierstrass data $(z, dz)$.
    The set of singularities is 
    $\{|z| = 1\}$. Moreover, $f$ is $\mathcal{A}$-equivalent to 
    a cuspidal edge at every point of 
    $\{|z| = 1\} \setminus 
    \{\pm 1, \pm \I,
    e^{\I \frac{\pi}{4}}, e^{\I \frac{3}{4}\pi},
    e^{\I \frac{5}{4}\pi}, e^{\I \frac{7}{4}\pi}\}$; 
    it is $\mathcal{A}$-equivalent to a swallowtail at every point of 
    $\{\pm 1, \pm \I\}$, and 
    to a cuspidal cross cap at every point of 
    $\{e^{\I \frac{\pi}{4}}, e^{\I \frac{3}{4}\pi}, 
     e^{\I \frac{5}{4}\pi}, e^{\I \frac{7}{4}\pi}\}$.
  \end{example}

  \begin{example}[{\cite{OT18.duality.cb.cS}}] \label{ex:cB.cS}
    The maxface $f_1 : \R \times (-\pi, \pi) \to \L$ with the Weierstrass data
    $(-e^z + 1/\sqrt{2}, \, (-\I e^{-z}/ 2) \, dz)$ is $\mathcal{A}$-equivalent to
    a cuspidal butterfly at each point of $\{\log(1/\sqrt{2}) \pm \I \pi/2\}$,
    while the maxface $f_2 : \R \times (-\pi, \pi) \to \L$ with the Weierstrass data
    $(-e^z + 1/\sqrt{2}, \, (e^{-z} / 2) \, dz)$ is $\mathcal{A}$-equivalent 
    to a cuspidal $S_1^-$ singularity at each point of
    $\{\log(1/\sqrt{2}) \pm \I \pi/2\}$.
  \end{example}

  In \cite{AFL21}, a map $f : M \to \mathbb{C}^n$ from a connected manifold $M$ 
  to $\C^n$ is said to be \emph{full} if the $\C$-linear span of its image satisfies
  $\mathrm{Span}_{\C} \, f(M) = \mathbb{C}^n$. 
  Moreover, a conformal minimal immersion $x$ (resp. a null immersion $z$) is referred to as 
  a full conformal minimal immersion (resp. a full null immersion) 
  if the map $2 \, \partial x / \theta$ (resp. $\partial z / \theta$) is full. 
  Following this approach, we introduce the notions of a full maxface and a 
  full Lorentzian null immersion.
   
  \begin{definition} \label{def:full}
    Let $M$ be an open Riemann surface and 
    let $\theta$ be a nonvanishing holomorphic $1$-form on $M$.
    We call a maxface $f : M \to \L$  
    (resp. a Lorentzian null immersion $F : M \to \C^3$) \emph{full} if the map
    $2 \, \partial f / \theta$ (resp. $\partial F / \theta$) $: M \to \B \subset \C^3$ 
    is full.
  \end{definition}

  Just as the image of a full conformal minimal immersion does not lie in any plane, 
  the image of a full maxface also does not lie in a plane. 
  
  Below, we provide lemmas regarding full maps.
  
  \begin{lemma}\label{lem.rest.full}
    Let $M$ be a connected Riemann surface. 
    If $F:M\to\C^n$ is a full holomorphic map, and if $A\subset M$ has an accumulation point, 
    then $\mathrm{Span}F_{\C}(A)=\C^n$. 
  \end{lemma}

  \begin{proof}
    Assume $\mathrm{Span}F_{\C}(A)\subsetneq\C^n$. 
    We can take an orthonormal basis $\{v_1, \dots, v_k\}$ of $\mathrm{Span}_{\C}F(A)$. 
    Let $f : M\to\C^n$ be a holomorphic map defined by 
    \begin{equation}
      f(p) \coloneqq \sum_{i=1}^k\left(^t\overline{v_i}\cdot F(p)\right)v_i.
    \end{equation}
    It holds that $f(M) \subset \mathrm{Span}_{\C} F(A)$, i.e., 
    $f$ is not a full map. On the other hand, the identity theorem implies 
    $f = F$ because of $f|_A = F|_A$. 
    This contradicts the fullness of $F$. Now, we know that $\mathrm{Span}F_{\C}(A)=\C^n$.
  \end{proof}

  \begin{lemma}\label{lem:full.then.not.identically.zero}
    Let $M$ be an open Riemann surface and 
    let $\phi=(\phi^0,\phi^1,\phi^2) : M \to \B$ be a holomorphic map. 
    If $\phi$ is full, then $-|\phi^0|^2+|\phi^1|^2+|\phi^2|^2$ does not vanish identically.
  \end{lemma}
  
  \begin{proof}
    Assume that there exists a full holomorphic map $\phi=(\phi^0,\phi^1,\phi^2)$ such that 
    $-|\phi^0|^2+|\phi^1|^2+|\phi^2|^2$ vanishes identically. Then, it holds that
    \begin{equation}
      |(\phi^1)^2+(\phi^2)^2|=|(\phi^0)^2|=|\phi^1|^2+|\phi^2|^2
    \end{equation}
    on $M$. We can show that
    \begin{equation}
      \Re(\phi^1) \, \Im(\phi^2) - \Im(\phi^1) \, \Re(\phi^2)=0
    \end{equation}
    by a direct calculation. Thus $\phi^1/\phi^2$ is a real-valued meromorphic function, 
    i.e., it is a constant function on $M$. Let $c\in\R$ satisfy 
    $\phi^1 = c \, \phi^2$ on $M$. This implies
    \begin{equation}
      (\phi^0)^2=(\phi^1)^2+(\phi^2)^2=(c^2+1)(\phi^2)^2, 
    \end{equation}
    from which, we see that there is a real number $c'$ such that 
    $\phi^0=c'\phi^2$ holds on $M$. Consequently, $\phi$ must be a function of 
    the form $\phi=(c',c,1)\phi^2$, however this contradicts its fullness.
  \end{proof}

\section{Maxfaces on Admissible Sets and Preparations for Approximations and Interpolations}

  In this section, we state the definition of an \emph{admissible set} introduced by 
  Alarcón, Forstnerič, and López \cite{AFL21}. 
  We then introduce the notions of \emph{generalized maxfaces} and 
  \emph{generalized Lorentzian null immersions}, 
  by analogy with their definitions of conformal minimal immersions and 
  null curves on admissible sets 
  (i.e., generalized conformal minimal immersions and generalized null curves).
  We then state and prove the propositions required for 
  the approximation and interpolation theorems for maxfaces, following \cite{AFL21}.
  
  \begin{definition}[{\cite[Definition 1.12.9]{AFL21}}]\label{def.admissible.set}
    Let $M$ be a Riemann surface. An \emph{admissible set} in $M$ is 
    a compact set of the form $S=K\cup E$, where $K$ is a finite union of 
    pairwise disjoint compact domains with piecewise $C^1$ boundaries in $M$ 
    and $E=S\setminus\Int(K)$ is a union of finitely many pairwise disjoint smooth Jordan arcs 
    and closed Jordan curves meeting $K$ only at their endpoints (if at all) and 
    such that their intersections with the boundary $\partial K$ of $K$ are transverse.
  \end{definition}

  \begin{remark} \label{rem:reg.nbd.of.admissible.set}
    Let $S$ be an admissible set of a Riemann surface $M$. 
    For any Riemannian distance function $d$ on $M$ and 
    any sufficiently small $r > 0$, the open neighborhood 
    \begin{equation}
      S_r \coloneqq \{p \in M : d(p, S) < r\}
    \end{equation}
    of $S$ satisfies $\hml{S} \emb \hml{S_r}$ (see \cite[p. 69]{AFL21}). 
    This set $S_r$ is called a \emph{regular neighborhood} of $S$.
  \end{remark}

  \begin{fact}[{\cite[Lemma 1.12.10]{AFL21}}] \label{fact:homology.basis.of.addmissible}
    A connected admissible set $S=K\cup E$ has finitely generated first homology group $\hml{S}$. 
    Furthermore, there is a homology basis $\mathcal{C} = \{C_1, \dots, C_l\}$ consisting of 
    closed piecewise smooth Jordan curves in $S$ such that $C=\bigcup_{i=1}^lC_i$ is connected and 
    Runge in any regular neighborhood $S_r$ of S, where $r$ is a sufficiently small positive number. 
    Moreover, every curve $C_i\in\mathcal{C}$ contains a nontrivial arc $I_i$ 
    disjoint from $\bigcup_{j\neq i}C_j$.
  \end{fact}

  Since the argument in the proof of \cref{fact:homology.basis.of.addmissible} is 
  necessary to prove the approximation and interpolation theorems for maxfaces, 
  we reproduce the proof given in \cite{AFL21} here.

  \begin{proof}[Proof of \cref{fact:homology.basis.of.addmissible}]
    If $K=\varnothing$, then the conclusion is trivial. 
    We now assume that $K\ne\varnothing$. 
    Let $K_1, \dots, K_m$ and $E_1, \dots, E_n$ be connected components of 
    $K$ and $E$ respectively. The boundary 
    $\partial K_i=\bigcup_{j=1}^{m_i}\Gamma_{i,j}$ consists of finitely many 
    Jordan curves for some $m_i\ge 1$.  We choose an interior point $q_i\in\Int(K_i)$ 
    from each component of $K$. There exists a basis $\mathcal{C}_1$ of $H_1(K_i,\Z)$ 
    consisting of finitely many Jordan curves in $\Int(K_i)$ passing through $q_i$. 
    We take two points $a_{i,j}, b_{i,j}\in\Gamma_{i,j}$ with $b_{i,j}\notin E$ and 
    connect $a_{i,j}$ to $q_i$ by an arc $A_{i,j}\subset\Int(K_i)\cup\{a_{i,j}\}$. 
    We can choose $A_{i,j}$ not to intersect each other and not to intersect elements 
    of $\mathcal{C}_1$. 

    Let the end points $e$ and $e'$ of $E_k$ be contained in $\partial K_i$. 
    Suppose that $e\in\Gamma_{i,j_1}$ and $e'\in\Gamma_{i,j_2}$. 
    We construct piecewise smooth Jordan curve whose base point is $e$ 
    in the following steps:
    \begin{enumerate}
      \item
      connect $e$ and $e'$ by $E_k$,
     
      \item 
      connect $e'$ and $a_{i,j_2}$ by a part of $\Gamma_{i,j_2}$ 
      that does not contain $b_{i,j_2}$,
     
      \item 
      connect $a_{i,j_2}$ and $q_i$ by $A_{i,j_2}$,
     
      \item
      connect $q_i$ and $a_{i,j_1}$ by $A_{i,j_1}$, and
     
      \item
      connect $a_{i,j_1}$ and $e$ by a part of $\Gamma_{i,j_1}$ 
      that does not contain $b_{i,j_1}$.
    \end{enumerate}
    We denote the collection of Jordan curves obtained in this way by $\mathcal{C}_2$. 
    Let us construct an admissible set $S_2$ in the following way:
    \begin{enumerate}
      \item[(vi)] 
      let $S_1$ be an admissible set obtained by removing bridges from $S$, 
      where a bridge is a connected components of $E$ 
      whose endpoints belong to different connected component of $K$, and
    
      \item[(vii)] 
      let $S_2$ be a connected admissible set obtained by attaching to $S_1$ 
      a collection of bridges such that removing any one of them disconnects $S_2$. 
    \end{enumerate}
    We note that $H_1(S_1,\Z)\emb H_1(S_2,\Z)$ and their homology basis is
    $\mathcal{C}_1\cup\mathcal{C}_2$. For every bridge $E_k$ 
    that is not contained in $S_2$, there exist pairwise distinct bridges 
    $E_k = E_{k_1}, \dots,E_{k_s}$ and 
    connected components $K_{i_1}, \dots,K_{i_s}$ of $K$ such that 
    $E_{k_1}$ connects $K_{i_1}$ to $K_{i_2}$, 
    $E_{k_2}$ connects $K_{i_2}$ to $K_{i_3}$, etc., 
    until the cycle closes with the last bridge $E_{k_s}$ connecting $K_{i_s}$ to $K_{i_1}$. 
    We obtain a new closed curve in $S$ by connecting the endpoint of each $E_{k_{j}}$ 
    to the initial point of the next bridge $E_{k_{j+1}}$ in $K_{i_j}$, 
    where $E_{k_{s+1}}=E_{k_1}$. The connecting curves in $K_{i_j}$ are obtained by
    replacing $e$ and $e'$ in steps (i) to (v) with the end point of $E_{k_j}$ and 
    the initial point of $E_{k_{j+1}}$, respectively. We denote the collection of 
    these closed curves by $\mathcal{C}_3$.

    Then, $\mathcal{C}\coloneq\mathcal{C}_1\cup\mathcal{C}_2\cup\mathcal{C}_3$ 
    is a homology basis of $H_1(S,\Z)$ and it is clear that
    every $C_i\in\mathcal{C}$ contains a nontrivial arc which is disjoint from 
    all other curves in $\mathcal{C}$. Let $C$ be the union of all curves in $\mathcal{C}$. 
    Any point in $K_i\setminus C$ can be connected to $b_{i,j}\in\Gamma_{i,j}$ 
    by an arc in $K_i\setminus C$ for some $j\in\{1, \dots, m_i\}$. 
    Hence, we have $\hml{C} \emb \hml{S} \emb \hml{S_r}$. 
    Thus, $C$ is Runge in $S_{r}$. We can make $C$ connected by modifying each closed curve 
    to pass through $q_1\in\Int(K_1)$. Indeed, every curve in $\mathcal{C}$ 
    passes through $q_i\in\Int(K_i)$, so it suffices to connect $q_i$ to $q_1$ 
    in the same way to construct an element of $\mathcal{C}_3$.
  \end{proof}
  
  \begin{definition}\label{def:GMG.GNLI}
    Let $S=K\cup E$ be an admissible set in a Riemann surface $M$ and
    let $\theta$ be a nonvanishing holomorphic $1$-form on an open neighborhood of $S$
    \begin{enumerate}
      \item
      A pair $(f,\phi \, \theta)$ is called a \emph{generalized maxface $S\to \L$} 
      if a $C^1$ map $f : S \to \L$ and 
      $\phi = (\phi^0, \phi^1, \phi^2) \in \mathcal A(S,\B)$ satisfy 
      the following conditions.
      \begin{enumerate}
         \item 
        $-|\phi^0|^2 + |\phi^1|^2 + |\phi^2|^2$ does not vanish identically on $\Int(S)$. 
     
        \item 
        $\Re\displaystyle\int_C\phi \, \theta = 0$ for all closed curve $C \subset S$.
      
        \item 
        For fixed $p_0 \in S$, it holds that 
        $f(p) = f(p_0) + \Re\displaystyle\int_{p_0}^p \phi \, \theta$
        on the connected component of $S$ containing $p_0$.
      \end{enumerate}
      For a homology basis $\mathcal{C}$ of $S$, 
      we define $\Flux_f^{\mathcal{C}} : \hml{S} \to \R^3$
      \begin{equation}
        \Flux_f^{\mathcal{C}}([C]) 
        \coloneqq \Im \int_C 2 \, \partial f
        = \Im \int_C \phi \, \theta \qquad (C \in \mathcal{C}).
      \end{equation}
      The map $\Flux_f^{\mathcal{C}}$ is called the \emph{flux} of $f$ along $\mathcal{C}$. 
      We denote the space of generalized maxface $S\to\L$ by $\GMF(S)$.
  
      \item
      A pair $(F,\phi \, \theta)$ is called a \emph{generalized Lorentzian null immersion $S\to \C^3$} 
      if $F \in \mathcal{A} (S,\C^3)$ and $\phi \in \mathcal{A}(S,\B)$ satisfy the following conditions.
      \begin{enumerate}
        \item 
        $\displaystyle\int_C\phi \, \theta = 0$ for all closed curve $C \subset S$.
      
        \item 
        For fixed $p_0 \in S$, it holds that 
        $F(p) = F(p_0) + \displaystyle\int_{p_0}^p \phi \, \theta$
        on the connected component of $S$ containing $p_0$.
      \end{enumerate}
    \end{enumerate}
  \end{definition}

  \begin{definition}[{\cite[p. 136]{AFL21}}]
    Let $M$ be an open Riemann surface, 
    let $\theta$ be a nonvanishing holomorphic $1$-form on $M$ and 
    let $\mathcal{C} = \{C_1, \dots, C_l\}$ be a collection of oriented Jordan curves and arcs in $M$. 
    We define the \emph{period map} 
    $\Prd = (\Prd_1, \dots, \Prd_l):C^0(\,\bigcup_{i=1}^{l}C_i\;,\C^3)\to(\C^3)^l$ 
    associated to $\mathcal{C}$ by
    \begin{equation}
      \Prd(\phi) =
      (\Prd_1(\phi), \dots, \Prd_l(\phi))
      \coloneqq\left(\int_{C_1}\phi \, \theta, \dots, \int_{C_l} \phi \, \theta\right),
    \end{equation}
    where $C^0(\,\bigcup_{i=1}^{l}C_i\;,\C^3)$ is the space of continuous maps from 
    $\bigcup_{i=1}^{l}C_i$ to $\C^3$.
  \end{definition}

  The following proposition is an analogue of {\cite[Lemma 3.3.1]{AFL21}}. 
  As in the case of conformal minimal immersions, this result plays a central role 
  in the approximation theorem for maxfaces.

  \begin{proposition}\label{prop:per.prerv.full.approx.}
    Assume that $M$ is an open Riemann surface, 
    $S=K\cup E$ is a Runge admissible set in $M$, 
    $\phi\in\mathcal A(S,\B)$, and 
    $\mathcal{C} = \{C_1, \dots, C_l\}$ is a collection of smooth oriented Jordan curves and 
    arcs in $S$ such that every $C_i \in \mathcal{C}$ contains a nontrivial arc $I_i$ 
    disjoint from $\bigcup_{j\neq i}C_j$ and $C=\bigcup_{i=1}^lC_i$ is Runge in $M$. 
    Let $\Prd$ be the period map associated to $\mathcal{C}$. 
    Then, given a finite set $A \subset S$ and $s \in \Z_{>0}$, 
    there exists a sequence $\{\phi_{n}\}_{n}\subset \mathcal O(M,\B)$ satisfying 
    the following conditions:
    \begin{enumerate}
      \item
      $\phi_{n}$ is a full map for all $n\in\Z_{>0}$,
      
      \item
      $\|\phi_{n}-\phi\|_S\rightarrow 0\quad(n\rightarrow\infty)$,
      
      \item
      $\Prd(\phi_{n})=\Prd(\phi)$ holds for all $n\in\Z_{>0}$,
      
      \item
      $\phi_{n}$ agrees with $\phi$ on $A$, and
      
      \item
      $\phi_{n}$ agrees with $\phi$ to order $s$ 
      at every point of  $A\cap\Int(S)$.
    \end{enumerate}
  \end{proposition}

  Below, we prove \cref{prop:per.prerv.full.approx.} following the method of \cite{AFL21}.

  \begin{lemma} \label{lem:tan.sp.of.B}
    Let
    \begin{equation}
    \eta=
    \begin{pmatrix}
     -1 & 0 & 0 \\
     0  & 1 & 0 \\
     0  & 0 & 1
    \end{pmatrix}.
    \end{equation}
    Then we have the following:
    \begin{enumerate}
     \item 
      if $z\in\B$ then the tangent space $T_z\B\subset\C^3$ of $\B$ is given by
      \begin{equation}
        T_z\B=\{(v^0,v^1,v^2)\in\C^3:-z^0v^0+z^1v^1+z^2v^2=0\},
      \end{equation}
   
    \item 
      $v\in T_z\B$ if and only if $\,^tv\eta z=0$,
      
    \item
     $z\in\B$ if and only if $\,^tz\eta z=0$, and
     
    \item
     $T_z\B=T_w\B$ if and only if there exists $\alpha\in\C$ such that $w=\alpha z$.
    \end{enumerate}
  \end{lemma}

  \begin{proof}
    We prove only (iv). It is clear that if $w=\alpha z$ for some 
    $\alpha\in\C$ then $T_z\B=T_w\B$ holds because of (i). 
    Assume $T_z\B=T_w\B$. We define linear functions $L_z,\ L_w:\C^3\to\C$ by 
    \begin{equation}
      L_z(u)\coloneqq\,^tz\eta u, \quad L_w(u)\coloneqq\,^tw\eta u.
    \end{equation}
    Then, there exists $u_0\in\C^3$ which does not lie in the $2$-dimensional subspace
    $K\coloneqq\ker(L_z)=T_z\B=T_w\B=\ker(L_w)$. Fix $u\in\C^3$. Since 
    \begin{equation}
      u-\frac{L_z(u)}{L_z(u_0)}u_0\in K,
    \end{equation}
    we get
    \begin{equation}
      L_w(u)
      = L_w\left(u-\frac{L_z(u)}{L_z(u_0)}u_0+\frac{L_z(u)}{L_z(u_0)}u_0\right)
      = \frac{L_w(u_0)}{L_z(u_0)}L_z(u).
    \end{equation}
    Therefore, $\alpha\coloneqq L_w(u_0)/L_z(u_0)$ satisfies $L_w=\alpha L_z$. 
    This means $w=\alpha z$.
  \end{proof}

  \begin{lemma} \label{lem:vec.fld.and.flow.maxface}
    Let $z\in\B$ and $v\in T_z\B\subset\C^3$. 
    Then there exist holomorphic maps $V:\C^3\to\C^3$ and 
    $\psi : \C \times \C^3 \to \C^3$ satisfying the following:
    \begin{enumerate}
      \item
      $V(z)=v$,
     
      \item
      $V(w)\in T_w\B$ for all $w\in\B$,
     
      \item 
      $\partial_t \, \psi(t,w)=V\left(\psi(t,w)\right)$ for all $(t,w)\in\C\times\C^3$,
     
      \item 
      $\psi(0,\,\cdot\,)=\mathrm{id}$, and
     
      \item 
      if $w\in\B$ then $\psi(t,w)\in\B$ for all $t\in\C$.
    \end{enumerate}
  \end{lemma}

  \begin{proof}
    Since $z=(z^0,z^1,z^2)\ne0$, we have $z^i\ne0$ for some $i\in\{0,1,2\}$. 
    Let $A\in \M(3,\C)$ be a complex matrix of the form
    \begin{equation}
      A=
     \begin{dcases}
       \begin{pmatrix}
         0 & v^1/z^0 & v^2/z^0 \\
         v^1/z^0 & 0 & 0 \\
         v^2/z^0 & 0 & 0
      \end{pmatrix}
       & \text{if $i = 0$}, \\
     \begin{pmatrix}
       0 & v^0/z^1 & 0 \\
       v^0/z^1 & 0 & -v^2/z^1 \\
       0 & v^2/z^1 & 0
     \end{pmatrix}
      & \text{if $i = 1$}, \\
     \begin{pmatrix}
       0 & 0 & v^0/z^2 \\
       0 & 0 & v^1/z^2 \\
       v^0/z^2 & -v^1/z^2 & 0
     \end{pmatrix}
     & \text{if $i = 2$}, \\
    \end{dcases}
   \end{equation}
   where $v=(v^0,v^1,v^2)\in T_z\B$ is a given tangent vector. We define 
   \begin{equation}
     V(w)\coloneqq Aw, \qquad \psi(t,w)\coloneqq\exp(tA)w.
   \end{equation}
   \cref{lem:tan.sp.of.B} yields (i) and (ii). 
   And it is clear that (iii) and (iv) hold. 
   Noting ${}^t A = -\eta A \eta$, we get (v) by 
   the following calculation:
   \begin{align}
    {}^t \psi(t,w) \, \eta \, \psi(t,w)
     &= \, {}^tw\exp(t \; {}^tA) \, \eta \, \exp(tA)w \\
     &= \, {}^tw \, \eta \, \exp(-tA) \, \exp(tA)w 
     = \, {}^tw \, w
     =0.
   \end{align}
  \end{proof}

  With these preparations, we obtain the following assertion, 
  which is analogous to {\cite[Lemma 3.2.1]{AFL21}}.
  
  \begin{proposition} \label{prop:existence.Prd.dominating.spray}
    Let $S=K\cup E$ be an admissible set in an open Riemann surface $M$, 
    let $\phi \in \mathcal{A}(S,\B)$, 
    let $A \subset S$ be a finite set, let $s$ be a positive integer, and
    let $\mathcal{C}=\{C_1, \dots,C_l\}$ be a collection of piecewise smooth oriented Jordan curves 
    and arcs in $S$ such that $C\coloneqq\bigcup_{j=1}^lC_j$ is Runge 
    in an open neighborhood $\tilde{S}$ of $S$. 
    Assume that every curve $C_i \in \mathcal{C}$ contains a nontrivial arc $I_i$ 
    disjoint from $\bigcup_{j\neq i}C_j$ such that 
    $\phi(I_i)$ is not contained in any complex line passing through the origin. 
    Then there exists $\Phi_{\phi}\in\mathcal A (S\times\C^{3l},\B)$ such that 
    $\Phi_{\phi}(\;\cdot\;,0) = {\phi}$ and the derivative of 
    $\:\C^{3l}\ni t\mapsto\Prd(\Phi_{\phi}(\;\cdot\;,t))\in\C^{3l}$ at $t=0$
    \begin{equation}
      \frac{\partial}{\partial t}\bigg |_{t=0}\Prd(\Phi_{\phi}(\;\cdot\;,t))
    \end{equation} 
    determines an isomorphism $\C^{3l}\to\C^{3l}$, and for each $t\in\C^{3l}$,
    the map $\Phi_{\phi}(\;\cdot\;,t) \colon S \to \B$ agrees with $\phi$ on $A$, 
    and to order $s$ on $A \cap \Int(S)$.
  \end{proposition}

  Following \cite{AFL21}, we call $\Phi_{\phi} \in \mathcal{A} (S \times \C^{3l},\B)$ in 
  \cref{prop:existence.Prd.dominating.spray} a \emph{period dominating spray} 
  of maps $S \to \B$ with the core $\Phi_{\phi}(\;\cdot\;,0)= {\phi}$.

  \begin{proof}[Proof of \cref{prop:existence.Prd.dominating.spray}]
    Fix $i\in\{1, \dots, l\}$. By the assumption, there exist points $p$ and $p'$ of $I_i \setminus A$ 
    such that $\phi(p)\neq \alpha \, \phi(p')$ for all $\alpha \in \C$. 
    For these two points $p$ and $p'$, $T_{\phi(p)}\B \ne T_{\phi(p')}\B$ holds by 
    \cref{lem:tan.sp.of.B}. We choose a basis $\{v_{i1}, v_{i2}, v_{i3}\}$ of $\C^3$ from 
    $T_{\phi(p)}\B \cup T_{\phi(p')}\B \subset \C^3$ and points $p_{i1}, p_{i2}, p_{i3} \in \{p, p'\}$ 
    with $v_{ik}\in T_{\phi(p_{ik})}\B$ $(k=1, 2, 3)$. 
    \cref{lem:vec.fld.and.flow.maxface} yields the existence of holomorphic maps 
    $V_{ik}:\C^3 \to \C^3$ and $\psi_{ik}:\C\times\C^3 \to \C^3$ satisfying 
    the following conditions:
    \begin{itemize} 
      \item
      $V_{ik}\left(\phi(p_{ik})\right)=v_{ik}$,
      
      \item
      $V_{ik}(w)\in T_w \B$ for all $w\in\B$,
      
      \item 
      $\partial_t\psi_{ik}(t,w)=V_{ik}\left(\psi_{ik}(t,w)\right)$ for all $(t,w)\in\C\times\C^3$,
      
      \item 
      $\psi_{ik}(0,\,\cdot\,)=\mathrm{id}$, and
      
      \item
      if $w \in \B$ then $\psi_{ik}(t,w) \in \B$ for all $t \in \C$.
    \end{itemize}
    Let $A_{ik} \in \M(3,\C)$ be a matrix with $V_{ik}(w)=A_{ik}w$ and 
    $\psi_{ik}(t,w)=\exp(tA_{ik})w$. 
    For each positive integer $n$, 
    we can choose a continuous function $h_{ik}^{n} : C \to \R_{\geq0}$ with 
    $\mathrm{supp}(h_{ik}^{n})\subset I_i$, $\mathrm{supp}(h_{ik}^n)$ is connected,
    $\bigcap_{n = 1}^{\infty} \mathrm{supp}(h_{ik}^{n}) = \{p_{ik}\}$, and
    $h_{ik}^{n}\equiv 1$ on a neighborhood of $p_{ik}$.
    Moreover, by \cref{fact:Weierstrass.thm.op.Riem.surf}, 
    there exists $g\in\mathcal O(\tilde{S})$ 
    that vanishes to order $s$ at every point of $A$.
    Now, we define $\Phi_{n}:C\times (\C^3)^l\to\B$ by
    \begin{align}
      \Phi_{n}(p,t_{11}, t_{12}, t_{13}, \dots, & t_{l1}, t_{l2},t_{l3})
      \coloneqq 
      \psi^{g(p) \, h_{11}^{n}(p) \, t_{11}}_{11}\circ\cdots\circ
       \psi^{g(p) \, h_{l3}^{n}(p)\, t_{l3}}_{l3}
      \left(\phi(p)\right)\\
      &=\exp\left(g(p) \, h_{11}^{n}(p) \, t_{11} \, A_{11}\right) \cdots
        \exp\left(g(p) \, h_{l3}^{n}(p) \, t_{l3} \, A_{l3}\right)\,\phi(p),
    \end{align}
    where $\psi^{\ t}_{ik}(p)=\psi_{ik}(t,p)$. Since 
    \begin{align}
      \frac{\partial}{\partial t_{ik}}\bigg |_{t_{ik}=0}\Phi_{n}(p,0, \dots, t_{ik}, \dots, 0)
      &= \frac{\partial}{\partial t_{ik}}\bigg |_{t_{ik}=0}
             \exp\left(g(p) \, h_{ik}^{n}(p) \, t_{ik} \, A_{ik}\right) \, \phi(p) \\
      &= g(p) \, h_{ik}^{n}(p) \, V_{ik}\left(\phi(p)\right),
    \end{align}
    we obtain
    \begin{align} \label{eq.6.1}
      \frac{\partial}{\partial t_{ik}}\bigg |_{t_{ik}=0}
             \Prd_{j}\left(\Phi_{n}(\;\cdot\;,0, \dots, t_{ik}, \dots, 0)\right)
      =\int_{C_j} g \, h_{ik}^{n}\cdot\left(V_{ik}\circ \phi\right)\theta
    \end{align}
    for each $j\in\{1, \dots, l\}$. If $i\ne j$, then
    \begin{align} \label{eq.6.2}
      \int_{C_j} g \, h_{ik}^{n}\cdot\left(V_{ik}\circ \phi\right)\theta=0.
    \end{align}
    Let us consider the case $i=j$. Suppose that $\gamma_i:[0,1]\to C_i$ is 
    a parameterization of $C_i$, and real numbers $a_{ik}^{n}$, $b_{ik}^{n}$, 
    and $\tau_{ik}$ satisfy
    $[a_{ik}^{n},\;b_{ik}^{n}]=\gamma_i^{-1}(\mathrm{supp}(h_{ik}^{n}))$ and 
    $p_{ik}=\gamma_i(\tau_{ik})$. Note that $\tau_{ik}\in[a_{ik}^{n},\;b_{ik}^{n}]$, and
    $\lim_{n\rightarrow\infty}(b_{ik}^{n}-a_{ik}^{n})=0$.
    If we choose $F_{ik}^{n}:[a_{ik}^{n},\;b_{ik}^{n}]\to\C^3$ and 
    $\tilde{\theta}:[a_{ik}^{n},\;b_{ik}^{n}]\to\C$ satisfying 
    $F_{ik}^{n}dt=\gamma_i^*(gh_{ik}^{n}\cdot(V_{ik}\circ \phi) \, \theta)$ and
    $\tilde{\theta} \, dt=\gamma_i^*\theta$, then we get
    \begin{equation}
      \begin{aligned} \label{eq.6.3}
        &\left\|\frac{1}{b_{ik}^{n}-a_{ik}^{n}}\left(\int_{C_i}g \, h_{ik}^{n} \cdot 
        \left(V_{ik}\circ \phi\right)\theta\right)-\tilde{\theta}(\tau_{ik})\;v_{ik}\right\| \\
        =&\left\|\frac{1}{b_{ik}^{n}-a_{ik}^{n}}
           \left(\int_{a_{ik}^{n}}^{b_{ik}^{n}}F_{ik}^{n}dt\right) - 
             F_{ik}^{n}(\tau_{ik})\right\| 
        =\left\|\int_{a_{ik}^{n}}^{b_{ik}^{n}}
            \frac{F_{ik}^{n}-F_{ik}^{n}(\tau_{ik})}{b_{ik}^{n}-a_{ik}^{n}}dt\right\| \\
        \leq&\left(b_{ik}^{n}-a_{ik}^{n}\right)
              \left\|\frac{F_{ik}^{n}-F_{ik}^{n}(\tau_{ik})}{b_{ik}^{n}-a_{ik}^{n}}
                    \right\|_{[a_{ik}^{n}, b_{ik}^{n}]}
        =\left\|F_{ik}^{n}-F_{ik}^{n}(\tau_{ik})\right\|_{[a_{ik}^{n},b_{ik}^{n}]}
        \rightarrow 0.
      \end{aligned}
    \end{equation}
    Let $\{J_{i}^{n}\}_{n=1}^{\infty}$ and $\{H_i^{n}\}_{n=1}^{\infty}$ be sequences 
    in $\M(3,\C)$ defined by
    \begin{gather}
      J_{i}^{n} \coloneqq
      \left(
      \frac{\partial}{\partial t_{ik}}\bigg |_{t_{ik}=0}
         \Prd_i\left(\Phi_{n}(\;\cdot\;,0, \dots,t_{ik}, \dots,0)\right)
      \right)_{k = 1, 2, 3},\\
      H_i^{n}\coloneqq
      \begin{pmatrix}
        (b_{i1}^{n}-a_{i1}^{n})^{-1}&     0                       & 0 \\
                    0             &(b_{i2}^{n}-a_{i2}^{n})^{-1} & 0  \\
            0                     &            0          & (b_{i3}^{n}-a_{i3}^{n})^{-1}
      \end{pmatrix}.
    \end{gather}
    From $\eqref{eq.6.1}$, $\eqref{eq.6.2}$ and $\eqref{eq.6.3}$, we know that the sequence 
    $\{J_{i}^{n}\,H_i^{n}\}_{n=1}^{\infty}$ converge to the following invertible matrix:
    \begin{equation}
      (\tilde{\theta}(\tau_{i1})v_{i1}, 
         \tilde{\theta}(\tau_{i2})v_{i2},
           \tilde{\theta}(\tau_{i3})v_{i3})
      = (v_{i1}, v_{i2}, v_{i3})
      \begin{pmatrix}
        \tilde{\theta}(\tau_{i1})& 0 & 0 \\
         0 & \tilde{\theta}(\tau_{i2}) & 0 \\
         0 & 0 & \tilde{\theta}(\tau_{i3}) 
      \end{pmatrix}.
    \end{equation}
    Hence, there is an integer $n_0$ such that $J_i^{n_0}\,H_i^{n_0}\in \GL(3,\C)$, i.e., 
    $J_i^{n_0}\in \GL(3,\C)$ for all $i\in \{1, \dots, l\}$. Thus $\Phi_{n_0}$ satisfies 
    \begin{equation}
      \frac{\partial}{\partial t}\bigg |_{t=0}\Prd(\Phi_{n_0}(\;\cdot\;,t))=
      \left(
        \begin{array}{c|c|c}
          J_1^{n_0}&\cdots&0 \\ \hline
          \vdots&\ddots&\vdots \\ \hline
          0&\cdots&J_l^{n_0}
        \end{array}
      \right) \in \GL(3l,\C).
    \end{equation}
    Since $C$ is Runge in an open neighborhood $\tilde{S}$ of $S$, 
    the Bishop--Mergelyan theorem (\cref{fact:Mergelyan}) implies that 
    there exists a sequence 
    $\{\tilde{h}_{ik}^{\nu}\}_{\nu=1}^{\infty}\subset\mathcal O (\tilde{S})$ such that
    $\|\tilde{h}_{ik}^{\nu}-h_{ik}^{n_0}\|_{C}\rightarrow0$ $(\nu\rightarrow\infty)$.
    We define $\tilde{\Phi}_{\nu}\in\mathcal{A}(S\times\C^{3l},\B)$ by
    \begin{equation}\label{e.q.period.dominating}
      \begin{split}
        \tilde{\Phi}_{\nu}(p,t_{11}, t_{12},t_{13}, \dots, & t_{l1}, t_{l2}, t_{l3})
        \coloneqq
         \psi^{g \, \tilde{h}_{11}^{\nu}(p)t_{11}}_{11} \circ \cdots \circ
          \psi^{g \, \tilde{h}_{l3}^{\nu}(p)t_{l3}}_{l3} \left (\phi(p) \right)\\
        &=\exp\left(g(p) \, \tilde{h}_{11}^{\nu}(p)t_{11}A_{11}\right) \cdots 
           \exp\left(g(p) \, \tilde{h}_{l3}^{\nu}(p)t_{l3}A_{l3}\right)\,\phi(p).
      \end{split}
    \end{equation}
    There is a constant $L>0$ such that 
    \begin{align}
      &\left\|\frac{\partial}{\partial t_{ik}}\bigg |_{t_{ik}=0}
        \Prd_{j}\left(\tilde{\Phi}_{\nu}(\;\cdot\;,0, ...,t_{ik}, ..., 0)\right)
          - \frac{\partial}{\partial t_{ik}}\bigg |_{t_{ik}=0}
             \Prd_{j}\left(\Phi_{n_0}(\;\cdot\;,0, ...,t_{ik}, ..., 0)\right)\right\|\\
      =&\left\|\int_{C_j} g \, \left(\tilde{h}_{ik}^{\nu}-h_{ik}^{n_0}\right)
         \left(V_{ik}\circ \phi\right)\theta\right\|
      \le L \left\|\tilde{h}_{ik}^{\nu}-h_{ik}^{n_0}\right\|_C
    \end{align}
    holds for all $i,j\in\{1, \dots, l\}$ and $k\in\{1, 2, 3\}$. Thus,
    $\{\partial_t|_{t=0}\Prd(\tilde{\Phi}_{\nu}(\,\cdot\,,t))\}_{\nu=1}^{\infty}$ converges to
    $\partial_t|_{t=0}\Prd(\Phi_{n_0}(\,\cdot\,,t))\in \GL(3l,\C)$. 
    Therefore, we can take an integer $\nu_0$ so that 
    $\partial_t|_{t=0}\Prd(\tilde{\Phi}_{\nu_0}(\,\cdot\,,t))$ is an invertible matrix. 
    Then, $\Phi_{\phi} \coloneqq \tilde{\Phi}_{\nu_0}$ is the desired map.
    Indeed, it is clear that $\Phi_{\phi} (p,t) = \phi(p)$ holds 
    for every $p\in A$ and $t\in\C^{3l}$. 
    Set $B_{ik}(p,t) =\tilde{h}_{ik}^{\nu_0}(p) \, t_{ik} \, A_{ik}$. 
    Then, there is a holomorphic map $F:\tilde{S}\times\C^{3l} \to \M(3,\C)$ such that
    \begin{equation}
      \Phi_{\phi}(p,t)
      = \left( I_3 + \sum_{i=1}^l\sum_{k=1}^3 g(p) \, B_{ik}(p,t) + 
              g(p)\,F(p,t)\right)\,\phi(p).
    \end{equation}
    We get
    $\Phi_{\phi}(p,t) - \phi(p) = g(p) \, (\sum_{i,k}B_{ik}(p,t) + F(p,t)) \, \phi(p)$, 
    and hence $\Phi_{\phi}(\,\cdot\,,t)$ agrees with $\phi$ to order $s$ 
    at every point of $\Int(S) \cap A$.
  \end{proof}

  \begin{remark}\label{rem.6.9}
    Let $\Phi_\phi$ be a period dominating spray with core $\phi$ defined by 
    \eqref{e.q.period.dominating}. Then, there exists $\epsilon > 0$ such that 
    for every $\tilde{\phi} \in \mathcal{A}(S,\B)$ with 
    $\|\phi-\tilde{\phi}\|_S < \epsilon$ and $\tilde{\phi}|_A = \phi|_A$, 
    and such that $\tilde{\phi} - \phi$ vanishes to order $s$ on $A \cap \Int(S)$, 
    replacing $\phi$ with $\tilde{\phi}$ in \eqref{e.q.period.dominating} yields 
    a period dominating spray $\Phi_{\tilde{\phi}}$ with core $\tilde \phi$.
    Indeed, let $A_{ik}\in\M(3,\C)$ $(1\le i\le l,\ k = 1, 2, 3)$ be matrices 
    in \eqref{e.q.period.dominating} and 
    let $L_{ik}\coloneqq\|g\tilde{h}_{ik}^{\nu_0}\|_S\,\|A_{ik}\|$. Then,
    \begin{align}
      &\left\|\frac{\partial}{\partial t_{ik}}\bigg |_{t_{ik}=0}
        \Prd_{j}\left(\Phi_{\phi}(\;\cdot\;,0, ...,t_{ik}, ..., 0)\right)
         -\frac{\partial}{\partial t_{ik}}\bigg |_{t_{ik}=0}
          \Prd_{j}\left(\Phi_{\tilde \phi}(\;\cdot\;,0, ...,t_{ik}, ..., 0)\right)\right\|\\
     =&\left\|
        \int_{C_j}
          g \, \tilde{h}_{ik}^{\nu_0} \, A_{ik}\cdot\left(\phi-\tilde \phi\right)\theta
         \right\|
     \le L_{ik}\cdot\left\|\phi-\tilde \phi\right\|_S\int_{C_j}|\theta|.
   \end{align}
   Since $\partial_t|_{t=0}\Prd(\Phi_{\phi}(\cdot,t))\in\GL(3l,\C)$, 
   there is a positive number $\epsilon>0$ such that 
   if $\|\phi-\tilde \phi\|_S<\epsilon$ then 
   $\partial_t|_{t=0}\Prd(\Phi_{\tilde \phi}(\cdot,t))\in\GL(3l,\C)$.
   Furthermore, $\Phi_{\tilde{\phi}}(\cdot,t)$ agrees with $\tilde{\phi}$ on $A$,
   and to order $s$ on $A \cap \Int(S)$.
  \end{remark}

  The proposition below is a slightly weaker statement, 
  which is needed to prove \cref{prop:per.prerv.full.approx.}.

  \begin{proposition}\label{prop:Mergelyan.for.B-valued.map}
    Let $M$ be an open Riemann surface, 
    let $S\subset M$ be a Runge admissible set of $M$,
    $A\subset S$ be a finite subset, and 
    let $s \in \Z_{>0}$.
    Then, given $\phi \in \mathcal{A}(S,\B)$, there exists a sequence 
    $\{\phi_{n}\}_{n=1}^{\infty}$ in $\mathcal{O}(M,\B)$ satisfying the following:
    \begin{enumerate}
      \item 
      $\|\phi_{n}-\phi\|_S\rightarrow0\quad(n\rightarrow\infty)$,
    
      \item
      $\phi_{n}$ agrees with $\phi$ at every point of $A$, and
    
      \item 
      $\phi_{n}$ agrees with $\phi$ to order $s$ 
      at every point of $A \cap \Int(S)$.
    \end{enumerate}
  \end{proposition}

  To show \cref{prop:Mergelyan.for.B-valued.map}, 
  we prepare lemmas.

  \begin{lemma} \label{lem:how.to.extend}
    \begin{enumerate}
      \item 
      Let $M$ be a smooth manifold, 
      let $K$ be a compact subset of $M$, 
      let $N$ be a smooth submanifold of $\R^m$, and 
      let $\phi:K\to N$ be a continuous map. 
      If there exists a sequence $\{\phi_{n}\}_{n=1}^{\infty}$ 
      of continuous maps from $M$ to $N$ such that $\|\phi_{n}-\phi\|_K\rightarrow0$, 
      then $\phi$ admits a continuous extension to $M$.

      \item 
      Let $M$ be an open Riemann surface, 
      let $K$ be a compact subset of $M$, and 
      let $\{\phi_{n}\}_{n=1}^{\infty}$ be a sequence in $\mathcal{O}(K, \B)$. 
      If there exists a continuous map $\phi \colon M \to \B$ such that 
      $\|\phi_{n} - \phi\|_K \to 0$, 
      then for all sufficiently large $n$, $\phi_{n}$ admits a continuous extension to $M$ 
      such that $\phi_n|_K \in \mathcal{O}(K, \B)$.
    \end{enumerate}
  \end{lemma}

  \begin{proof}
    (i): 
    Let $T\subset\R^m$ be a tubular neighborhood of $N$ and $r:T\to N$ be a retraction. 
    By Tietze's theorem, there exists $\hat \phi:M\to\R^m$ such that $\hat \phi|_K=\phi$. 
    We set $U\coloneqq\hat \phi^{-1}(T)$ and $\tilde \phi\coloneqq r\circ\hat \phi|_U$. 
    Then, $\tilde \phi:U\to N$ is a continuous extension of $\phi$, i.e., $\tilde \phi|_K=\phi$.
    Let $\delta : N \to \R$ be a continuous function defined by
    \begin{equation} \label{eq:def.delta}
      \delta(x) \coloneqq \sup \{\epsilon\in(0,1] : B_{\epsilon}(x) \subset T\}.
    \end{equation}
    and let $\delta_0\coloneqq\min\{\delta(x):x\in \phi(K)\}$. 
    We can take $n_0$ such that $\|\phi_{n_0}-\phi\|_K<\delta_0$. Since a function 
    $U\ni p\mapsto\delta(\tilde \phi(p))-\|\phi_{n_0}(p)-\tilde \phi(p)\|\in\R$ is continuous,
    $W\coloneqq\{p\in U:\|\phi_{n_0}(p)-\tilde \phi(p)\|<\delta(\tilde \phi(p))\}$
    is an open neighborhood of $K$. 
    We take a bump function $\eta : M \to [0,1]$ with $\eta|_K \equiv 1$ and 
    $\eta|_{M\backslash W}\equiv 0$. Then, for each $p\in W$, we get
    \begin{equation}
      \left\|\left\{\eta(p) \, \tilde{\phi}(p) + 
        (1-\eta(p)) \, \phi_{n_0}(p)\right\} - \tilde{\phi}(p)\right\|
      \le\|\phi_{n_0}(p)-\tilde \phi(p)\|<\delta(\tilde \phi(p)).
    \end{equation}
    Thus, it holds that
    \begin{equation}
      \eta(p) \, \tilde{\phi}(p) + (1-\eta(p)) \, \phi_{n_0}(p) 
      \in B_{\delta(\tilde{\phi}(p))}(\tilde{\phi}(p))\subset T
    \end{equation}
    for all $p \in W$.
    Therefore, we can define a continuous map 
    $F : M\to N$ by $F \coloneqq r\circ(\eta \, \tilde{\phi}+(1-\eta) \, \phi_{n_0})$. 
    This map satisfies $F|_K=\phi$.

    (ii):
    Let $T\subset\C^3$ be a tubular neighborhood of $\B$,  
    $r:T\to\B$ be a smooth retraction, and  
    $\delta$ be a continuous function defined by \eqref{eq:def.delta}. 
    We set $\delta_1\coloneqq\min\{\delta(z):z\in \phi(K)\}$ and
    take sufficiently large $n$ such that $\|\phi_{n}-\phi\|_K<\delta_1$.
    Suppose that $\phi_{n}$ is holomorphic on an open neighborhood $U$ of $K$.
    For the open neighborhood
    $W\coloneqq\{p\in U:\|\phi_{n}(p)-\phi(p)\|<\delta(\phi(p))\}$
    of $K$, we choose a closed subset $A$ satisfying 
    $K\subset \Int(A) \subset A\subset W$ and 
    a bump function $\eta_1 : M \to [0,1]$ satisfying 
    $\eta_1|_A\equiv 1$ and $\eta_1|_{M\backslash W}\equiv 0$.
    Then, $\eta_1(p) \, \phi_{n}(p)+(1-\eta_1(p)) \, \phi(p)\in T$ holds for all $p\in W$. 
    We define a continuous extension $\hat{\phi}_{n}:M\to\B$ of $\phi_n$ by 
    $\hat{\phi}_{n}\coloneqq r\circ(\eta_1 \, \phi_{n}+(1-\eta_1) \, \phi)$.
    This map $\hat{\phi}_{n}$ is holomorphic on the open neighborhood $\Int(A)$ of $K$.
  \end{proof}

  \begin{lemma}\label{lem.sesshoku}
    Let $U$ be a domain in the complex plane $\C$ containing the origin, and 
    let $f,\ F \in \mathcal O(U)$ be holomorphic functions with $f(0)=0$. 
    If $F - f$ vanishes at the origin and $(F-f)(0)>(f)(0)$, 
    then $F$ has a zero at the origin of order $(F)(0)=(f)(0)$, 
    where $(F-f)$, $(F)$, and $(f)$ are principal divisors.
  \end{lemma}
  
  \begin{proof}
    Let $k\coloneqq(F-f)(0)$ and $l\coloneqq(f)(0)$. 
    There exist holomorphic functions $g$ and $h$ such that
    \begin{gather}
      F(z)-f(z)=z^kg(z),\qquad g(0)\ne0,\\
      f(z)=z^lh(z),\qquad h(0)\ne0.
    \end{gather}
    Thus, we get $F(z)=z^l(z^{k-l}g(z)+h(z))$ and 
    this implies $(F)(0)=l=(f)(0)$.
  \end{proof}
  
  \begin{lemma}\label{lem:extendedness}
    Let $M$ be an open Riemann surface and 
    let $K$ be a compact Runge subset of $M$. 
    Then, every map in $\mathcal{A}(K,\B)$ admits a continuous extension to $M$.
  \end{lemma}

  \begin{proof}
    We define a $2$-dimensional complex submanifold $\mathbb{S}_*$ of $\C^3$ and 
    a biholomorphic map $\Xi : \B \to \mathbb{S}_*$ by
    \begin{gather}
      \mathbb{S}_* 
      \coloneqq \{(z^0, z^1, z^2) \in \C^3 \setminus \{0\} : z^0 \, z^1 = (z^2)^2\},\\
      \Xi(z^0, z^1, z^2) \coloneqq (z^0 - z^1, z^0 + z^1, z^2).
    \end{gather}
    \cref{fact:Mergelyan.for.mfd-valued} implies that there is a sequence 
    $\{u_{n} = (u_n^0,u_n^1,u_n^2)\}_{n=1}^{\infty} \subset \mathcal{O}(K,\mathbb{S}_*)$ 
    which uniformly converges to
    $u \coloneqq \Xi \circ \phi \in \mathcal{A}(K,\mathbb{S}_*)$ on $K$. We note that 
    \begin{equation}\label{eq.3.7.1}
      2(u_{n}^2)=(u_{n}^0)+(u_{n}^1),
    \end{equation}
    where $(u_{n}^i)\ (i=0,1,2)$ are principal divisors. 
    Fix $n \in \Z_{>0}$. Let $Z_{n}^1 \coloneqq \{p \in K : u_{n}^1(p) = 0\}$ and 
    $U$ be an open neighborhood of $K$ such that $u_{n}$ is holomorphic on it. 
    By \cref{fact:str.subham.Morse.exh.fucn}, 
    we can take a compact Runge subset $\tilde{K}$ of $M$ such that
    \begin{equation} \label{eq.3.7.2}
      K\subset \Int(\tilde{K}), \qquad \tilde{K} \subset U.
    \end{equation}
    Since $u_{n}\in\mathcal O(\tilde{K},\mathbb{S}_*)$, 
    \cref{fact:Weierstrass.Florack} guarantees the existence of a sequence 
    $\{u_{n,\nu}^1\}_{\nu}\subset \mathcal{O}(M)$ such that
    \begin{itemize}
      \item 
      $\|u_{n,\nu}^1 - u_{n}^1\|_{\tilde{K}}\rightarrow0\quad(\nu\rightarrow\infty)$,
      
      \item 
      $u_{n,\nu}^1 - u_{n}^1$ vanishes $Z_{n}^1$ to order 
      $s_1 \coloneqq \max\{(u_{n}^1)(p):p\in Z_{n}^1\}+1$, and
      
     \item 
     $u_{n,\nu}^1$ has no zeros on $M\setminus Z_{n}^1$.
    \end{itemize}
    Since $(u_{n,\nu}^1 - u_{n}^1)(p)\ge s_1>(u_{n}^1)(p)\ \text{for all}\ p\in Z_{n}^1$, 
    \cref{lem.sesshoku} gives 
    \begin{align}\label{eq.3.7.3}
      (u_{n,\nu}^1)(p)=(u_{n}^1)(p)\qquad(p\in Z_{n}^1).
    \end{align}
    By Fact $\ref{fact:Weierstrass.Florack}$, 
    we get $\{u_{n,\nu}^2\}_{\nu}\subset\mathcal O(M)$ such that
    \begin{itemize}
      \item 
      $\|u_{n,\nu}^2 - u_{n}^2\|_{\tilde{K}} \rightarrow 0\quad (\nu\rightarrow\infty)$, and
      
      \item 
      $u_{n,\nu}^2 - u_{n}^2$ vanishes on $Z_{n}^1$ to order 
      $s_2\coloneqq\max\{(u_{n}^2)(p) : p \in Z_{n}^1\}+1$.
    \end{itemize}
    Since $(u_{n,\nu}^2 - u_{n}^2)(p) \ge s_2 > (u_{n}^2)(p)\ \text{for all}\ p\in Z_{n}^1$, 
    \cref{lem.sesshoku}, \eqref{eq.3.7.1} and \eqref{eq.3.7.3} give
    \begin{equation}
      2(u_{n,\nu}^2)(p)=2(u_{n}^2)(p)=(u_{n}^1)(p)=(u_{n,\nu}^1)(p) \qquad (p\in Z_{n}^1).
    \end{equation}
    Hence, we get $((u_{n,\nu}^2)^2 / u_{n,\nu}^1)(p)=0$ $(p\in Z_{n}^1)$. 
    This means that $(u_{n,\nu}^2)^2 / u_{n,\nu}^1$ has no zeros and poles at $Z_{n}^1$. 
    We set
    \begin{gather}
      u_{n,\nu}^0 \coloneqq \frac{(u_{n,\nu}^2)^2}{u_{n,\nu}^1},\\
      u_{n,\nu}\coloneqq(u_{n,\nu}^0, u_{n,\nu}^1, u_{n,\nu}^2).
    \end{gather}
    If $p \notin Z_{n}^1$, then $u_{n,\nu}^1(p) \ne 0$, and 
    if $p \in Z_{n}^1$, then $u_{n,\nu}^0(p) \ne 0$. 
    Therefore, we obtain $u_{n,\nu} \in \mathcal{O}(M, \mathbb{S}_*)$.

    Let us show $\|u_{n,\nu}-u_{n}\|_{\tilde{K}}\rightarrow0$ $(\nu\rightarrow\infty)$. 
    It is sufficient to show that 
    $\|u_{n,\nu}^0-u_{n}^0\|_{\tilde{K}}\rightarrow0$ $(\nu\rightarrow\infty)$. 
    The maximum modulus principle implies that there exists $p \in \partial \tilde{K}$ 
    such that $\|u_{n, \nu}^0 - u_n^0\|_K = |u_{n, \nu}^0(p) - u_n^0(p)|$.
    We know that $Z_n^1 \subset \tilde{K}$  because of $\eqref{eq.3.7.2}$, thus
    $u_{n, \nu}^1(p) \ne 0$.
    Therefore, we obtain the following:
     \begin{align}
       \|u_{n,\nu}^0 - u_{n}^0\|_{\tilde{K}}
       &=|u_{n, \nu}^0(p) - u_n^0(p)| \\
       &\le\left|u_{n,\nu}^0(p) - \frac{u_{n}^0(p) u_{n}^1(p)}{u_{n,\nu}^1(p)}\right|
        + \left|\frac{u_{n}^0(p) u_{n}^1(p)}{u_{n,\nu}^1(p)} - u_{n}^0(p)\right|\\
       &\le\left|\frac{1}{u_{n,\nu}^1(p)}\right| \, |u_{n,\nu}^2(p)^2-u_{n}^2(p)^2|
         + \left|\frac{1}{u_{n,\nu}^1(p)}\right| \, 
          |u_{n}^0(p)| \, |u_{n,\nu}^1(p)-u_{n}^1(p)|\\
       &\le \left\|\frac{1}{u_{n,\nu}^1}\right\|_{\partial\tilde{K}}
         \left(\|u_{n,\nu}^2\|_{\tilde{K}} + \|u_{n}^2\|_{\tilde{K}}\right) \, 
          \|u_{n,\nu}^2 - u_{n}^2\|_{\tilde{K}}\\
       &\hspace{100pt} + \left\|\frac{1}{u_{n,\nu}^1}\right\|_{\partial\tilde{K}}
         \|u_{n}^0\|_{\tilde{K}} \, \|u_{n,\nu}^1-u_{n}^1\|_{\tilde{K}}.
     \end{align} 
     The sequence $\{\|(u_{n,\nu}^1)^{-1}\|_{\partial\tilde{K}}\}_{\nu} \subset \R$ converges, 
     since the sequence $\{(u_{n,\nu}^1|_{\partial\tilde{K}})^{-1}\}_{\nu}$ converges uniformly 
     to $(u_{n}^1|_{\partial\tilde{K}})^{-1}$ on $\partial\tilde{K}$. 
     Moreover, the sequence $\{\|u_{n,\nu}^2\|_{\tilde{K}}\}_{\nu}$ also converges.
     Hence, there are positive constants $C_1$ and $C_2$ that do not depend on $\nu$, such that
     \begin{equation}
       \|u_{n,\nu}^0 - u_{n}^0\|_{\tilde{K}}
       \le C_1\|u_{n,\nu}^2 - u_{n}^2\|_{\tilde{K}} + 
            C_2\|u_{n,\nu}^1 - u_{n}^1\|_{\tilde{K}}
       \rightarrow 0 \quad (\nu\rightarrow\infty).
     \end{equation}
     Therefore, $\{u_{n,\nu}\}_{\nu}\subset\mathcal O(M,\mathbb{S}_*)$ 
     uniformly approximates $u_{n}$ on $\tilde{K}$. \cref{lem:how.to.extend} (i) implies 
     the existence of a continuous map $\tilde{u}_{n}:M\to\mathbb{S}_*$ satisfying 
     $\tilde{u}_{n}|_{\tilde{K}}=u_{n}$. 
     Clearly, $\{\tilde{u}_{n}\}_{n}$ uniformly converges to $u$ on $K$. We set 
     $\phi_{n}\coloneqq\Xi^{-1}\circ\tilde{u}_{n}$. Then, there exists a constant $C>0$ with 
     $\|\phi_{n}- \phi\|_K\le C \, \|\tilde{u}_{n}-u\|_K$. 
     Consequently, \cref{lem:how.to.extend} (i) gives a continuous extension of $\phi$.
  \end{proof}

  \begin{proof}[Proof of \cref{prop:Mergelyan.for.B-valued.map}]
    By \cref{fact:Mergelyan.for.mfd-valued}, there exists a sequence 
    $\{\tilde{\phi}_{n}\}_{n=1}^{\infty}$ in $\mathcal{O}(S,\B)$ satisfying the following:
    \begin{itemize}
      \item 
      $\|\tilde{\phi}_{n}-\phi\|_S\rightarrow0\quad(n\rightarrow\infty)$,
      
      \item 
      $\tilde{\phi}_{n}$ agrees with $\phi$ at every point of $A$, and
      
      \item 
      $\tilde{\phi}_{n}$ agrees with $\phi$ to order $s$ at every point of $A\cap\Int(S)$.
    \end{itemize}
    Without loss of generality, we may assume $\phi$ is a continuous map defined 
    on entire $M$ because of \cref{lem:extendedness}. By \cref{lem:how.to.extend} (ii), 
    for sufficiently large $n$, the map $\tilde \phi_{n}$ has a continuous extension
    $\hat{\phi}_{n}$ which is holomorphic on an open neighborhood of $S$. 
    Thus, \cref{fact:Runge.for.Oka-valued,fact:B.is.Oka} imply that 
    there exists a sequence of holomorphic maps $\{\hat{\phi}_{n,\nu}\}_{\nu}$ 
    in $\mathcal O(M,\B)$ satisfying the following conditions:
    \begin{itemize}
      \item
      $\|\hat{\phi}_{n,\nu}-\hat \phi_{n}\|_S\rightarrow0\quad(\nu\rightarrow\infty)$,
    
      \item 
      $\hat{\phi}_{n,\nu}$ agrees with $\hat \phi_{n}$ at every point of $A$, and
    
      \item 
      $\hat{\phi}_{n,\nu}$ agrees with $\hat \phi_{n}$ to order $s$ 
      at every point of $A\cap\Int(S)$.
    \end{itemize}
    There exists a positive integer $\nu_{n}$ such that 
    $\|\hat{\phi}_{n,\nu_{n}}-\hat \phi_{n}\|_K < 1/n$.
    Set $\phi_{n}\coloneqq\hat{\phi}_{n,\nu_{n}}$. 
    Then $\{\phi_{n}\}_{n}\subset\mathcal O(M,\B)$ satisfies (i), (ii) and (iii).
  \end{proof}

  The following is the last proposition required to prove \cref{prop:per.prerv.full.approx.}

  \begin{proposition}\label{prop:subspecies.of.inverse.func.thm}
    Let $F:\C^N\to\C^N$ $($resp. $F_{n}:\C^N\to\C^N$ $(n\in\Z_{>0}))$ be holomorphic maps and 
    let $J:\C^N\to \M(N,\C)$ $($resp. $J_{n}:\C^N\to \M(N,\C))$ be Jacobian matrices of 
    $F$ $($resp. $F_{n})$. If $J(0)$ and $J_{n}(0)$ are invertible matrices and 
    $\{J_{n}\}_{n}$ converges uniformly to $J$ on $B=\{t\in\C^N:\|t\|<1\}$, 
    then there exists $\delta > 0$ such that
    for every $n$, there exist open neighborhoods $V_{n}$ of the origin $0\in\C^N$ and 
    $W_{n}$ of $F_{n}(0)\in\C^N$ satisfying the following conditions:
    \begin{enumerate}
      \item 
      the restriction $F_{n}:V_{n}\to W_{n}$ is biholomorphic, and
      
      \item
      $B_{\delta}(F_{n}(0))\subset W_{n}$.
    \end{enumerate}
  \end{proposition}

  \begin{proof}
    We set $L\coloneqq\sup_{n\in\Z_{>0}}\left\|J_{n}\right\|_{B} < \infty$ and 
    denote the $(i,j)$ component of $J_{n}$ by $(J_{n})_{ij}$. 
    For all $t\in B$, we have that 
    \begin{equation}
      |(J_{n})_{ij}(t)|\leq\left\|J_{n}\right\|_{B}\leq L \qquad(n\in\Z_{>0}).
    \end{equation}
    Fix $r>0$ with $D_r(0)^N \subset B$ and 
    $t=(t_1, \dots, t_N)\in D_{r/2}(0)^N$. 
    Then, the sets $C_k\coloneqq \partial D_{r/2}(t_k)$ $(k=1, \dots, N)$ satisfy
    $\prod_{k=1}^NC_k\subset B$. Thus, the Cauchy's integral formula implies that 
    \begin{align}
      \left|\frac{\partial(J_{n})_{ij}}{\partial t_k}(t)\right|
      & = \left(\frac{1}{2\pi}\right)^N\left|\int_{C_1} \dots \int_{C_N}
           \frac{(J_{n})_{ij}(\tau_1, \dots, \tau_N)}
             {(\tau_1-t_1)\cdots(\tau_k-t_k)^2\cdots (\tau_N-t_N)}d\tau_1\cdots d\tau_N\right|\\
      & \leq \left(\frac{1}{2\pi}\right)^N\frac{L}{(r/2)^{N+1}}
          \left(\int_{C_1}|d\tau_1|\right)\cdots\left(\int_{C_N} |d\tau_N|\right)\\
      & = \left(\frac{1}{2\pi}\right)^N\frac{2^{N+1}L}{r^{N+1}}(\pi r)^N
        = \frac{2L}{r}.
    \end{align}
    Hence if $t\in D_{r/2}(0)^N$, then we get 
    \begin{align}
      \left|(J_{n})_{ij}(t)-(J_{n})_{ij}(0)\right|
      &\leq\sum_{k=1}^{N}|(J_{n})_{ij}(t_1, \dots, t_{k-1},t_{k},0, \dots, 0) \\
      & \hspace{100pt} - (J_{n})_{ij}(t_1, \dots, t_{k-1},0,0, \dots, 0)|\\
      &= \sum_{k=1}^{N}
          \left|\int_0^{t_{k}}
           \frac{\partial(J_{n})_{ij}}{\partial\tau_k}(t_1, \dots, t_{k-1},\tau_{k},0, \dots, 0)
             d\tau_{k}\right|\\
      &\leq\sum_{k=1}^{N}\frac{2L}{r}|t_k|
      \leq\left(\sum_{k=1}^N\left(\frac{2L}{r}\right)^2\right)^{\frac{1}{2}}
        \left(\sum_{k=1}^N|t_k|^2\right)^{\frac{1}{2}}
      =\frac{2L\sqrt{N}}{r}\|t\|.
    \end{align}
    Therefore, for every $\epsilon>0$, if $\|t\|<\min\{r\epsilon/(2L\sqrt{N}),\,r/2\}$, then 
    \begin{equation}
      \left|(J_{n})_{ij}(t)-(J_{n})_{ij}(0)\right|<\epsilon\qquad(n\in\Z_{>0})
    \end{equation}
    holds. This guarantees $\{(J_{n})_{ij}\}_{n}$ is equicontinuous at the origin $0\in B$. 
    We define $\tilde{F}_{n}:\C^{N}\to\C^{N}$ by 
    \begin{equation}
      \tilde{F}_{n}(t)\coloneqq J_{n}(0)^{-1}(F_{n}(t)-F_{n}(0)).
    \end{equation}
    The equicontinuity of $\{(J_{n})_{ij}\}_{n}$
    gives the equicontinuity of $\{(J\tilde{F}_{n})_{ij}\}_{n}$ at $0\in B$. 
    We choose $\tilde{\delta}>0$ satisfying the following:
    \begin{itemize}
      \item[(a)]
      $B_{2\tilde{\delta}}(0)\subset B$, and
    
      \item[(b)]
      for all $n$, if $t\in B_{2\tilde{\delta}}(0)$ then 
      $|(J\tilde{F}_{n})_{ij}(t)-(J\tilde{F}_{n})_{ij}(0)|<\dfrac{1}{2N}.$
    \end{itemize}
    Fix $n\in\Z_{>0}$. Let us show that 
    \begin{equation}\label{eq.6.5}
      \|(\tilde{F}_{n}(t)-t)-(\tilde{F}_{n}(u)-u)\|\leq\frac{1}{2}\|t-u\|
    \end{equation}
    holds for all $t,u\in\tilde{V}_{n}\coloneqq B_{2\tilde{\delta}}(0)$. 
    Suppose that $k_{n}=(k^1_{n}, \dots, k^{N}_{n}):B\to\C^{N}$ is a holomorphic map defined by 
    $k_{n}(t)\coloneqq \tilde{F}_{n}(t)-t$. If $t,u\in\tilde{V}_{n}$, 
    then the above condition (b) yields
    \begin{align}
      |k^i_{n}(t)-k^i_{n}(u)|
       &=\left|\int_u^t dk^i_n\right| 
          \le\int_0^1\sum_{j=1}^{N}\left|(Jk_{n})_{ij}(u+s(t-u))\cdot (t_j-u_j)\right| ds \\
       &= \int_0^1
           \sum_{j=1}^{N}
            \left|(J\tilde{F}_{n})_{ij}(u+s(t-u))-J\tilde{F}_n(0)\right|\cdot |t_j-u_j| ds \\
       &< \sum_{j = 1}^N\frac{1}{2N} |t_j-u_j|
           \le \left(\sum_{j=1}^{N}\left(\frac{1}{2N}\right)^2\right)^\frac{1}{2}\|t-u\|
       = \frac{1}{2\sqrt{N}}\|t-u\|.
    \end{align}
    This implies $\left\|k_{n}(t)-k_{n}(u)\right\|^2 < (1/4)\|t-u\|^2$.
    Now, we obtain $\eqref{eq.6.5}$. Moreover, 
    \begin{equation}\label{eq.3.8}
      \|\tilde{F}_{n}(t)-\tilde{F}_{n}(u)\|
      \geq\|t-u\| - \|(\tilde{F}_{n}(t)-t)-(\tilde{F}_{n}(u)-u) \|
      \geq\frac{1}{2}\|t-u\|
    \end{equation}
    holds for all $t, u\in \tilde V_{n}$. 
    Let $\tilde{W}_{n}\coloneqq\tilde{F}_{n}(\tilde{V}_{n})$. 
    The inequality  $\eqref{eq.3.8}$ ensures that the map 
    $\tilde{F}_{n}:\tilde{V}_{n}\to\tilde{W}_{n}$ is injective, and thus
    it is biholomorphic. Next, let us prove
    \begin{equation}\label{eq.6.6}
      B_{\tilde{\delta}}(0)\subset\tilde{W}_{n}\qquad(n\in\Z_{>0}).
    \end{equation}
    It is sufficient to show that for fixed $u\in B_{\tilde{\delta}}(0)$, 
    there exists $t\in\tilde{V}_{n}$ such that $u=\tilde{F}_{n}(t)$. 
    We define a sequence $\{t_i\}_{i}\subset\C^{N}$ by 
    \begin{equation}\label{eq.6.7}
      t_1=0,\qquad t_i=u+t_{i-1}-\tilde{F}_{n}(t_{i-1})\quad(i\geq2).
    \end{equation}
    Then, $\{t_i\}_{i}\subset\tilde{V}_{n}$ is shown by induction. 
    Indeed, $t_1=0\in\tilde{V}_{n}$ and $t_2=u\in\tilde{V}_{n}$ hold. 
    Assume $t_1$, $t_2$, \dots,  $t_{i-1}\in\tilde{V}_{n}$ for $i\geq3$. 
    If $3\leq j\leq i$, then $\eqref{eq.6.5}$ gives
    \begin{align}
      \|t_j-t_{j-1}\|
       &=\left\|\left(u+t_{j-1}-\tilde{F}_{n}(t_{j-1})\right) - 
          \left(u+t_{j-2}-\tilde{F}_{n}(t_{j-2})\right)\right\|\\
       &\leq\frac{1}{2}\left\|t_{j-1}-t_{j-2}\right\|
        \leq\left(\frac{1}{2}\right)^{j-2}\|t_2-t_1\|
        =\left(\frac{1}{2}\right)^{j-2}\|u\|.
    \end{align}
    Hence, we get
    \begin{equation}
      \|t_i\|=\|t_i-t_1\|
      \leq\sum_{j=2}^i\|t_j-t_{j-1}\|
      \leq\sum_{j=2}^i\left(\frac{1}{2}\right)^{j-2}\|u\|<2\tilde{\delta}.
    \end{equation}
    From this, it follows that $\{t_i\}_{i}\subset\tilde{V}_{n}$. 
    Let $i < j$. Then we have
    \begin{equation}
      \|t_i-t_j\|
       \leq \sum_{k=i+1}^{j}\|t_k - t_{k-1}\|
       \leq \sum_{k=i+1}^{j}\left(\frac{1}{2}\right)^{k-2}\|u\|
       < \left(\frac{1}{2}\right)^{i-2}\tilde{\delta}.
    \end{equation}
    Thus $\{t_i\}_i$ is a Cauchy sequence. 
    The limit $t\coloneqq\lim_{i\rightarrow\infty}t_i$ of $\{t_i\}_i$ satisfies 
    \begin{equation}
      \|t\|
       = \lim_{i\rightarrow\infty}\|t_i\|
       \leq \lim_{i\rightarrow\infty}\sum_{j=2}^i\left(\frac{1}{2}\right)^{j-2}\|u\|
       = 2 \, \|u\|
       < 2 \, \tilde{\delta}.
    \end{equation}
    This means that $t\in\tilde{V}_{n}$. 
    Furthermore, the limit on both sides of the second equality 
    $\eqref{eq.6.7}$ yields $u=\tilde{F}_{n}(t)$. Now, we obtain $\eqref{eq.6.6}$.
    The map 
    $G_{n} : \C^{N} \to \C^{N}$ given by 
    $G_{n}(t)\coloneqq J_{n}(0)^{-1}\left(t-F_{n}(0)\right)$ satisfies 
    $\tilde{F}_{n}=G_{n}\circ F_{n}$. Since $G_{n}$ is biholomorphic, 
    $V_{n}\coloneqq\tilde{V}_{n}$ and $W_{n}\coloneqq G_{n}^{-1}(\tilde{W}_{n})$ satisfy (i). 
    Finally, let us show (ii). By $\eqref{eq.6.6}$, it is sufficient to show that 
    there exists $\delta>0$ such that 
    $B_{\delta}(F_{n}(0))\subset G_{n}^{-1}(B_{\tilde{\delta}}(0))$ $(n\in\Z_{>0})$.
    Since the sequence $\{J_{n}(0)^{-1}\}_{n}$ converges to $J(0)^{-1}$, 
    there is a positive number $C>0$ such that $\|J_{n}(0)^{-1}\|<C$ holds for all $n$. 
    We set $\delta\coloneqq\tilde{\delta}/C$. If $t\in B_{\delta}(F_{n}(0))$ then we get
    $\|G_{n}(t)\| = \|J_{n}(0)^{-1}(t-F_{n}(0))\| < C\,\|t-F_{n}(0)\| < \tilde{\delta}$.
    Therefore $t$ is an element of $G_{n}^{-1}(B_{\tilde{\delta}}(0))$.
  \end{proof} 

  \begin{proof}[Proof of \cref{prop:per.prerv.full.approx.}]
    We prove this proposition in four steps.
    
    \emph{Step 1}: Let us show that if $\phi \in \mathcal{A}(S,\B)$ satisfies
    $\Sigma(\phi)\coloneqq\mathrm{Span}_{\C} \, \phi(S)=\C^3$, then there exists 
    $\{F_{n}\}_{n} \subset\mathcal{O}(M,\B)$ with conditions (i)--(v).
    By \cref{prop:Mergelyan.for.B-valued.map}, there is a sequence 
    $\{\phi_{n}\}_{n}\subset\mathcal{O}(M,\B)$ satisfying the following properties:
    \begin{itemize}
     \item 
     $\|\phi_{n}-\phi\|_S\rightarrow 0 \quad(n\rightarrow\infty)$,
     
     \item 
     $\phi_{n}$ agrees with $\phi$ at every point of $A$, and
     
     \item 
     $\phi_{n}$ agrees with $\phi$ to order $s$ at every point of $A\cap\Int(S)$.
    \end{itemize}
    Since $\Sigma(\phi)=\C^3$ holds, $\phi_{n}$ is a full map for sufficiently large $n$. 
    By retaking index numbers, we may assume $\phi_{n}$ is full for each $n\in\Z_{>0}$. 
    \cref{lem.rest.full} implies $\mathrm{Span}_{\C} \, \phi_{n}(I_i)=\C^3$, and in particular, 
    $\phi_{n}(I_i)$ is not contained in any complex line passing through the origin. 
    From \cref{prop:existence.Prd.dominating.spray}, 
    there is a period dominating spray $\tilde\Phi_{n}$ with core $\phi_{n}|_S$ 
    that is expressed as follows, using holomorphic functions 
    $\tilde{h}_{ik,n}$ and $g_n$ depending on $n$ (see \eqref{e.q.period.dominating}):
    \begin{equation}
      \tilde\Phi_{n}(p,t_{11}, \dots, t_{l3})
       = \psi^{g_n(p) \, \tilde{h}_{11,n}(p) \, t_{11}}_{11}\circ\cdots\circ
          \psi^{g_n(p) \, \tilde{h}_{l3,n}(p) \, t_{l3}}_{l3}\left(\phi_{n}(p)\right),
    \end{equation}
    where the functions $\tilde{h}_{11,n}, \dots, \tilde{h}_{l3,n}$ and $g_n$ are chosen 
    as holomorphic functions defined on the entire $M$ since $S$ and 
    $C=\bigcup_{i=1}^lC_i$ are Runge in $M$. 
    Thus $\tilde\Phi_{n}$ is a holomorphic map from $M\times\C^{3l}$ into $\B$. 
    We fix sufficiently large $n_0$ and define $\Phi_{\phi}\in\mathcal A(S\times\C^{3l},\B)$ and
    $\Phi_{n}\in\mathcal O(M\times\C^{3l},\B)$ $(n\ge n_0)$ by 
    \begin{align}
      \Phi_{\phi}(p,t_{11}, \dots, t_{l3})
       &=\psi^{g_{n_0}(p) \, \tilde{h}_{11,n_0}(p) \, t_{11}}_{11}\circ\cdots\circ
          \psi^{g_{n_0}(p) \, \tilde{h}_{l3,n_0}(p) \, t_{l3}}_{l3}\left(\phi(p)\right),\\
      \Phi_{n}(p,t_{11}, \dots, t_{l3})
       &=\psi^{g_{n_0}(p)\tilde{h}_{11,n_0}(p)t_{11}}_{11}\circ\cdots\circ
          \psi^{g_{n_0}(p) \, \tilde{h}_{l3,n_0}(p) \, t_{l3}}_{l3}\left(\phi_{n}(p)\right).
    \end{align}
    By \cref{rem.6.9}, there is $\epsilon>0$ such that if 
    $\phi$ and $\phi_{n}$ satisfy 
    \begin{equation}\label{eq.314}
      \|\phi-\phi_{n_0}\|_S  < \epsilon \quad\text{and}\quad
      \|\phi_{n}-\phi_{n_0}\|_S<\epsilon,
    \end{equation}
    then $\Phi_{\phi}$ (resp. $\Phi_{n}$) is a period dominating spray 
    with core $\phi$ (resp. $\phi_{n}$). We may assume $n_0$ is sufficiently large 
    to satisfy \eqref{eq.314}. 
    Let $\mathcal{F}, \, \mathcal{F}_{n} : \C^{3l}\to\C^{3l}$ $(n\ge n_0)$ be 
    holomorphic maps defined by 
    \begin{equation}
      \mathcal F(t)\coloneqq\Prd(\Phi_{\phi}(\cdot,t)),\qquad
      \mathcal F_{n}(t)\coloneqq\Prd(\Phi_{n}(\cdot,t)).
    \end{equation}
    Then, the partial derivatives of $\mathcal F$ and $\mathcal F_{n}$ are given by
    \begin{align}
      &\frac{\partial\mathcal F}{\partial t_{ik}}
        =\left(\int_{C_j}\frac{\partial}{\partial t_{ik}}
           \left(\exp(t_{11}a_{11}^{n_0}A_{11}^{n_0})\cdots
            \exp(t_{l3}a_{l3}^{n_0}A_{l3}^{n_0})\phi\right)\theta\right)_{j=1,...,l}\\
        &\ =\left(\int_{C_j}\exp(t_{11}a_{11}^{n_0}A_{11}^{n_0})\cdots
            a_{ik}^{n_0}A_{ik}^{n_0}\exp(t_{ik}a_{ik}^{n_0}A_{ik}^{n_0})\cdots
             \exp(t_{l3}a_{l3}^{n_0}A_{l3}^{n_0})\phi\theta\right)_{j=1,...,l},\\
      &\frac{\partial\mathcal F_{n}}{\partial t_{ik}}
        =\left(\int_{C_j}\frac{\partial}{\partial t_{ik}}
           \left(\exp(t_{11}a_{11}^{n_0}A_{11}^{n_0})\cdots
            \exp(t_{l3}a_{l3}^{n_0}A_{l3}^{n_0})\phi_{n}\right)\theta\right)_{j=1,...,l}\\
        &\ =\left(\int_{C_j}\exp(t_{11}a_{11}^{n_0}A_{11}^{n_0})\cdots
            a_{ik}^{n_0}A_{ik}^{n_0}\exp(t_{ik}a_{ik}^{n_0}A_{ik}^{n_0})\cdots
             \exp(t_{l3}a_{l3}^{n_0}A_{l3}^{n_0})\phi_{n}\theta\right)_{j=1,...,l},
    \end{align}
    where we set $a_{ik}^{n_0}\coloneqq gh_{ik,n_0}$ and 
    $A_{ik}^{n_0}\in \M(3,\C)$ $(1\leq i \leq l, \, k = 1, 2, 3)$ are suitable matrices. 
    Let $B\coloneqq\{t\in\C^{3l}:\|t\|<1\}$. We have a constant $C>0$ such that 
    \begin{equation}
      \left\|
        \frac{\partial\mathcal F_{n}}{\partial t}-\frac{\partial\mathcal F}{\partial t}
         \right\|_B
      \leq C\|\phi_{n}-\phi\|_S \quad \text{for all $n\ge n_0$}.
    \end{equation} 
    Thus the sequence of Jacobian matrices $\{\partial_t\mathcal F_{n}\}_{n}$ 
    uniformly converges to $\partial_t\mathcal F$ on $B$. 
    By \cref{prop:subspecies.of.inverse.func.thm}, 
    we obtain $\delta > 0$, an open neighborhood $V_{n}$ of the origin $0\in\C^{3l}$, and 
    an open neighborhood $W_{n}$ of $\mathcal F_{n}(0)$ satisfying the following conditions:
    \begin{itemize}
      \item 
      the restriction $\mathcal F_{n}:V_{n}\to W_{n}$ is biholomorphic, and
      
      \item  
      $B_{\delta}(\mathcal{F}_{n}(0))\subset W_{n}$ for all $n\ge n_0$.
    \end{itemize}
    By the continuity of $\mathcal F_{n}^{-1}$ at $\mathcal{F}_n(0)$, 
    there is a positive number $r>0$ such that
    \begin{equation}\label{eq.6.10}
      \|\mathcal F_{n}^{-1}(u)\| < \frac{1}{n} 
      \quad \text{for all } u\in B_r(\mathcal{F}_{n}(0)).
    \end{equation} 
    Let $r'\coloneqq\min\{\delta,r\}$. We can take $n_1\in\Z_{>n_0}$ so that 
    if $n\ge n_1$ then $\|\mathcal{F}_{n}(0)-\mathcal{F}(0)\| = \|\Prd(\phi_{n})-\Prd(\phi)\|<r'$. 
    Hence $\Prd(\phi)$ is an element of $B_{\delta}(\mathcal{F}_{n}(0))\subset W_{n}$ 
    for all $n\ge n_1$. This gives $t_{n}\in V_{n}$ with $\mathcal F_{n}(t_{n})=\Prd(\phi)$. 
    Also, the fact that $\mathcal{F}_{n}(t_{n})\in B_r(\mathcal{F}_{n}(0))$ and 
    \eqref{eq.6.10} imply that
    \begin{equation}
      \|t_{n}\|
      = \|\mathcal F_{n}^{-1}\left(\mathcal F_{n}(t_{n})\right)\|
      < \frac{1}{n} \qquad (n\ge n_1).
    \end{equation}
    Thus, the sequence $\{t_{n}\}_{n}$ converges to $0\in\C^{3l}$.
    We set $F_{n}\coloneqq\Phi_{n}(\cdot,t_{n})$. 
    The definitions of $\Phi_{n}$ and $\Phi_{\phi}$ yield 
    the existence of positive numbers $C_1, \, C_2>0$ such that 
    $\|F_{n}-\phi\|_S
    \le \|\Phi_n(\cdot, t_n) - \Phi_{\phi}(\cdot, t_n)\|_S + \|\Phi_{\phi}(\cdot, t_n) - \phi\|_S
    \le C_1 \,\|\phi_{n}-\phi\|_S + C_2 \, \|t_n\| \rightarrow 0$.
    Since $\Sigma(\phi) = \C^3$ holds, $\Phi_{n}(\cdot,t_{n})$ is a full map 
    for sufficiently large $n$. Moreover, the properties for 
    the sequence $\{\phi_{n}\}_{n}$ and \cref{prop:existence.Prd.dominating.spray} guarantee that 
    $F_{n}$ agrees with $\phi$ at every point of $A$ and agrees with $\phi$ to order $s$ 
    at every point of $A\cap\Int(S)$. 
    Therefore, $\{F_{n}\}_{n}\subset\mathcal O(M,\B)$ satisfies conditions (i)--(v).

    \emph{Step 2}: Let us show that for every $\phi \in \mathcal{A}(S,\B)$
    with $\Sigma(\phi)\subsetneq\C^3$, there exists a sequence 
    $\{g_{n}\}_{n} \subset \mathcal A(S,\B)$ satisfying 
    $\dim\Sigma(\phi)<\dim\Sigma(g_{n})$, and conditions (ii)--(v).
    Let $ P \subset S\setminus A$ be a set consisting of one or two points such that 
    $\{\phi(p) : p \in P\}$ is a basis of $\Sigma(\phi)$. There exist a point 
    $p_0\in S\setminus(A \cup P)$ and a holomorphic map $V:\C^3\to\C^3$ satisfying
    $V(\phi(p_0))\notin\Sigma(\phi)$ and $V(w)\in T_w\B$ for all $w\in\B$. 
    Indeed, if $\dim\Sigma(\phi) = 1$, then every point $p_0\in S\setminus(A\cup P)$ satisfies 
    $T_{ \phi(p_0)}\B\neq\Sigma(\phi)$. We fix $p_0$ and take 
    $v\in T_{\phi(p_0)}\B\setminus\Sigma(\phi)$. 
    \cref{lem:vec.fld.and.flow.maxface} implies that there exists a holomorphic map 
    $V:\C^3 \to \C^3$ with $V(\phi(p_0)) = v \notin \Sigma(\phi)$ and 
    $V(w)\in T_w\B$ for all $w\in\B$. If $\dim\Sigma(\phi) = 2$, then we can take two points 
    $p,p'\in S\setminus(A \cup P)$ such that $\phi(p)$ and $\phi(p')$ are 
    $\C$-linearly independent. This and \cref{lem:tan.sp.of.B} yields 
    $T_{\phi(p)}\B\neq T_{\phi(p')}\B$. We choose 
    $v\in (T_{\phi(p)}\B\cup T_{ \phi(p')}\B)\setminus\Sigma(\phi)$. 
    If $v\in T_{\phi(p)}\B$, then we set $p_0\coloneqq p$, and if $v\in T_{\phi(p')}\B$, 
    then we set $p_0\coloneqq p'$. Again \cref{lem:vec.fld.and.flow.maxface} guarantees 
    the existence of a suitable holomorphic map $V:\C^3\to\C^3$. 
    Let $\psi:\C\times\C^3\to\C^3$ be a holomorphic flow of $V$ 
    obtained by the matrix exponential map.
  
    By \cref{fact:Weierstrass.thm.op.Riem.surf}, we can choose a holomorphic function 
    $h\in\mathcal{O}(S)$ that has zeros only at $A \cup P$ and 
    vanishes to order $s$ at every point of $A$.
    Given $\xi\in\mathcal{O}(M)$, we define $\Theta(\xi)\in\mathcal A(S,\B)$ by 
    $\Theta(\xi) \coloneqq \psi(h \, \xi, \phi) 
      = \phi+  h \, \xi \cdot (V\circ \phi) + o(|h \,\xi|) $.
    Since $\mathcal{O}(M)$ is an infinite-dimensional vector space, 
    we can choose linearly independent functions 
    $\xi_1, \xi_2, \dots, \xi_{3l +1} \in \mathcal{O}(M)$.
    Let $\Xi : \C^{3l + 1} \to \C^{3l}$ be a holomorphic map given by
    $\Xi(z_1, \dots, z_{3l + 1}) \coloneqq
      \Prd\left(\Theta(\sum_{i = 1}^{3l + 1} z_i \, \xi_i) \right) - \Prd(\phi)$.
    Then, the origin $0 \in \C^{3l + 1}$ is contained in the analytic set 
    $X \coloneqq\{z \in \C^{3 l + 1} : \Xi(z) = 0\}$. 
    Hence, we can take
    $\{z_n = (z_{n, 1}, \dots, z_{n, 3l+1})\}_n \subset X \setminus \{0\}$ 
    such that $z_n \rightarrow 0$.
    The functions $\eta_{n} \coloneqq \sum_{i = 1}^{3l + 1}z_{n,i} \, \xi_i$ $(n\in\Z_{>0})$
    satisfy $\|\eta_{n}\|_{\tilde S} \rightarrow 0$ and 
    $\Prd(\Theta(\eta_{n})) - \Prd(\phi) = \Xi(z_n) = 0$. 
    We set $g_{n}\coloneqq\Theta(\eta_{n})$. It satisfies (iii) and the following:
    \begin{equation}\label{eq.6.8.5}
      g_{n}(p) = \psi(h(p) \, \eta_{n}(p), \phi(p))=\phi(p) \qquad (p \in P).
    \end{equation}
    This implies that $\phi(P)\subset g_{n}(S)$, i.e., $\Sigma(\phi)\subset\Sigma(g_{n})$. 
    Since $\eta_{n}$ is holomorphic at $p_0\in S$, 
    we can assume $\eta_{n}(p_0)\neq0$ by moving $p_0$ slightly under the conditions $h(p_0)\neq0$ 
    and $V(\phi(p_0)) \notin \Sigma(\phi)$. 
    It follows that $\phi(p_0) + h(p_0) \eta_{n}(p_0) V(\phi(p_0)) \notin \Sigma(\phi)$. 
    Since $g_n(p_0) \approx \phi(p_0) + h(p_0) \eta_{n}(p_0) V(\phi(p_0))$ 
    for sufficiently large $n$, we obtain $g_n(p_0) \notin \Sigma(\phi)$. 
    Thus, $\Sigma(\phi) \subsetneq \Sigma(g_{n})$, 
    which implies $\dim\Sigma(\phi) < \dim\Sigma(g_{n})$.
    We take a matrix $B\in \M(3,\C)$ and a holomorphic map $\beta:\C\to \M(3,\C)$ satisfying 
    $g_{n} = (E_3 + \eta_{n}\,h \, B + \eta_{n}^2\, h^2 \, \beta(\eta_{n}h))\cdot \phi$ 
    and set 
    \begin{equation}
      C_1 \coloneqq \|h\|_S,\qquad 
      C_2 \coloneqq \sup\{\|\beta(z)\|:z\in\C,\,|z|\leq C_1\}.
    \end{equation}
    Since $|\eta_{n}(p) \, h(p)| \leq C_1$ holds for all $p\in S$ and 
    sufficiently large $n\in\Z_{>0}$, we have  
    $\{\eta_{n}(p) \, h(p) : p \in S,\,n : \text{sufficiently large}\}
      \subset\{z\in\C:|z|\leq C_1\}$. 
    Hence it holds that
    $\|\beta(\eta_{n} \, h) \, \phi\|_S \leq C_2\,\|\phi\|_S$.
    Thus, for all $p\in S$, we obtain 
    \begin{align}
      \|g_{n}(p)- \phi(p)\|
      &\leq|\eta_{n}(p)h(p)| 
        \left\{\|B \, \phi(p)\|+ |h(p)| \, \|\beta(\eta_{n}(p)h(p))\phi(p)\|\right\}\\
      &\leq C_1\, \|\eta_n\|_S \, 
        \left(\|B\| \, \|\phi\|_S + C_1 \, C_2 \, \| \phi\|_S \right).
    \end{align}
    This and $\|\eta_n\|_S \rightarrow 0$ mean that $\{g_{n}\}_{n}$ satisfy (ii). Since 
    $g_{n} - \phi=h\cdot(\eta_{n}V \circ \phi + \eta_{n}^2 h \beta(\eta_{n}h))$ 
    vanishes on $A$, and to order $s$ on $A\cap\Int(S)$, $g_{n}$ satisfies (iv) and (v). 

    \emph{Step 3}: Let us show that for each $j\in\{1, 2\}$, every map 
    $\phi\in X_j\coloneqq\{\hat \phi\in\mathcal{A}(S,\B):\dim\Sigma(\hat{\phi})= j\}$ 
    is a limit of a sequence $\{G_{n}\}_{n}\subset\mathcal A(S,\B)$ satisfying $\Sigma(G_{n})=\C^3$ 
    and the conditions (ii)--(v).
    If $j = 2$, then Step 2 implies the conclusion. 
    Let $\phi\in X_1$. By Step 2, we can take a sequence $\{g_{n}\}_{n}\subset\mathcal{A}(S,\B)$ 
    with $\dim\Sigma(\phi)<\dim\Sigma(g_{n})$ and the conditions (ii)--(v). 
    Fix $n\in\Z_{>0}$. If $\Sigma(g_n) = \C^3$, then we set $g_{n, \nu} \coloneqq g_n$ for all 
    $\nu \in \Z_{>0}$. If $g_{n}\in X_2$, then we take a sequence 
    $\{g_{n,\nu}\}_{\nu}\subset\mathcal{A}(S,\B)$ obtained by applying Step 2 to $g_n$. 
    In either case, $\{g_{n,\nu}\}_{\nu}$ satisfies the following conditions:
    \begin{itemize}
      \item 
      $\Sigma(g_{n,\nu})=\C^3$ holds for all $\nu\in\Z_{>0}$,
      
      \item 
      $\|g_{n,\nu}-g_{n}\|_S \rightarrow 0\quad(\nu\rightarrow\infty)$,
      
      \item 
      $\Prd(g_{n,\nu})=\Prd(g_{n})=\Prd(\phi)$ holds for all $\nu\in\Z_{>0}$, and
      
      \item 
      $g_{n,\nu} - \phi = (g_{n, \nu} - g_n) + (g_n - \phi)$ vanishes on $A$, and
      to order $s$ on $A \cap \Int(S)$.
    \end{itemize}
    We take a positive integer $\nu_{n}$ such that $\|g_{n,\nu_{n}}-g_{n}\|_S<1/n$ and 
    set $G_{n}\coloneqq g_{n,\nu_{n}}$. This map $G_n$ satisfies
    \begin{equation}
      \|G_{n}-\phi\|_S
      \leq\|G_{n}-g_{n}\|_S + \|g_{n} - \phi\|_S
      \leq\frac{1}{n}+\|g_{n}-\phi\|_S \rightarrow 0.
    \end{equation}
    Moreover, $\{G_{n}\}_{n}$ also satisfies (iii), (iv) and (v).

    \emph{Step 4: Completion of proof.} 
    Let $\phi \in\mathcal{A}(S,\B)$. By Step 1, if $\Sigma(\phi) = \C^3$, 
    then the assertion holds true. Let $\Sigma(\phi) \subsetneq \C^3$, i.e., 
    $\phi \in X_1 \cup X_2$. By Step 3, there is a sequence 
    $\{G_{n}\}_{n}$ in $\mathcal{A}(S,\B)$ satisfying $\Sigma(G_{n})=\C^3$ 
    and conditions (ii)--(v). Fix $n\in\Z_{>0}$. Step 1 gives a sequence 
    $\{G_{n,\nu}\}_{\nu}\subset\mathcal{O}(M,\B)$ satisfying the following:
    \begin{itemize}
      \item 
      $G_{n,\nu}$ is a full map for all $\nu$,
      
      \item 
      $\|G_{n,\nu}- G_{n}\|_S\rightarrow 0\quad(\nu\rightarrow\infty)$,
      
      \item 
      $\Prd(G_{n,\nu})=\Prd(G_{n})=\Prd(\phi)$ holds for all $\nu$, and 
      
      \item 
      $G_{n,\nu} - \phi = (G_{n, \nu} - G_n) + (G_n - \phi)$ vanishes on $A$, and
      to order $s$ on $A \cap \Int(S)$.
    \end{itemize}
    We choose $\nu_{n}\in\Z_{>0}$ with $\|G_{n,\nu_{n}}-G_{n}\|_S<1/n$ and 
    define holomorphic maps $\phi_{n}:M\to\B$ $(n=1,2,...)$ by 
    $\phi_{n}\coloneqq G_{n,\nu_{n}}$. Then, $\{\phi_{n}\}_{n}\subset\mathcal O(M,\B)$ 
    is the desired sequence.
  \end{proof}
  
\section{Approximation and Interpolation Theorems for Maxfaces}
  In this section, we state and prove the approximation and interpolation theorems for maxfaces
  (\cref{prop:weak.approx.,thm:main}), 
  which correspond to {\cite[Proposition 3.3.2, Theorem 3.6.1]{AFL21}}, and 
  present their corollaries (\cref{cor:main.,cor:dense.im.sing}). 
  Although the proofs of \cref{prop:weak.approx.,thm:main} follow the arguments in \cite{AFL21},
  their assertions incorporate new conditions concerning surface singularities. 
  \cref{cor:main.,cor:dense.im.sing} are results obtained by using these conditions about singularities.
  
  \begin{proposition} \label{prop:weak.approx.}
    Assume that $M$ is an open Riemann surface, 
    $\theta$ is a nonvanishing holomorphic 1-form on $M$,
    $S=K \cup E\subset M$ is a connected admissible set such that 
      $\hml{S} \emb \hml{M}$ $($in particular, $S$ is Runge in $M )$, 
    $A \subset S$ and $\Sigma \subset A \cap \Int(S)$ are finite subsets,
    $\mathcal{C}$ is a homology basis of $S$ 
      obtained by \cref{fact:homology.basis.of.addmissible}, and
    $(f,\phi \,\theta)$ is a generalized maxface from $S$ into $\L$.
    Given a positive number $\epsilon > 0$ and a positive integer $s \in \Z_{>0}$, 
    there exists a full maxface $\tilde{f} : M\to \L$ satisfying the following conditions. 
    \begin{enumerate}
      \item 
      $\|\tilde{f} - f\|_S < \epsilon$, $\|\tilde{\phi} - \phi\|_S < \epsilon$,
      where $\tilde{\phi} = 2 \, \partial \tilde{f} / \theta \in \mathcal{O}(M, \B)$.
      
      \item
      The differences $\tilde{f} - f$ and $\tilde{\phi} - \phi$ vanish 
      at every point of $A$, and to order $s$ at every point of $A \cap \Int(S)$.

      \item 
      $\Flux_{\tilde{f}} = \Flux_f^{\mathcal{C}}$ on $\hml{S} \emb \hml{M}$.

      \item 
      If the maxface $f|_{\Int(S)}$ is $\mathcal{A}$-equivalent to a cuspidal edge
      $($resp. swallowtail, cuspidal cross cap, cuspidal butterfly, cuspidal $S_1^-$ singularity$)$
      at $p \in \Sigma$, then $\tilde{f}$ is also $\mathcal A$-equivalent to a cuspidal edge
      $($resp. swallowtail, cuspidal cross cap, cuspidal butterfly, cuspidal $S_1^-$ singularity$)$
      at $p$.
    \end{enumerate}
  \end{proposition}

  \begin{proof}
    Suppose $K_i$ (resp. $E_k$) is a connected component of $K$ (resp. $E$). 
    Let $a_{i, j},q_i\in K_i$ and $A_{i,j}\subset K_i$ be the points and the arcs chosen in 
    the proof of \cref{fact:homology.basis.of.addmissible}. 
    We enlarge $A$ by adding to it the endpoints of all connected components $E_k$ and $q_1 \in K_1$. 
    Let us construct a collection $\tilde{\mathcal{C}}$ of arcs and closed curves in $S$ 
    in the following way.
    \begin{itemize}
      \item[(a)] 
      Suppose $C \in \mathcal{C}$ satisfies $C \cap A \neq \varnothing$. 
      We split $C$ into a union of arcs whose endpoints lie in $C \cap A$, and 
      label these arcs as $C_1, \dots, C_N$ so that the terminal point of $C_i$ coincides with 
      the initial point of $C_{i+1}$, where $C_{N+1} = C_1$. 
      Since $q_1 \in C$, we may assume that the initial point of $C_1$ and 
      the terminal point of $C_N$ are both $q_1$. 
      Let $\tilde{C}_i \coloneqq \bigcup_{j=1}^i C_j$ for $i = 1, \dots, N$, and 
      let $\mathcal{C}_1$ be the collection of all arcs $\tilde{C}_i$ obtained by 
      this procedure for each $C$ with $C \cap A \neq \varnothing$. 
     
      \item[(b)] 
      Suppose $E_k$ is not contained in any $C \in \mathcal{C}$ and 
      one of its endpoints $e$ belongs to $K_i$. We choose the point $a_{ij}$ 
      that lies in the connected component of $\partial K_i$ containing $e$. 
      Let $e'$ be the other endpoint of $E_k$. We construct an arc connecting 
      $q_1 \in K_1$ to $e'$ as follows: first, we connect $q_1$ to $q_i \in K_i$ 
      as in the proof of \cref{fact:homology.basis.of.addmissible}; then, 
      using $a_{ij}$ and the arc $A_{ij} \subset K_i$, we connect $q_i$ to $e$ in the same manner; 
      finally, we connect $e$ to $e'$ along $E_k$. 
      We split the resulting arc into a union of arcs $\tilde{C}_i$ in the same way as in (a), 
      noting that the initial point of $C_1$ is $q_1$ and the terminal point of $C_N$ is $e'$. 
      Let $\mathcal{C}_2$ be the collection of all arcs obtained in this way.
    
      \item[(c)] 
      Suppose $a\in A'\coloneqq A\setminus(\bigcup(\mathcal{C}\cup\mathcal{C}_1\cup\mathcal{C}_2))$. 
      Then $a$ lies in $K_i$. We choose an arc $\Lambda_a\subset K_i$ connecting $q_i\in K_i$ to $a$
      so that 
      \begin{gather}
        \Lambda_a \cap 
        \left(
        \left(\bigcup(\mathcal{C}\cup\mathcal{C}_1\cup\mathcal{C}_2)\right) \setminus\{q_i\}
        \right)
        =\varnothing,\\
        \Lambda_a\cap(\Lambda_{a'}\setminus\{q_i\})=\varnothing \ 
         \text{for all }a'\in A'\setminus\{a\}.
      \end{gather}
      Let $\tilde{\Lambda}_a$ be the arc obtained by connecting the arc from $q_1$ to $q_i$ 
      with $\Lambda_a$. We then define $\mathcal{C}_3 \coloneqq \{\tilde{\Lambda}_a \mid a \in A'\}$.
    \end{itemize}
    We define a collection $\tilde{\mathcal{C}}$ of arcs and closed curves in $S$ as
    $\tilde{\mathcal{C}}\coloneqq\mathcal{C}\cup\mathcal{C}_1\cup\mathcal{C}_2\cup\mathcal{C}_3$.
    By the construction, $\bigcup\tilde{\mathcal{C}}$ is connected and Runge in $M$ and 
    every $C\in\tilde{\mathcal{C}}$ contains a nontrivial arc disjoint from 
    all other curves in $\tilde{\mathcal{C}}$. 

    We define the Riemannian metric on $M$ by 
    $g \coloneqq |\theta|^2 = \theta \, \overline{\theta}$. For each $p, q \in S$, 
    let $\Omega_{p, q}$ be the set of all piecewise $C^1$ curves in $S$ from $p$ to $q$ and
    let $d_S(p, q) \coloneqq \inf_{\gamma \in \Omega_{p, q}} L_g(\gamma)$,
    where $L_g(\gamma) = \int_a^b \|\gamma'\|_g  dt = \int_{\gamma} |\theta|$.
    Then, $d_S$ defines a metric on $S$. Indeed, it is clear that $d_S \ge 0$, 
    $d_S(p, q) = d_S(q, p)$, and $d_S(p, q) \le d_S(p, r) + d_S(r, q)$ hold. 
    Let $d_g$ be the Riemannian distance function on $M$ induced by $g$. 
    Since $d_g \le d_S$ on $S$, it follows that $d_S(p, q) > 0$ for $p \ne q$.
    Let us show that the topology on $S$ induced by $d_S$ coincides with the relative 
    topology from $M$. It suffices to show that for each point $p \in S$ and any $\epsilon > 0$,
    there exists $\delta > 0$ such that $B_g(p, \delta) \cap S \subset B_S(p, \epsilon)$, 
    where $B_g(p, \delta) \coloneqq \{q \in M : d_g(p, q) < \delta\}$ and 
    $B_S(p, \epsilon) = \{q \in S : d_S(p, q) < \epsilon\}$.
    Since $S$ is an admissible set, if we choose $\delta > 0$ sufficiently small, 
    then for every $q \in B_g(p, \delta) \cap S$, we can find a curve $\gamma$ in $S$ 
    from $p$ to $q$ such that $L_g(\gamma) < \epsilon$. 
    Therefore, we have $q \in B_S(p, \epsilon)$.
    Now that we have shown that $(S, d_S)$ is a compact metric space, 
    $\max_{p \in S} d_S(p, q_1)$ is positive. Let $L \coloneqq \max_{p \in S} d_S(p, q_1) + 1$.

    Let $\Prd$ be a period map associated to $\tilde{\mathcal{C}}$.  
    \cref{prop:per.prerv.full.approx.} implies that there exists a full map 
    $\tilde{\phi} = (\tilde{\phi}^0, \tilde{\phi}^1, \tilde{\phi}^2) \in \mathcal{O}(M,\B)$ 
    which satisfies 
    $\|\phi - \phi \|_S < \epsilon/L$, $\Prd(\tilde{\phi}) = \Prd(\phi)$, 
    $\tilde{\phi} = \phi$ on $A$, and 
    $\tilde{\phi}$ agrees with $\phi$ to order $\max\{s, \ 3\}$ at every point of $A\cap\Int(S)$. 
    We define $\tilde{f} : M \to \L$ by
    \begin{equation}
       \tilde{f}(p) \coloneqq f(q_1) + \Re\int_{q_1}^p \tilde{\phi} \, \theta.
    \end{equation}
    Since $\Prd(\tilde{\phi}) = \Prd(\phi)$, $\tilde{f}$ is well-defined.
    Furthermore, \cref{lem:full.then.not.identically.zero} implies 
    $-|\tilde{\phi}^0|^2 + |\tilde{\phi}^1|^2 + |\tilde{\phi}^2|^2$ does not vanish identically.
    Thus $\tilde{f}$ is a well-defined full maxface on $M$. 
    Fix an arbitrary point $p \in S$. Since $d_S$ is the infimum of the lengths of curves in $S$, 
    there exists a curve $\gamma_p \in \Omega_{q_1, p}$ such that 
    $L_g(\gamma_p) < d_S(p_0, p) + 1 = L$. Hence, we have
    \begin{equation}
      \|\tilde{f}(p) - f(p)\| 
      = \left\|\Re\int_{\gamma_p} (\tilde{\phi} - \phi )\theta \right\| 
      \le \|\tilde{\phi} - \phi \|_S \int_{\gamma_p}|\theta| 
      < L \|\tilde{\phi} - \phi \|_S.
    \end{equation}
    That is, $\|\tilde{f} - f\|_S \le L \|\tilde{\phi} - \phi \|_S < \epsilon$ holds.
    Moreover, we have the following for all $C \in \mathcal{C}$:
    \begin{equation}
      \Flux_{\tilde{f}}([C])
      = \Im\int_C \tilde{\phi} \, \theta
      = \Im\int_C \phi \, \theta
      = \Flux_f^{\mathcal{C}}([C]).
    \end{equation}
    Thus, $\Flux_{\tilde{f}} = \Flux_{f}^{\mathcal{C}}$ holds. Let $p\in A$. 
    Since $\bigcup\tilde{\mathcal{C}}$ is connected, there exist
    $C_1, \dots, C_N\in\tilde{\mathcal{C}}$ such that $\bigcup_{i=1}^NC_i$ is 
    a curve from $q_1$ to $p$. Hence,
    \begin{equation}
      \tilde{f}(p)
        = f(q_1) + \sum_{i=1}^N \Re\int_{C_i} \tilde{\phi} \, \theta 
        = f(q_1) + \sum_{i=1}^N \Re\int_{C_i} \phi \, \theta 
        = f(p)
    \end{equation}
    holds and we get $\tilde{f} = f$ on $A$.
    For each point $p \in A \cap \Int(S)$, taking a chart $(U, z = u + \I \, v)$ around $p$, 
    we have
    \begin{equation}
      \frac{\partial}{\partial u}(\tilde{f} - f)
      = \Re\left((\tilde{\phi} - \phi) \cdot 
         \theta\left(\frac{\partial}{\partial z}\right)\right),\quad
      \frac{\partial}{\partial v}(\tilde{f} - f)
      = - \Im\left((\tilde{\phi} - \phi) \cdot
           \theta\left(\frac{\partial}{\partial z}\right)\right).
    \end{equation}
    Since $\tilde{\phi}-\phi$ vanishes on $A \cap \Int(S)$ to order $s$, 
    so does $\tilde{f} - f$.

    Let $f|_{\Int(S)}$ be $\mathcal{A}$-equivalent to a cuspidal edge at $p\in\Sigma$.
    Around $p$, the Weierstrass data $(g, \omega)$ of $f|_{\Int(S)}$ can be expressed as 
    \begin{equation}\label{eq.EW.kakikae}
      g=-\frac{\phi^0}{\phi^1-\I\phi^2},\qquad \omega=\frac{1}{2}(\phi^1-\I\phi^2)\theta.
    \end{equation}
    Since $\tilde{\phi}$ agrees with $\phi$ at $p$ to order at least $3$, 
    we have $\tilde{\phi}^{(k)}(p) = \phi^{(k)}(p)$ ($k = 0, 1$).
    Rewriting the Weierstrass data $(\tilde{g}, \tilde{\omega})$ of $\tilde{f}$ similarly to 
    \eqref{eq.EW.kakikae} yields 
    $\tilde{g}(p) = g(p)$,  
    $d\tilde{g}_p = dg_p$, and 
    $\omega_p = \tilde{\omega}_p$.
    Therefore, we get 
    \begin{equation}
      \Re\left(\frac{d\tilde{g}_p}{\tilde{g}(p)^2\tilde\omega_p}\right)
       = \Re\left(\frac{dg_p}{g(p)^2\omega_p}\right)\ne0,
      \quad
      \Im\left(\frac{d\tilde{g}_p}{\tilde{g}(p)^2\tilde\omega_p}\right)
       = \Im\left(\frac{dg_p}{g(p)^2\omega_p}\right)\ne0.
    \end{equation}
    \cref{fact:hanntei} (i) implies that $\tilde f$ is $\mathcal A$-equivalent to a cuspidal edge.
    When $f$ is $\mathcal{A}$-equivalent to a swallowtail or a cuspidal cross cap at $p$,
    we can apply the same argument as in the case of the cuspidal edge, 
    using $\tilde{\phi}^{(k)}(p) = \phi^{(k)}(p)$ for $k = 0, 1, 2$ 
    together with \cref{fact:hanntei} (ii) and (iii).
    Furthermore, when $f$ is $\mathcal{A}$-equivalent to a cuspidal butterfly or 
    a cuspidal $S_1^-$ singularity at $p$, we obtain the conclusion using 
    $\tilde{\phi}^{(k)}(p) = \phi^{(k)}(p)$ for $k = 0, 1, 2, 3$, 
    together with \cref{fact:hanntei} (iv) and (v).
  \end{proof}
  
  The following lemma is required for the proof of \cref{thm:main}.
  
  \begin{lemma} \label{prop:existence.of.Morse}
   Let $M$ be an open Riemann surface, 
   let $K \subset M$ be a compact Runge set, 
   let $U \subset M$ be an open neighborhood of $K$, and
   let $\Lambda \subset M$ be a closed discrete set.
   Then, there exists a strongly subharmonic Morse exhaustion function $\rho : M \to \R$ 
   satisfying the following:
   \begin{enumerate}
     \item 
     $K \subset \{\rho < 0\} \subset \{\rho \le 0\} \subset U$,

     \item 
     $\Lambda \cap \mathrm{Crit}(\rho) = \varnothing$,
     
     \item 
     $\rho|_{\mathrm{Crit}(\rho) \cup \Lambda}$ is injective, and

     \item 
     $0 \notin \rho(\mathrm{Crit}(\rho) \cup \Lambda)$,
   \end{enumerate}
    where $\mathrm{Crit}(\rho)$ is the set of its critical points of $\rho$.
 \end{lemma}

 \begin{proof}
   \emph{Step 1: There exists a strongly subharmonic Morse exhaustion function
   $\rho_1 : M \to \R$ satisfying 
   $K \subset \{\rho_1 < 0\} \subset \{\rho_1 \le 0\} \subset U$ and 
   $\mathrm{Crit}(\rho_1) \cap \Lambda = \varnothing$.}
   By \cref{fact:str.subham.Morse.exh.fucn}, 
   there exists a strongly subharmonic Morse exhaustion function
   $\rho_0 : M \to \R$ such that $K \subset \{\rho_0 < 0\} \subset \{\rho_0 \le 0\} \subset U$.
   If $\mathrm{Crit}(\rho_0) \cap \Lambda = \varnothing$, then we set $\rho_1 \coloneqq \rho_0$.
   Assume $I \coloneqq \mathrm{Crit}(\rho_0) \cap \Lambda \ne \varnothing$.
   Since $I$ is a closed discrete set, it is at most countable. 
   We write $I = \{p_k : k = 1, 2, \dots \}$.
   For each $p_k \in I$, there is a chart $(W_k, w_k = u_k + \I \, v_k)$ around $p_k$ such that
   $W_k \cap \mathrm{Crit}(\rho_0) = \{p_k\}$,
   $W_k \cap W_l = \varnothing$ $(k \ne l)$, $w_k(p_k) = 0$, and $w_k(W_k) = D_2(0)$.
   We set $\overline{D}_k \coloneqq w_k^{-1}(\overline{D_1(0)})$. 
   Let $F_k : \overline{D}_k \to \R^2$ be a map defined by $F_k \coloneqq \grad \rho_0$, where
   $\grad \rho_0 \coloneqq (\partial \rho_0/\partial u_k, \partial \rho_0/ \partial v_k)$.
   Hereafter, whenever it is clear from the context that the support of the function $\psi$ is 
   contained in a chart $(V, z = u + \I \, v)$, 
   we write $\grad \psi \coloneqq (\partial \psi/ \partial u, \partial \psi / \partial v)$.
   Since $p_k$ is a nondegenerate critical point of $\rho_0$, it holds that $F_k(p_k) = 0$ and 
   the Jacobian matrix of $F_k$ at $p_k$ is invertible.
   Hence, there are a constant $r_k \in (0, 1)$ and an open neighborhood $\hat{W}_k$ of 
   the origin $F_k(p_k)$ 
   such that the restriction $F_k : w_k^{-1}(D_{r_k}(0)) \to \hat{W}_k$ is a diffeomorphism.

   Let $\hat{D}_k \coloneqq w_k^{-1}(D_{r_k}(0))$ and 
   $A_k \coloneqq \overline{D}_k \setminus \hat{D}_k$.
   Let us show that $\min_{A_k}\|F_k\| > 0$ and $\inf_{M \setminus U} \rho_0 > 0$.
   The function $\rho_0$ has no critical point on $A_k$. Thus, $F_k \ne 0$ on $A_k$. 
   This implies $\min_{A_k}\|F_k\| > 0$. To show $\inf_{M \setminus U} \rho_0 > 0$, we assume
   $\inf_{M \setminus U} \rho_0 = 0$. Then there is a sequence $\{p_n\}_n \subset M \setminus U$
   such that $\rho_0(p_n) \rightarrow 0$. We may assume $\{p_n\}_n \subset \{\rho_0 \le 1\}$.
   Since $\{\rho_0 \le 1\}$ is compact, $\{\rho_0 \le 1\} \cap (M \setminus U)$ is also compact.
   Therefore, $\{p_n\}_n \subset \{\rho_0 \le 1\} \cap (M \setminus U)$ has a subsequence which 
   converges to some point $p_* \in M \setminus U$. The continuity of $\rho_0$ gives 
   $\rho_0(p_*) = 0$, i.e., $p_* \in \{\rho_0 \le 0\} \subset U$. 
   This is a contradiction. Now, we have $\inf_{M \setminus U} \rho_0 > 0$.
   
   We take a smooth function $\chi_k : M \to \R$ 
   such that $\chi_k \equiv 1$ on $\hat{D}_k$, $0 \le \chi_k \le 1$ on $A_k$, and
   $\chi_k \equiv 0$ on $M \setminus \overline{D}_k$. 
   Fix $a_k \coloneqq (a_{k,1}, a_{k,2}) \in D_{r_k}(0) \setminus \{0\}$ satisfying
   \begin{equation} \label{eq:how.to.take.ak}
     \begin{gathered}
       \|a_k\| < \min\bigg\{
       \dfrac{\min_{\overline{D}_k} \Delta \rho_0}
           {\max_{\overline{D}_k}(|\Delta \chi_k| + 2 \, \|\grad \chi_k \|) + 1}, 
       \dfrac{\min_{A_k}\|F_k\|}{\max_{A_k}(\chi_k + |w_k| \, \|\grad \chi_k\|)}, \\
       \hspace{190pt}
        2^{-k-1} \, \min\{\inf_{M \setminus U}\rho_0, - \max_K \rho_0\},
       \bigg\}, \\
         a_k \in F_k((\mathrm{Hess} \, \rho_0)^{-1}(\mathrm{GL}(2, \R))),
     \end{gathered}
   \end{equation}
   where $\Delta = \partial^2 / \partial u_k^2 + \partial^2 / \partial v_k^2$ and
   $\mathrm{Hess} \, \rho_0 : W_k \to \M(2, \R)$ is the Hessian with respect to 
   the coordinates $u_k + \I v_k$. 
   We note that $\max_K \rho_0$ is negative and $\Delta \rho_0 > 0$ holds on $\overline{D}_k$ 
   since $\rho_0$ is strongly subharmonic. Let $\psi_k : \overline{D}_k \to \R$ and 
   $\rho_1 : M \to \R$ be functions defined by 
   \begin{align}
     \psi_k \coloneqq a_{k,1} \, u_k + a_{k,2} \, v_k,  \qquad 
     \rho_1 \coloneqq \rho_0 - \sum_{p_k \in I} \chi_k \, \psi_k.
   \end{align}
   
   First, we prove that $\rho_1$ is strongly subharmonic. 
   It suffices to show that $\Delta \rho_1 > 0$ on each $\overline{D}_k$, 
   which follows from \eqref{eq:how.to.take.ak} and the following inequality on $\overline{D}_k$:
   \begin{align}
     \Delta \rho_1 
     &\ge \Delta \rho_0 - |\Delta(\chi_k \, \psi_k)|
     \ge \min_{\overline{D}_k} \Delta \rho_0 
           - (|\Delta \chi_k| \, \|a_k\| + 2 \, \|\grad \chi_k\| \, \|a_k\|) \\
     &\ge \min_{\overline{D}_k} \Delta \rho_0
           - \|a_k\| \, 
           \left\{
           \max_{\overline{D}_k}(|\Delta \chi_k| + 2 \, \|\grad \chi_k\|) + 1
           \right\}
     > 0.
   \end{align}

   Next, let us show that $\rho_1$ is a Morse function. 
   Since $\rho_1 = \rho_0$ holds on $M \setminus (\bigcup_k \overline{D}_k)$, 
   $\rho_1$ has only nondegenerate critical points on 
   $M \setminus (\bigcup_k \overline{D}_k)$. 
   For a fixed $p_k \in I$ and for all $p \in A_k$, we obtain
   \begin{align}
     \|\chi_k(p) \, a_k + \psi_k(p) \, \grad \chi_k (p)\|
     &\le \chi_k(p) \, \|a_k\| + \|a_k\| \, \|(u_k(p), v_k(p))\| \, \|\grad \chi_k (p)\| \\
     &\le \|a_k\| \, \max_{A_k}(\chi_k + |w_k| \, \|\grad \chi_k \|).
   \end{align}
   Therefore, by \eqref{eq:how.to.take.ak}, the following holds on $A_k$:
   \begin{align}
     \|\grad \rho_1 \|
     &\ge \|\grad \rho_0\| - \|\chi_k \, a_k + \psi_k \, \grad \chi_k \| \\
     &\ge \min_{A_k} \|F_k\| - \|a_k\| \, \max_{A_k} (\chi_k + |w_k| \, \|\grad \chi_k\|) > 0.
   \end{align}
   This means that $\rho_1$ has no critical points on $A_k$. By the definition of $\psi_k$, 
   $\grad \rho_1 = \grad \rho_0 - a_k$ holds on $\hat{D}_k$. 
   Thus, $\grad \rho_1 = 0$ is equivalent to $F_k = a_k$. 
   Hence, $\rho_1$ has a unique critical point $F_k^{-1}(a_k)$ in $\hat{D}_k$.
   This implies that $\mathrm{Crit}(\rho_1) \cap \Lambda = \varnothing$.
   Furthermore, it is clear that 
   $\mathrm{Hess} \, \rho_1 = \mathrm{Hess} \, \rho_0$ holds on $\hat{D}_k$, and 
   by \eqref{eq:how.to.take.ak}, we obtain:
   \begin{equation}
     \mathrm{Hess} \, \rho_1 (F_k^{-1}(a_k)) 
     = \mathrm{Hess} \, \rho_0(F_k^{-1}(a_k)) 
     \in \mathrm{GL}(2, \mathbb{R}).
   \end{equation}
   This shows that $\rho_1$ is a Morse function.

   Let us show that $\rho_1$ is an exhaustion function. 
   There exists a constant $L > 0$ such that $\sum_k |\chi_k \, \psi_k| < L$ holds on $M$.
   Thus, for each $c \in \R$ and $p \in \{\rho_1 \le c\}$, we obtain:
   \begin{equation}
     \rho_0(p) 
     = \rho_1(p) + \sum_k \chi_k \, \psi_k 
     \le c + L.
   \end{equation}
   Hence, $\{\rho_1 \le c\} \subset \{\rho_0 \le c + L \}$ holds. 
   Since $\{\rho_1 \le c\}$ is a closed subset of the compact set $\{\rho_0 \le c + L \}$, 
   it is compact.
   
   Finally, we prove that $K \subset \{\rho_1 < 0\} \subset \{\rho_1 \le 0\} \subset U$.
   Noting $\|\psi_k\| \le |w_k| \, \|a_k\| \le \|a_k\|$ and \eqref{eq:how.to.take.ak}, 
   we obtain the following for all $p \in K$:
   \begin{align}
     \rho_1(p) 
     &\le \max_K \rho_0 + \left|\sum_k \chi_k \, \psi_k \right|  
     \le \max_K \rho_0 + \sum_k \|a_k\| \\ 
     &\le \max_K \rho_0 + \sum_k \frac{-\max_K \rho_0}{2^{k+1}}
     < \frac{1}{2} \max_K \rho_0 < 0
   \end{align}
   Thus, $K \subset \{\rho_1 < 0\}$. Let $p \in M \setminus U$. Then, we have:
   \begin{equation}
     \rho_1(p)
     \ge \inf_{M \setminus U} \rho_0 - \sum_k \|a_k\|
     \ge \inf_{M \setminus U} \rho_0 - \sum_k \frac{\inf_{M \setminus U} \rho_0}{2^{k + 1}}
     \ge \frac{1}{2} \inf_{M \setminus U} \rho_0 > 0.
   \end{equation}
   Hence, it follows that $M \setminus U \subset \{\rho_1 > 0\}$, i.e., 
   $\{\rho_1 \le 0\} \subset U$.

   \emph{Step 2: Completion of the proof.}
   We set $X = \{q_i : i = 1, 2, \dots\} \coloneqq \mathrm{Crit}(\rho_1) \cup \Lambda$.
   For each $q_i \in X$, there exists a chart $(U_i, z_i)$ around $q_i$ such that 
   $U_i \cap U_j = \varnothing$ $(i \ne j)$, $z_i(q_i) = 0$, and $z_i(U_i) = D_2(0)$.
   Let $\overline{V}_i \coloneqq z_i^{-1}(\overline{D_1(0)})$ and let $\eta_i : M \to \R$
   be a smooth function such that $\eta_i \equiv 1$ on $z_i^{-1}(D_{1/2}(0))$, 
   $0 \le \eta_i \le 1$ on $\overline{V}_i \setminus z_i^{-1}(D_{1/2}(0))$, and 
   $\eta_i \equiv 0$ on $M \setminus \overline{V}_i$. 
   If we set $B_i \coloneqq \overline{V}_i \setminus z_i^{-1}(D_{1/2}(0))$, then 
   the same argument as in Step 1 implies $\min_{B_i} \|\grad \rho_1\| > 0$ and
   $\inf_{M \setminus U} \rho_1 > 0$. Take $\delta_i > 0$ satisfying
   \begin{equation} \label{eq:how.to.take.deltai}
     \begin{gathered}
       \delta_i < \min\bigg\{
       \dfrac{\min_{B_i}\|\grad \rho_1\|}{\max_{B_i}(\|\grad \eta_i\|) + 1}, \ \  
        2^{-i-1} \, \min\{\inf_{M \setminus U}\rho_1, \ - \max_K \rho_1\}, \\ 
        \hspace{240pt}
        \dfrac{\min_{\overline{V}_i} \Delta \rho_1}
           {\max_{\overline{V}_i}(|\Delta \eta_i|) + 1}
       \bigg\},
     \end{gathered}
   \end{equation}
   Let $I_i \coloneqq (0, \delta_i) \setminus \rho_1(X)$ and
   let $\epsilon_1, \epsilon_2, \dots$ be positive numbers such that
   \begin{gather}
     \epsilon_1 \in I_1, \quad 
     \epsilon_2 \in I_2 \setminus \{-\rho_1(q_1) + \epsilon_1 + \rho_1(q_2)\}, \\
     \epsilon_3 \in I_3 \setminus \{-\rho_1(q_1) + \epsilon_1 + \rho_1(q_3), \, 
                                    -\rho_1(q_2) + \epsilon_2 + \rho_1(q_3)\}, \\
     \epsilon_4 \in I_4 \setminus \{-\rho_1(q_1) + \epsilon_1 + \rho_1(q_4), \,
                                    -\rho_1(q_2) + \epsilon_2 + \rho_1(q_4), \,
                                    -\rho_1(q_3) + \epsilon_3 + \rho_1(q_4)\} \\
     \vdots
   \end{gather}
   We define $\rho : M \to \R$ by
   \begin{equation}
     \rho = \rho_1 - \sum_{q_i \in X} \epsilon_i \, \eta_i.
   \end{equation}
   Noting \eqref{eq:how.to.take.deltai}, we can show that 
   $\rho$ is a strongly subharmonic Morse exhaustion function
   satisfying $K \subset \{\rho < 0\} \subset \{\rho \le 0\} \subset U$
   by the same argument as in Step 1.

   We now show that $\rho|_X$ is injective. 
   Note that $\mathrm{Crit}(\rho) = \mathrm{Crit}(\rho_1)$. If $i < j$, then 
   \begin{equation}
     \epsilon_j \in I_j \setminus \{-\rho_1(q_1) + \epsilon_1 + \rho_1(q_j), ...,
                                    -\rho_1(q_i) + \epsilon_i + \rho_1(q_j), ...,
                                    -\rho_1(q_{j-1}) + \epsilon_{j-1} + \rho_1(q_j)\}.
   \end{equation}
   Thus, $\rho(q_i) = \rho_1(q_i) - \epsilon_i \ne \rho_1(q_j) - \epsilon_j = \rho(q_j)$.
   Furthermore, each $\epsilon_i \in I_i$ satisfies $\epsilon_i \ne \rho_1(q_i)$, i.e.,
   $\rho(q_i) = \rho_1(q_i) - \epsilon_i \ne 0$. 
   Therefore, $\rho(q_i) \ne 0$ for all $q_i \in X$. 
 \end{proof}
 \begin{fact}[{\cite[Lemma 3.5.4]{AFL21}}] \label{fact:connecting.path.in.B}
   Assume that $A$ is an irreducible nondegenerate algebraic subvariety of $\C^n$ 
   $($i.e., $A$ is not contained in any affine hyperplane of $\C^n)$.
   Given a continuous map $f_0 : [0, 1] \to A_{\mathrm{reg}}$ 
   into the regular locus $A_{\mathrm{reg}}$ of $A$, 
   a continuous function $g : [0,1] \to \C \setminus \{0\}$, 
   a vector $v \in \C^n$, and 
   a connected domain $\Omega \subset \C^n$ containing $0$ and $v$,
   there exists a homotopy $f_{\tau} : [0,1] \to A_{\mathrm{reg}}$ $(\tau \in [0,1])$ 
   fixing the endpoints such that 
   the map $f = f_1$ is smooth, 
   $f([0, \epsilon])$ is not contained in any affine complex line in $\C^n$ 
   for sufficiently small $\epsilon > 0$, and
   \begin{equation} \label{eq:integral}
     \int_0^1 f(s) \, g(s) \, ds = v \quad \text{and} \quad
     \int_0^t f(s) \, g(s) \, ds \in \Omega \quad \text{for all $t \in [0, 1]$.}
   \end{equation}
   In particular, any pair of points in $A_{\mathrm{reg}}$ can be connected by 
   a smooth path $f : [0,1] \to A_{\mathrm{reg}}$ satisfying condition \eqref{eq:integral}.
 \end{fact}
 \begin{theorem} \label{thm:main}
   Assume that $M$ is an open Riemann surface, 
   $\theta$ is a nonvanishing holomorphic $1$-form on $M$, 
   $S \subset M$ is a connected Runge admissible set,
   $\Lambda \subset M$ and $\Sigma \subset \Int(S) \cup \Lambda$ are closed discrete subsets, 
   $V \subset M$ is an open neighborhood of $\Lambda$,
   $f : S \cup V \to \L$ is a map such that 
   $(f|_S, \phi|_S \, \theta)$ is a generalized maxface and $f|_V$ is a maxface, 
   where $\phi = 2 \, \partial f / \theta$.

   Given a positive number $\epsilon$, 
   a map $k \colon \Lambda \to \Z_{>0}$, and 
   a group homomorphism $\mathfrak{p} \colon H_1(M, \Z) \to \R^3$ with 
   $\mathfrak{p}|_{H_1(S, \Z)} = \Flux_f^{\mathcal{C}}$
   $($where $\mathcal{C}$ is a homology basis of $S$ 
   obtained by \cref{fact:homology.basis.of.addmissible}$)$,
   there exists a full maxface $\tilde{f} \colon M \to \L$ satisfying the following conditions.
   \begin{enumerate}
     \item
     $\|\tilde{f} - f \|_S \le \epsilon$.
     
     \item
     The difference $\tilde{f} - f$ vanishes to order $k(p)$ at every point $p \in \Lambda$.
     
     \item
     $\Flux_{\tilde{f}} = \mathfrak{p}$ on $H_1(M, \Z)$.

     \item 
     If the maxface $f|_{\Int(S) \cup V}$ is $\mathcal{A}$-equivalent to a cuspidal edge 
     $($resp. swallowtail, cuspidal cross cap, cuspidal butterfly, cuspidal $S_1^-$ singularity$)$ 
     at $p \in \Sigma$, then $\tilde{f}$ is also $\mathcal{A}$-equivalent to a cuspidal edge 
     $($resp. swallowtail, cuspidal cross cap, cuspidal butterfly, cuspidal $S_1^-$ singularity$)$
     at $p$.
   \end{enumerate}
 \end{theorem}

 \begin{proof}
   Let $g$ be a complete Riemannian metric on $M$ and 
   $d_g$ be the Riemannian distance function induced by $g$.
   The function $p \mapsto \mathrm{Inj}(p)$, 
   which assigns the injectivity radius of $\exp_p$ to each point $p$, is continuous on $M$
   since $g$ is complete.
   Thus, $\max_S \mathrm{Inj}$ is positive. 
   By \cref{rem:reg.nbd.of.admissible.set}, 
   there is $r_1 > 0$ such that for each $r \in (0, r_1)$, 
   a regular neighborhood $S_r$ of $S$ obtained by using $d_g$ satisfies 
   $H_1(S, \Z) \xhookrightarrow[]{\cong} H_1(S_r, \Z)$. 
   Fix a positive number 
   $0 < r < \min\{\max_S \mathrm{Inj}, \, r_1, \, d_g(S, \Lambda \setminus S)\}$.
   Then, a regular neighborhood $W_0 \coloneqq S_r$ is connected and satisfies
   $\hml{S} \emb \hml{W_0}$ and $\Lambda \cap W_0 \subset S$.
   The connectedness of $W_0$ follows from 
   $W_0 = \bigcup_{p \in S} D_g(p, r)$, where $D_g(p, r)$ is the geodesic ball of radius $r$
   centered at $p$.
   For each point $p \in \Lambda \cap S$ with $p \notin \Int(S)$, we take a simply connected 
   compact neighborhood $\overline{D}_p \subset W_0$ such that 
   \begin{equation}
     \hat{S} \coloneqq S \cup 
     \bigcup_{\substack{p \in \Lambda \cap S \\ p \notin \Int(S)}} \overline{D}_p
   \end{equation}
   is a connected Runge admissible set in $W_0$ satisfying
   $\hml{S} \emb \hml{\hat{S}} \emb \hml{W_0}$. 
   Furthermore, we define  $\hat{k} : \Lambda \cup \Sigma \to \Z$ by
   \begin{equation}
       \hat{k}(p) = 
       \begin{cases}
         3 & (\text{$p \in \Int(S) \cap \Sigma$}), \\
         \max\{3, \ k(p)\} & (\text{$p \in \Lambda \cap \Sigma$}), \\
         k(p) & (\text{otherwise}). \\
       \end{cases}
    \end{equation}
   Applying \cref{prop:weak.approx.} to 
   $(f|_{\hat{S}}, \phi|_{\hat{S}}\,\theta) \in \GMF(\hat{S})$ and $\epsilon_0 \in (0, \epsilon/2)$, 
   we obtain a full maxface $f_0 \in \MF(W_0)$ and 
   a full map $\phi_0 \in \mathcal{O}(W_0, \B)$ such that
   \begin{itemize}
     \item 
     $\|f_0 - f\|_{\hat{S}} < \epsilon_0, \quad \|\phi_0 - \phi\|_{\hat{S}} < \epsilon_0$,
     
     \item
     $f_0(p) = f(p_0) + \Re\displaystyle\int_{p_0}^p \phi_0 \, \theta$ 
     for a fixed point $p_0 \in S$,
     
     \item
     $\Flux_{f_0} = \Flux_{f}^{\mathcal{C}} = \mathfrak{p}$ on 
     $H_1(S, \Z) \xhookrightarrow[]{\cong} H_1(\hat{S}, \Z) \xhookrightarrow[]{\cong} H_1(W_0, \Z)$,
     and
     
     \item
     The differences $f_0 - f$ and $\phi_0 - \phi$ vanish at every point of 
     $(\Lambda \cup \Sigma) \cap \Int(\hat{S}) = (\Lambda \cup \Sigma) \cap S$ to order 
     $\max \{\hat{k}(p) \colon p \in (\Lambda \cup \Sigma) \cap S\}$.
   \end{itemize}
   
   Let $\rho : M \to \R$ be a strongly subharmonic Morse exhaustion function obtained by 
   applying \cref{prop:existence.of.Morse} to the compact set $\hat{S}$, 
   its open neighborhood $W_0$,
   and the closed discrete set $\Lambda$.
   Since $\rho$ is an exhaustion function, 
   $\rho(\mathrm{Crit}(\rho) \cup \Lambda)$ is a closed discrete subset of 
   $\R$ not containing $0$. 
   Thus, there exists a sequence $\{c_i\}_{i = 0}^{\infty}$ of 
   regular values of $\rho$ satisfying the following conditions.
   \begin{itemize}
     
     \item
     $ 0 = c_0 < c_1 < \cdots < c_i < c_{i + 1} < \cdots$, 
     $\lim_{i \rightarrow \infty}c_i = \infty.$
     
     \item 
     $A_i \coloneqq \{c_{i-1} < \rho < c_i\}$ $(i = 1, 2,\dots)$ contains 
     at most one critical point of $\rho$ or at most one point of $\Lambda$, but not both.
     
     \item
     There is a sequence $\{b_i\}_{i = 1}^{\infty} \subset \R$ with $c_i + b_i < c_{i + 1}$
     such that
     the set $\{c_i < \rho < c_i + b_i \}$ contains  neither a critical point of $\rho$ nor 
     a point of $\Lambda$.
   \end{itemize}
   We note that the compact domains $M_i \coloneqq \{\rho \le c_i\}$ $(i = 0, 1, \dots)$
   with smooth boundaries have finitely many connected components, 
   and thus they are admissible sets.
   Moreover, the open neighborhoods 
   $W_i \coloneqq \{\rho < c_i + b_i\}$ $(i = 1, 2, \dots)$ of $M_i$ 
   satisfy $H_1(M_i, \Z) \xhookrightarrow[]{\cong} H_1(W_i, \Z)$.
   Let $n_i$ denote the number of connected components of both $M_i$ and $W_i$, 
   and let us denote these components by $M_{i,l}$ and $W_{i,l}$ ($l = 1, \dots, n_i$), respectively,
   where we assume that $M_{i,1}$ and $W_{i,1}$ are the components containing $S$.
   Let $\{\alpha_i\}_{i = 0}^{\infty} \subset \R$ be a sequence defined by 
   $\alpha_i \coloneqq \exp(-2^{-i-1} \log 2)$ $(i = 0, 1, \dots)$. Clearly, the following hold.
   \begin{equation}
     \frac{1}{2} < \alpha_0 < \alpha_1 < \cdots < \alpha_i < \alpha_{i + 1}< \cdots < 1, \qquad 
     \prod_{i = 0}^{\infty} \alpha_i = \frac{1}{2}.
   \end{equation}
   
   Now, we construct $f_i \in \MF(W_i)$,
   $\phi_i \in \mathcal{O}(W_i, \B)$, and
   $\epsilon_i > 0$ for $i = 1, 2, \dots$ inductively,
   satisfying the following conditions:
   \begin{itemize}
     \item[$(a_i)$] 
     $\epsilon_i < \min\{\epsilon/2^{i + 1}, \  (1 - \alpha_{i}) \, 
     \min_{M_{i - 1}}\|\phi_{i-1}\|\}$.
     
     \item[$(b_i)$]
     $\|f_i - f_{i - 1}\|_{M_{i-1}} < \epsilon_i$ and 
     $\|\phi_i - \phi_{i - 1}\|_{M_{i-1}} < \epsilon_i$.
     Furthermore, the restriction of $\phi_i$ to each connected component of $W_i$ is a full map.
        
     \item[$(c_i)$]
     $2 \partial f_i = \phi_i \, \theta$.
     
     \item[$(d_i)$]
     $\Flux_{f_{i}} = \mathfrak{p}$ on $H_1(M_i, \Z) \xhookrightarrow[]{\cong} H_1(W_i, \Z)$.

     \item[$(e_i)$]
     The differences $f_i - f$ and $\phi_i - \phi$ vanish at every point  
     $p \in (\Lambda \cup \Sigma) \cap M_i$ to order $\hat{k}(p)$.
   \end{itemize}
   We note that $f_0 \in \MF(W_0)$ and $\phi_{0} \in \mathcal{O}(W_0, \B)$ satisfy the conditions
   $(c_0)$--$(e_0)$. 
   For $i \ge 1$, assuming that $f_{i-1}$ and $\phi_{i-1}$ satisfy conditions 
   $(c_{i-1})$--$(e_{i-1})$, we construct $f_i$ and $\phi_i$ satisfying $(b_i)$--$(e_i)$.

   \emph{Case 1: The set $A_i$ contains neither a critical point of $\rho$ nor 
   a point of $\Lambda$.}
   In this case, $n_{i-1} = n_i$ holds, and each connected component of $W_i$ contains
   exactly one connected component of $M_{i-1}$. By reindexing if necessary, 
   we may assume that $M_{i-1, l} \subset W_{i,l}$ ($l = 1, \dots, n_{i-1}=n_i$). 
   For $l \in \{1, \dots, n_i\}$, 
   since $\hml{M_{i-1,l}} \emb \hml{W_{i,l}}$, we get $f_i \in \MF(W_{i,l})$ and 
   $\phi_i \in \mathcal{O}(W_{i, l}, \B)$ by applying \cref{prop:weak.approx.} to 
   the generalized maxface
   $(f_{i-1}|_{M_{i-1,l}}, \phi_{i-1}|_{M_{i-1,l}} \, \theta) \in \GMF(M_{i-1, l})$,
   the finite set $(\Lambda \cup \Sigma) \cap M_{i, l} = (\Lambda \cup \Sigma) \cap M_{i-1, l}$,
   and the positive integer $s = \max\{\hat{k}(p) : p \in (\Lambda \cup \Sigma) \cap M_{i, l}\}$.
   We can regard $f_i$ and $\phi_i$ as maps on $W_i$ that satisfy the conditions $(b_i)$--$(e_i)$, 
   where we set $p_{i, l} =  p_{i-1, l}$.
   Note that when applying \cref{prop:weak.approx.}, 
   it is not necessary to fix a homology basis of $M_{i-1, l}$. 
   This is because the flux of the generalized maxface 
   $(f_{i-1}|_{M_{i-1,l}}, \phi_{i-1}|_{M_{i-1, l}} \, \theta)$, 
   obtained as the restriction of a maxface defined on 
   an open neighborhood $W_{i-1,l}$, is independent of the homology basis of $M_{i-1, l}$.

   \emph{Case 2: The set $A_i$ contains a point $a \in \Lambda$.}
   As in Case 1, we may assume that $M_{i-1, l} \subset W_{i,l}$ ($l = 1, \dots, n_{i-1}=n_i$). 
   Fix the connected component $W_{i,l}$ containing $a \in \Lambda$, 
   and take a neighborhood $\overline{D}_a \subset V \cap W_{i, l}$ of $a$ 
   which is diffeomorphic to a closed disk.
   Furthermore, take an arc 
   $E \subset (W_{i, l} \setminus (\overline{D}_a \cup M_{i - 1, l})) \cup \{q, q'\}$ 
   whose endpoints are $q \in \partial M_{i-1}$ and $q' \in \partial \overline{D}_a$.
   Let $\gamma : [0, 1] \to E$ be a parameterization of $E$ with $\gamma(0) = q$ and $\gamma(1) = q'$.
   By \cref{fact:connecting.path.in.B}, there exists a smooth map $\hat{\phi} : [0, 1] \to \B$ such that
   $\hat{\phi}(0) = \phi_{i - 1}(q)$, $\hat{\phi}(1) = \phi(q')$, and
   \begin{equation}
     \int_E \hat{\phi} \, \theta = \int_0^1 \hat{\phi} \, \gamma^*\theta = f(q') - f_{i - 1}(q).
   \end{equation}
   We set $\hat{M}_{i - 1, l} = M_{i - 1,l} \cup E \cup \overline{D}_a$ 
   and define two maps $\hat{\phi}_{i - 1} : \hat{M}_{i - 1,l} \to \B$ and 
   $\hat{f}_{i-1} : \hat{M}_{i-1, l} \to \L$ by
   \begin{equation}
     \hat{\phi}_{i - 1}(p) \coloneqq 
     \begin{cases}
       \phi_{i - 1}(p) & (p \in M_{i -1,l}), \\
       \hat{\phi} \circ \gamma^{-1} (p) & (p \in E), \\
       \phi(p) & (p \in \overline{D}_a),
     \end{cases} \quad
     \hat{f}_{i-1}(p) 
     \coloneqq f_{i-1}(p_{i-1}) + \Re \int_{p_{i-1}}^p \hat{\phi}_{i-1} \, \theta,
   \end{equation}
   where $p_{i - 1} \in M_{i-1, l}$ is some fixed point.
   The map $\hat{f}_{i-1}$ is well-defined. Indeed, let $C \subset \hat{M}_{i-1, l}$ 
   be an arbitrary closed curve. If $C \subset \overline{D}_a$, 
   we have $\Re \int_C \hat{\phi}_{i - 1} \, \theta = 0$ because 
   $\hat{\phi}_{i - 1}|_{\overline{D}_a} \in \mathcal{O}(\overline{D}_a)$ and
   $\overline{D}_a$ is simply connected. If $C \subset M_{i-1, l}$, the same holds since 
   $\hat{\phi}_{i - 1}|_{M_{i-1, l}} = \phi_{i-1} = 2 \partial f_{i-1} / \theta$. 
   Finally, suppose $C \cap \overline{D}_a \neq \varnothing$ and $C \cap M \neq \varnothing$. 
   Then there exist closed curves $C_1 \subset M$ and $C_2 \subset \overline{D}_a$ 
   passing through $q$ and $q'$, respectively, such that
   \begin{equation}
     \Re \int_C \hat{\phi}_{i-1} \, \theta 
     = \Re \int_{C_1} \hat{\phi}_{i-1} \, \theta 
       + \Re \int_E \hat{\phi}_{i-1} \, \theta 
         - \Re \int_E \hat{\phi}_{i-1} \, \theta 
           + \Re \int_{C_2} \hat{\phi}_{i-1} \, \theta.
   \end{equation}
   Since we have already seen that 
   $\Re \int_{C_1} \hat{\phi}_{i-1} \, \theta = \Re \int_{C_2} \hat{\phi}_{i-1} \, \theta = 0$, 
   we obtain $\Re \int_C \hat{\phi}_{i-1} \, \theta = 0$. 
   By applying \cref{prop:weak.approx.} to 
   $(\hat{f}_{i-1}, \hat{\phi}_{i-1} \, \theta) \in \GMF(\hat{M}_{i-1})$, 
   we obtain $f_i \in \MF(W_{i,l})$ and $\phi_i \in \mathcal{O}(W_{i,l}, \B)$ such that 
   the differences $f_i - \hat{f}_{i-1}$ and $\phi_i - \hat{\phi}_{i-1}$ vanish on 
   $(\Lambda \cup \Sigma) \cap M_{i, l} = ((\Lambda \cap \Sigma) \cap M_{i-1, l}) \cup \{a\}$
   to order $\max\{\hat{k}(p) : p \in (\Lambda \cup \Sigma) \cap M_{i, l}\}$.
   For the connected components other than $W_{i,l}$, 
   we can construct $f_i$ and $\phi_i$ by applying the argument of Case 1. 
   By regarding them as maps defined on $W_i$, 
   we see that they satisfy conditions $(b_i)$ through $(e_i)$.

   \emph{Case 3: The set $A_i$ contains a critical point $a \in \mathrm{Crit}(\rho)$.}
   In this case, further case distinctions are required depending on 
   the Morse index of $a \in \mathrm{Crit}(\rho)$.
   Note that since $c_{i-1}$ and $c_i$ are regular values of $\rho$, 
   both $M_{i-1}$ and $M_i$ are manifolds with boundary.
   If the Morse index is $0$, then $M_i$ is diffeomorphic to the disjoint union of $M_{i-1}$ 
   and a closed disk containing $a$.
   If the Morse index is $1$, then $M_i$ is diffeomorphic to the manifold obtained by attaching 
   $I = [0,1]^2$ to $M_{i-1}$ via a diffeomorphism 
   $\psi : \{\pm 1\} \times [0,1] \to \partial M_{i-1}$ 
   onto its image. The image of $\psi(\{\pm 1\} \times [0,1])$ consists of 
   two connected components, and the situation differs depending on 
   whether they are contained in the same connected component of $M_{i-1}$ or not 
   (see \cite[pp. 21--22]{AFL21}).
   Since $\rho$ is strongly subharmonic, the Morse index of $a \in \mathrm{Crit}(\rho)$ cannot be $2$.
   Therefore, Case 3 is divided into the following three subcases:
   \begin{itemize}
     \item
     \emph{Subcase 3a: The Morse index of $a \in \mathrm{Crit}(\rho)$ is $0$.}

     \item 
     \emph{Subcase 3b: The Morse index of $a \in \mathrm{Crit}(\rho)$ is $1$, 
   and the two connected components of $\psi(\{\pm 1\} \times [0, 1])$ are contained in 
   the same connected component of $M_{i-1}$.}

     \item 
     \emph{Subcase 3c: The Morse index of $a \in \mathrm{Crit}(\rho)$ is $1$, 
   and the two connected components of $\psi(\{\pm 1\} \times [0, 1])$ are contained in 
   different connected components of $M_{i-1}$.}
   \end{itemize}

   \emph{Subcase 3a:}
   The sets $M_i$ and $W_i$ have exactly one more connected component than $M_{i-1}$ and $W_{i-1}$,
   respectively. That is, $n_{i-1} + 1 = n_i$. 
   As in Case 1, we may assume that $M_{i-1, l} \subset W_{i,l}$ for $l = 1, \dots, n_{i-1}$. 
   Furthermore, we can suppose that the component $M_{i, n_i}$ contains the critical point 
   $a \in \mathrm{Crit}(\rho)$ and is diffeomorphic to a closed disk.
   By making it sufficiently small if necessary, we can choose a chart $(U, z = u + \I \,v)$ 
   containing $M_{i, n_i}$. We set $\hat{M}_{i - 1} \coloneqq M_{i -1} \cup M_{i, n_i}$. 
   Let  
   $\hat{f}_{i-1} : \hat{M}_{i - 1} \to \L$ and 
   $\hat{\phi}_{i - 1} : \hat{M}_{i - 1} \to \B$ be the maps given by
   \begin{align}
     \hat{f}_{i-1}(p) =
     \begin{cases}
       f_{i-1}(p) & (p \in M_{i - 1}), \\
       (0, u(p), v(p))  &  (p \in M_{i, n_i}),
     \end{cases} 
     \qquad
     \hat{\phi}_{i-1} \coloneqq \frac{2 \, \partial \hat{f}_{i-1}}{\theta}.
   \end{align}
   There is a one-to-one correspondence between the connected components of $\hat{M}_{i-1}$ and
   those of $W_{i}$, and a homology basis of each connected component of $\hat{M}_{i-1}$ is also
   a homology basis for the connected component of $W_{i}$ containing it. 
   Therefore, by applying \cref{prop:weak.approx.} on each connected component, 
   we obtain $f_i \in \MF(W_i)$ and $\phi_i \in \mathcal{O}(W_i, \B)$ satisfying $(b_i)$--$(e_i)$.
   We deal with condition $(e_i)$ by choosing a finite set and 
   a positive integer in the same way as in Cases 1 and 2.

   \emph{Subcase 3b:} 
   There exists an arc $E \subset \Int(M_i) \setminus \Int(M_{i - 1})$ such that
   \begin{enumerate}
     \item[(1)]
     it intersects $\partial M_{i - 1}$ transversely only at its endpoints $q_0$ and $q_1$, and
     
     \item[(2)]
     $\hat{M}_{i-1} \coloneqq M_{i - 1} \cup E$ satisfies $\hml{\hat{M}_{i - 1}} \emb \hml{W_i}$.
   \end{enumerate}
   In this case, $q_0$ and $q_1$ are contained in the same connected component of $M_{i-1}$. 
   Since $n_{i - 1} = n_i$, by reindexing if necessary, we may assume that 
   $q_0, q_1 \in M_{i-1, n_{i-1}} \subset M_{i , n_i}$.
   Take an arc $E' \subset \Int(M_{i-1, n_{i - 1}}) \cup \{q_0, q_1\}$ connecting $q_0$ and $q_1$, 
   and set $C_0 \coloneqq E \cup E'$. By adding $C_0$ to the homology basis of $M_{i-1}$, 
   we obtain a homology basis $\hat{\mathcal{C}}$ for $\hat{M}_{i - 1}$, $M_i$, and $W_i$. 
   Let $\gamma : [0, 1] \to E$ and $\Gamma : [0,1] \to E'$ be parameterizations of $E$ and $E'$,
   respectively, satisfying 
   $\gamma(0) = q_0, \ \gamma(1) = q_1, \ \Gamma(0) = q_1$, and $\Gamma(1) = q_0$. 
   By \cref{fact:connecting.path.in.B}, we can choose a map 
   $\hat{\phi} : [0,1] \to \B$ satisfying the following:
   \begin{equation}
      \int_E \hat{\phi} \, \theta 
      = \I \, \mathfrak{p}([C_0]) - \int_{E'} \phi_{i - 1} \, \theta.
   \end{equation}
   Let $\hat{\phi}_{i-1} : \hat{M}_{i -1} \to \B$ be a map defined by
   \begin{equation} \label{eq:def.of.phi.hat}
     \hat{\phi}_{i-1}(p) \coloneqq
     \begin{cases}
       \phi_{i - 1}(p) & (p \in M_{i -1}), \\
       \hat{\phi}(\gamma^{-1}(p)) & (p \in E).
     \end{cases}
   \end{equation}
   Then, for each homology basis $C \in \hat{\mathcal{C}}$, it holds that
   \begin{equation}
    \int_C\hat{\phi}_{i - 1} \, \theta =
    \begin{dcases}
      \int_{C}\phi_{i-1} \, \theta = \I \, \mathfrak{p}([C]) 
      & (C \ne C_0) \\
      \int_{E} \hat{\phi} \, \theta + \int_{E'} \phi_{i -1} \, \theta
      = \I \, \mathfrak{p}([C_0]) 
      & (C = C_0).
    \end{dcases}
   \end{equation}
   Let $\hat{f}_{i-1} : \hat{M}_{i-1} \to \L$ be a map defined 
   on each connected component of $\hat{M}_{i-1}$ as
   \begin{equation} \label{eq:def.of.f.hat}
     \hat{f}_{i-1}(p) = f_{i - 1}(p_l) + \Re\int_{p_{l}}^p \hat{\phi}_{i - 1}\, \theta,
   \end{equation}
   where $p_{l} \in M_{i - 1, l}$ are fixed points $(l = 1, \dots n_{i - 1})$.
   Then, \cref{prop:weak.approx.} ensures the existence of 
   $f_i \in \MF(W_i)$ and $\phi_i \in \mathcal{O}(W_i,\B)$ satisfying 
   conditions $(b_i)$--$(e_i)$. Condition $(d_i)$ follows from 
   $\Flux_{f_i} = \Flux_{\hat{f}_{i-1}}^{\hat{\mathcal{C}}}$, and
   we deal with condition $(e_i)$ by choosing a finite set and
   a positive integer in the same way as in Cases 1 and 2.

   \emph{Subcase 3c:}
   In this case, $q_0$ and $q_1$ are contained in different connected components of $M_{i-1}$.
   Let $E \subset \Int(M_i) \setminus \Int(M_{i - 1})$ be an arc satisfying the conditions 
   (1) and (2) in Subcase 3b.
   Since $n_{i - 1} = n_i + 1$, by reindexing if necessary, we may assume that 
   $q_0 \in M_{i-1, n_{i-1}-1}$, $q_1 \in M_{i-1, n_{i-1}}$, and 
   $M_{i-1, n_{i-1}-1} \cup M_{i-1, n_{i-1}}\subset M_{i , n_i}$.
   Let $\gamma : [0,1] \to E$ be a parameterization of $E$ with 
   $\gamma(0) = q_0$ and $\gamma(1) =q_1$,
   and let $\hat{\phi} : [0,1] \to \B$ be a map obtained by 
   \cref{fact:connecting.path.in.B} satisfying
   \begin{equation}
     \int_{E} \hat{\phi} \, \theta = f_{i-1}(q_1) - f_{i-1}(q_0).
   \end{equation}
   We define $\hat{\phi}_{i-1} : \hat{M}_{i - 1} \to \B$ by \eqref{eq:def.of.phi.hat}, and 
   define $\hat{f}_{i - 1} : \hat{M}_{i - 1} \to \L$ on each connected component of $\hat{M}_{i-1}$
   as \eqref{eq:def.of.f.hat}.
   Proceeding as in Subcase 3b, \cref{prop:weak.approx.} ensures the existence of 
   $f_i \in \MF(W_i)$ and $\phi_i \in \mathcal{O}(W_i,\B)$ satisfying conditions $(b_i)$--$(e_i)$.
   
   Now, we obtain $f_i \in \MF(W_i)$, $\phi_i \in \mathcal{O}(W_i, \B)$, and $\epsilon_i > 0$ 
   ($i = 1, 2, \dots$) satisfying conditions $(a_i)$--$(d_i)$. 
   Fix an arbitrary $p \in M$. Then, there is an integer $i_0 > 0$ such that for all $i \ge i_0$, 
   $p$ is contained in $M_i$. The sequences $\{f_i(p)\}_{i \ge i_0}$ and $\{\phi_i(p)\}_{i \ge i_0}$
   satisfy the following for every $i_0 < i < j$.
   \begin{align}
     |f_i(p) - f_j(p)| 
     &\le \sum_{k = i + 1}^j\|f_{k}-f_{k-1}\|_{M_{k-1}}
     \le \sum_{k = i + 1}^j \epsilon_k
     < \sum_{k = i + 1}^j \frac{\epsilon}{2^{k + 1}},\\
     |\phi_i(p) - \phi_j(p)| 
     &\le \sum_{k = i + 1}^j\|\phi_{k}-\phi_{k-1}\|_{M_{k-1}}
     \le \sum_{k = i + 1}^j \epsilon_k
     < \sum_{k = i + 1}^j \frac{\epsilon}{2^{k + 1}}.
   \end{align}
   Hence, $\{f_i(p)\}_{i \ge i_0}$ and $\{\phi_i(p)\}_{i \ge i_0}$ are Cauchy sequences.
   By denoting the limits of these sequences by $\tilde{f}(p)$ and $\tilde{\phi}(p)$ respectively, 
   we define $\tilde{f} : M \to \L$ and $\tilde{\phi} : M \to \C^3$.
   Let us show that $\tilde{f}$ is the desired maxface.
   
   First, let us show that $\{\phi_i\}_i$ uniformly converges to $\tilde{\phi}$ on compact sets. 
   For any compact set $K \subset M$, there exists a positive integer $i_1$ such that
   $K \subset M_i$ for all $i \ge i_1$. 
   Take an arbitrary $r > 0$ and $p \in K$. Then, there exists a positive integer $i_2$ 
   such that $\|\phi_i(p) - \phi_j(p)\| < r$ holds for $j > i \ge i_2$. 
   Letting $j \to \infty$, we have $\|\phi_i(p) - \tilde{\phi}(p)\| \le r$. 
   Since $r > 0$ is independent of $p$, we obtain $\|\phi_i - \tilde{\phi}\|_K \le r$. 
   This shows that $\{\phi_i\}_i$ converges locally uniformly to $\tilde{\phi}$, 
   and therefore $\tilde{\phi} \in \mathcal{O}(M, \C^3)$.
   
   By condition $(a_i)$, we have 
   $\|\phi_i - \phi_{i - 1}\|_{M_{i - 1}} 
   < \epsilon_i < (1 - \alpha_i) \min_{M_{i - 1}}\|\phi_{i - 1}\|$. 
   Therefore, for all $p \in M_{i - 1}$, we have
   \begin{equation}
   \|\phi_i(p)\| - \|\phi_{i - 1}(p)\|
   > -\epsilon_i
   > -(1 - \alpha_i) \, \min_{M_{i - 1}}\|\phi_{i - 1}\|
   \ge -(1 - \alpha_i) \, \|\phi_{i-1}(p)\|.
   \end{equation}
   That is, $\|\phi_i\| > \alpha_i \|\phi_{i - 1}\|$ holds on $M_{i - 1}$. 
   Taking an arbitrary $p \in M$, there exists a positive integer $i_3$ such that 
   $p \in M_{i-1}$ for all $i \ge i_3$. 
   Thus, for all $i > i_3$, we have
   \begin{align}
   \|\phi_i(p)\|
    &> \alpha_i \, \|\phi_{i - 1}(p)\| 
    > \alpha_i \, \alpha_{i - 1} \, \|\phi_{i - 2}(p)\| \\ 
    &> \cdots 
    > \alpha_i \cdots \alpha_{i_3+1} \, \|\phi_{i_3}(p)\| 
    \ge \left(\prod_{j = 0}^{\infty} \alpha_j\right) \|\phi_{i_3}(p)\|.
   \end{align}
   Letting $i \to \infty$, we obtain $\|\tilde{\phi}(p)\| \ge (1/2) \, \|\phi_{i_3}(p)\| > 0$.
   Furthermore, we know that $\tilde{\phi} \in \mathcal{O}(M, \B)$ 
   since $-(\phi_i^0(p))^2 + (\phi_i^1(p))^2 + (\phi_i^2(p))^2 = 0$ holds for all $p \in M$ and 
   $\phi_i = (\phi_i^0, \phi_i^1, \phi_i^2)$.
   
   Take an arbitrary $[C] \in \hml{M}$. Since $C \subset M$ is compact, 
   $\{\phi_i\}_i$ converges uniformly to $\tilde{\phi}$ on $C$. Therefore,
   \begin{equation}
     \lim_{i \rightarrow \infty}\int_C \phi_i \, \theta = \int_C\tilde{\phi} \, \theta. 
   \end{equation}
   On the other hand, since $C \subset M_i$ for any sufficiently large $i$, condition ($d_i$) implies
   \begin{equation}
     \int_{C} \phi_i \, \theta = \I \, \mathfrak{p}([C]).
   \end{equation}
   From the above, we obtain
   \begin{equation} \label{eq:well-defness.and.flux}
     \Re\int_{C} \tilde{\phi} \, \theta = 0 \quad \text{and} \quad
     \Im\int_{C} \tilde{\phi} \, \theta = \mathfrak{p}([C]).
   \end{equation}
   Condition $(b_i)$ ensures that the locally uniform limit $\tilde{\phi}$ of $\{\phi_i\}$ is 
   a full map. 
   By \cref{lem:full.then.not.identically.zero}, 
   the map $\bar{f} : M \to \L$ defined below is a maxface:
   \begin{align}
     \bar{f}(p) \coloneqq \tilde{f}(p_0) + \Re \int_{p_0}^p \tilde{\phi} \, \theta,
   \end{align}
   where $p_0 \in M$ is a fixed point.
   Let us show that $\tilde{f} = \bar{f}$. 
   Choose a positive integer $i$ large enough so that $M_i$ is connected, 
   and take a curve $\gamma_p$ in $M_i$ connecting an arbitrary point $p \in M_i$ and $p_0$.
   Noting that $\phi_i$ converges uniformly on $\gamma_p$ since it is compact, we obtain
   \begin{align}
     \tilde{f}(p) 
     = \lim_{i \rightarrow \infty} f_i(p)
     = \lim_{i \rightarrow \infty} 
       \left( f_i(p_{0}) + \Re \int_{\gamma_p} \phi_{i} \, \theta \right)
     = \tilde{f}(p_0) + \Re \int_{\gamma_p} \tilde{\phi} \, \theta
     = \bar{f}(p)
   \end{align}
   Hence, $\tilde{f}$ is a maxface. 
   The following calculation shows that it satisfies assertion (i):
   \begin{align}
     \|\tilde{f} - f\|_S
     &\le \|f - f_0\|_S + \sum_{k = 1}^{i} \|f_{k} - f_{k - 1}\|_{M_{k-1}} + \|f_{i} - \tilde{f}\|_S \\
     &< \frac{\epsilon}{2} + \sum_{k = 1}^i \frac{\epsilon}{2^{k + 1}} + \|f_{i} - \tilde{f}\|_S 
     < \epsilon + \|f_{i} - \tilde{f}\|_S \rightarrow \epsilon \qquad (i \rightarrow \infty).
   \end{align}

    Let us verify that condition (ii) holds. 
    First, for each $p \in \Lambda \cup \Sigma$ and any sufficiently large $i$, 
    condition $(e_i)$ implies $f_i(p) = f(p)$. Letting $i \to \infty$, we obtain $\tilde{f}(p) = f(p)$.
    Moreover, since $\{\phi_i\}_i$ converges to $\tilde{\phi}$ uniformly on compact sets, 
    its derivatives $\{\phi_i^{(l)}\}_i$ also converge locally uniformly to $\tilde{\phi}^{(l)}$.
    Furthermore, condition $(e_i)$ ensures that for all $p \in \Lambda \cup \Sigma$ and 
    sufficiently large $i$, we have 
    $\phi_{i}^{(l)}(p) = \phi^{(l)}(p)$ ($l = 1, \dots, \hat{k}(p)$). 
    Letting $i \to \infty$, we obtain
    \begin{equation} \label{eq:order.of.phi}
      \tilde{\phi}^{(l)}(p) = \phi^{(l)}(p) \qquad (l = 1, \dots, \hat{k}(p)).
    \end{equation}
    This, together with the same argument as in \cref{prop:weak.approx.}, yields (ii). 
    Furthermore, (iii) is an immediate consequence of equation \eqref{eq:well-defness.and.flux}, 
    while (iv) follows from \eqref{eq:order.of.phi} together with the same argument as in 
    \cref{prop:weak.approx.}.
  \end{proof}
  
  By setting $\mathfrak{p} \equiv 0$ in \cref{thm:main}, 
  we obtain approximation and interpolation for Lorentzian null immersions.
  
  \begin{corollary}
    Assume that $M$ is an open Riemann surface, 
    $\theta$ is a nonvanishing holomorphic $1$-form on $M$, 
    $S \subset M$ is a connected Runge admissible set,
    $\Lambda \subset M$ is a closed discrete subset, 
    $V \subset M$ is an open neighborhood of $\Lambda$,
    $F : S \cup V \to \C^3$ is a map such that 
    $(F|_S, \phi|_S \, \theta)$ is a generalized Lorentzian null immersion and 
    $F|_V$ is a Lorentzian null immersion, 
    where $\phi = \partial F / \theta $.

    Given a positive number $\epsilon > 0$ and
    a map $k \colon \Lambda \to \Z_{>0}$, 
    there exists a Lorentzian null immersion 
    $\tilde{F} \colon M \to \C^3$ satisfying the following conditions.
    \begin{enumerate}
      \item
      $\|\tilde{F} - F \|_S \le \epsilon$.
     
      \item
      The difference $\tilde{F} - F$ vanishes to order $k(p)$ at every point $p \in \Lambda$.
    \end{enumerate}
  \end{corollary}

  \begin{corollary} \label{cor:main.}
    Let $M$ be an open Riemann surface, 
    let $\Lambda \subset M$ be a closed discrete subset, and 
    let $\alpha \colon \Lambda \to \L$ be a map. 
    Then there exists a maxface $\tilde{f} \colon M \to \L$ such that 
    $\tilde{f}|_{\Lambda} = \alpha$. 
    Furthermore, given a map 
    $\alpha = (\alpha_1, \alpha_2) \colon \Lambda \to \L \times \{0, 1, 2, 3, 4, 5\}$, 
    there exists a maxface $\tilde{f} \colon M \to \L$ 
    satisfying $\tilde{f}|_{\Lambda} = \alpha_1$ and 
    the following conditions for each $p \in \Lambda$:
    \begin{itemize}
      \item 
      if $\alpha_2(p) = 0$, then $p$ is a regular point of $\tilde{f}$,

      \item 
      if $\alpha_2(p) = 1$, 
      then $\tilde{f}$ is $\mathcal{A}$-equivalent to a cuspidal edge at $p$,

      \item 
      if $\alpha_2(p) = 2$, 
      then $\tilde{f}$ is $\mathcal{A}$-equivalent to a swallowtail at $p$,

      \item 
      if $\alpha_2(p) = 3$,
      then $\tilde{f}$ is $\mathcal{A}$-equivalent to a cuspidal cross cap at $p$,

      \item 
      if $\alpha_2(p) = 4$,
      then $\tilde{f}$ is $\mathcal{A}$-equivalent to a cuspidal butterfly at $p$, and 

      \item 
      if $\alpha_2(p) = 5$,
      then $\tilde{f}$ is $\mathcal{A}$-equivalent to a cuspidal $S_1^-$ singularity at $p$.
    \end{itemize} 
  \end{corollary}

  \begin{proof}
    It suffices to prove the latter part of the claim, since it implies the former.
    We choose a chart $(U_p, z_p)$ around each point $p \in \Lambda$ such that
    \begin{itemize}
      \item 
      $U_p \cap U_q = \varnothing$ if $p \neq q$,

      \item
      $z_p(U_p)$ is the disk of radius $1$ centered at $z_p(p)$,

      \item 
      $z_p(p) = 0$ if $\alpha_2(p)$ = 0,

      \item 
      $z_p(p) = e^{\I \, \pi/3}$ if $\alpha_2(p) = 1$,

      \item 
      $z_p(p) = 1$ if $\alpha_2(p) = 2$,

      \item 
      $z_p(p) = e^{\I \, \pi / 4}$ if $\alpha_2(p) = 3$, and

      \item 
      $z_p(p) = \log(1/\sqrt{2}) + \I \, \pi /2$ if $\alpha_2(p) \in \{4, 5\}$.
    \end{itemize}
    
    Let $f_0 : \C \to \L$ be the Lorentzian Enneper surface as in \cref{ex:L.Enneper} and
    let $f_1, \ f_2 : \R \times (-\pi, \pi) \to \L$ be the maxfaces as in \cref{ex:cB.cS}.
    For every $q \in U_p$, we set 
    \begin{equation}
      f(q) \coloneqq
      \begin{cases}
        (0, \Re(z_p(q)), \Im(z_p(q))) + \alpha_1(p) & (\alpha_2(p) = 0), \\
        f_0(z_p(q)) - f_0(z_p(p)) + \alpha_1(p)  & (\alpha_2(p) = 1, 2, 3), \\
        f_1(z_p(q)) - f_1(z_p(p)) + \alpha_1(p) & (\alpha_2(p) = 4), \\
        f_2(z_p(q)) - f_2(z_p(p)) + \alpha_1(p) & (\alpha_2(p) = 5).
      \end{cases}
    \end{equation}
    Noting that the map 
    $f \colon \bigcup_{p \in \Lambda} U_p \to \L$ 
    satisfies $f|_{\Lambda} = \alpha_1$, we obtain
    the desired maxface $\tilde{f}$ by applying \cref{thm:main} to $f$.
  \end{proof}

  \begin{corollary} \label{cor:dense.im.sing}
    Let $M$ be an open Riemann surface. 
    Then there exists a maxface $\tilde{f} \colon M \to \L$ such that 
    the image of its singular set is dense in $\L$.
  \end{corollary}
    
  \begin{proof}
    Let $\mathbb{Q}^3 = \{q_i\}_{i = 1}^{\infty} \subset \L$. 
    We take a closed discrete subset $\Lambda = \{p_i\}_{i = 1}^{\infty}$ of $M$.
    By applying \cref{cor:main.} to $\Lambda$ and 
    the map $\alpha : \Lambda \ni p_i \mapsto (q_i, 1) \in \L \times \{0, 1, 2, 3, 4, 5\}$, 
    we obtain a maxface $\tilde{f} : M \to \L$ whose singular set contains $\Lambda$.
    Furthermore, $\tilde{f}$ satisfies $\tilde{f}(\Lambda) = \mathbb{Q}^3$, and thus
    it is the desired maxface.
  \end{proof}

  \begin{remark}
    The maxface obtained in the proof of \cref{cor:dense.im.sing} is such that
    the image of its cuspidal edges is dense in $\L$. 
    On the other hand, by changing the choice of $\alpha$, 
    we can construct examples in 
    which the images of swallowtails and other singularities are also dense.
  \end{remark}

%\bibliographystyle{amsalpha}
%\bibliography{bibliography}

\begin{thebibliography}{9}
     \bibitem{AFL21}
     A. Alarc\'on, F. Forstneri\v c, and F. J. L\'opez,
     \emph{Minimal Surfaces from a Complex Analytic Viewpoint},
     Springer Monogr. Math., Springer, Cham, 2021.

     \bibitem{AH.hol.null.in.SL2}
     A. Alarc\'on and J. Hidalgo,
     \emph{Holomorphic null curves in the special linear group},
     Trans. Amer. Math. Soc. \textbf{379} (2026), no. 6, 4089--4119.

     \bibitem{BS.49.Runge.gene}
     H. Behnke and K. Stein,
     \emph{Entwicklung analytischer Funktionen auf Riemannschen Fl\"achen},
     Math. Ann. \textbf{120} (1949), 430--461.

     \bibitem{Mergelyan.thm}
     E. Bishop,
     \emph{Subalgebras of functions on a Riemann surface},
     Pacific J. Math. \textbf{8} (1958), 29--50.

     \bibitem{92.ER.generalized.maximal}
     F. J. M. Estudillo and A. Romero,
     \emph{Generalized maximal surfaces in Lorentz--Minkowski space $\L$},
     Math. Proc. Camb. Phil. Soc. \textbf{111} (1992), 515--524.

     \bibitem{Weierstrass.thm}
     H. Florack,
     \emph{Regul\"are und meromorphe Funktionen auf nicht geschlossenen Riemannschen Fl\"achen},
     Schr. Math. Inst. Univ. M\"unster (1948), no. 1, 34 pp.

     \bibitem{2020.legacy.hol.approx.}
     J. E. Forn\ae ss, F. Forstneri\v c, and E. F. Wold,
     \emph{Holomorphic approximation: the legacy of Weierstrass, Runge, Oka--Weil, and Mergelyan},
     in \emph{Advancements in Complex Analysis: From Theory to Practice},
     Springer, Cham, 2020, 133--192.

     \bibitem{F17}
     F. Forstneri\v c,
     \emph{Stein Manifolds and Holomorphic Mappings},
     2nd ed., Ergebnisse der Mathematik und ihrer Grenzgebiete. 3. Folge.
     A Series of Modern Surveys in Mathematics, vol. 56, Springer, Cham, 2017.

     \bibitem{F19.Mergelyan.for.mfd-valued.map}
     F. Forstneri\v c,
     \emph{Mergelyan's and Arakelian's theorems for manifold-valued maps},
     Mosc. Math. J. \textbf{19} (2019), no. 3, 465--484.

     \bibitem{FKKRUY10}
     S. Fujimori, Y. Kawakami, M. Kokubu, W. Rossman, M. Umehara, and K. Yamada,
     \emph{Hyperbolic metrics on Riemann surfaces and space-like CMC-1 surfaces in de Sitter 3-space},
     in \emph{Recent Trends in Lorentzian Geometry},
     Springer Proc. Math. Stat., vol. 26, Springer, New York, 2013, 1--47.

     \bibitem{FSUY08}
     S. Fujimori, K. Saji, M. Umehara, and K. Yamada,
     \emph{Singularities of maximal surfaces},
     Math. Z. \textbf{259} (2008), 827--848.

     \bibitem{GN67.nonvani.hol.1-form}
     R. C. Gunning and R. Narasimhan,
     \emph{Immersion of open Riemann surfaces},
     Math. Ann. \textbf{174} (1967), 103--108.

     \bibitem{83.Kobayashi}
     O. Kobayashi,
     \emph{Maximal surfaces in the $3$-dimensional Minkowski space $L^3$},
     Tokyo J. Math. \textbf{6} (1983), 297--309.

     \bibitem{51.Merglyan.original}
     S. N. Mergelyan,
     \emph{On the representation of functions by series of polynomials on closed sets},
     Dokl. Akad. Nauk SSSR (N.S.) \textbf{78} (1951), 405--408.

     \bibitem{OT18.duality.cb.cS}
     Y. Ogata and K. Teramoto,
     \emph{Duality between cuspidal butterflies and cuspidal $S_1^-$ singularities 
           on maximal surfaces},
     Note Mat. \textbf{38} (2018), 115--130.

     \bibitem{1885.Runge}
     C. Runge,
     \emph{Zur Theorie der eindeutigen analytischen Functionen},
     Acta Math. \textbf{6} (1885), no. 1, 229--244.

     \bibitem{UY06}
     M. Umehara and K. Yamada,
     \emph{Maximal surfaces with singularities in Minkowski space},
     Hokkaido Math. J. \textbf{35} (2006), 13--40.
  \end{thebibliography}

\end{document}